\documentclass[10pt,reqno]{amsart}
\usepackage{amsmath}
\usepackage{amssymb}
\usepackage{amsthm}
\usepackage{eepic,epic}
\usepackage{color}
\usepackage{epsfig}
\usepackage{graphicx}
\usepackage{hyperref}
\usepackage[margin=1.0in]{geometry}
\usepackage{array}

\newlength{\defbaselineskip}
\newcommand{\setlinespacing}[1]%
           {\setlength{\baselineskip}{#1 \defbaselineskip}}

\numberwithin{equation}{section}

\numberwithin{equation}{section}

\newtheorem{theorem}{Theorem}[section]
\newtheorem{lemma}{Lemma}[section]

\newtheorem{proposition}{Proposition}[section]
\newtheorem{remark}{Remark}[section]

\newcommand{\R}{\mathbb R}

\newcommand{\bq}{\begin{equation}}
	\newcommand{\eq}{\end{equation}}

\newcommand{\lt}{\left}
\newcommand{\rt}{\right}
\newcommand{\lal}{\langle}
\newcommand{\ral}{\rangle}
\newcommand{\pa}{\partial}

\makeatletter
\def\moverlay{\mathpalette\mov@rlay}
\def\mov@rlay#1#2{\leavevmode\vtop{%
		\baselineskip\z@skip \lineskiplimit-\maxdimen
		\ialign{\hfil$\m@th#1##$\hfil\cr#2\crcr}}}
\newcommand{\charfusion}[3][\mathord]{
	#1{\ifx#1\mathop\vphantom{#2}\fi
		\mathpalette\mov@rlay{#2\cr#3}
	}
	\ifx#1\mathop\expandafter\displaylimits\fi}
\makeatother

\begin{document}

\allowdisplaybreaks

\title[Global well-posedness for the 3D irrotational Euler-Riesz system]{
	Global smooth solutions to the irrotational Euler-Riesz system in three dimensions}

\author[Choi]{Young-Pil Choi}
\address[Young-Pil Choi]{\newline Department of Mathematics \newline
	Yonsei University, 50 Yonsei-Ro, Seodaemun-Gu, Seoul 03722, Republic of Korea}
\email{ypchoi@yonsei.ac.kr}

\author[Jung]{Jinwook Jung}
\address[Jinwook Jung]{\newline 
	Department of Mathematics and Research Institute for Natural Sciences\newline
	Hanyang University, 222 Wangsimni-ro, Seongdong-gu, Seoul 04763, Republic of Korea}
\email{jinwookjung@hanyang.ac.kr}

\author[Lee]{Yoonjung Lee}
\address[Yoonjung Lee]{\newline Department of Mathematics \newline
	Yonsei University, 50 Yonsei-Ro, Seodaemun-Gu, Seoul 03722, Republic of Korea}
\email{yjglee@yonsei.ac.kr}


\begin{abstract} 
This paper investigates the global dynamics of the Euler--Riesz system in three dimensions, focusing on the well-posedness and large-time behavior of solutions near equilibrium. The system generalizes classical interactions by incorporating the Riesz interactions $\nabla (-\Delta)^{-\sigma/2}(\rho - 1)$. We show that the system admits a global smooth solution for small irrotational initial perturbations. Specifically, we establish that if the initial data is sufficiently small, the solution remains regular globally in time and decays over time at a rate dependent on $\sigma$.
\end{abstract}

\maketitle 

\tableofcontents

%
%
%
%
%
%
%
%
%

\section{Introduction}
This paper investigates the Euler--Riesz system in three dimensions, described by the equations:
\begin{align}\label{ER_0}
	\begin{aligned}
		&\pa_t \rho + \nabla \cdot (\rho u) = 0, \quad (x,t)\in \R^3 \times \R_+,  \cr
		&\pa_t (\rho u) + \nabla \cdot (\rho u \otimes u) +\nabla(c_{\gamma}\rho^{\gamma}) = - \rho\frac{ \nabla }{|\nabla|^{\sigma}} (\rho-1)
	\end{aligned}
\end{align}
with initial data
\[
(\rho(x, 0), u(x, 0))=(\rho_0(x), u_0(x)).
\]
Here, $\rho=\rho(x,t)$ and $u= u(x,t)$ represent the density and velocity of the fluid, respectively, at time $t\in \mathbb{R}_{+}$ and position $x\in \mathbb{R}^3$.
The pressure term is expressed as $c_{\gamma}\rho^{\gamma}$ with a positive constant $c_{\gamma}$ and $\gamma >1$. The background state is assumed to be steady and uniform, and for simplicity, it is taken to be unity.  The operator $|\nabla|^{\sigma}$ is defined by $(-\Delta)^{\sigma/2}$, and this paper focuses on the range $0< \sigma <2$. 

The Euler--Riesz system \eqref{ER_0} encompasses not only the classical Coulomb interaction with $\sigma=2$ but also the so-called {\it pure Manev} correction \cite{BDIV97} with $\sigma=1$ proposed as a modification of Newtonian gravitation. The Riesz interaction serves as a generalization and has been extensively studied in the physics literature, e.g., \cite{BBDR05, Maz11}. It also appears in the study of equilibrium properties of systems of point particles interacting via Coulomb or Riesz interactions and confined by an external potential \cite{LS17, PR18, RS16}. 


\subsection{Main result}
The objective of the current work is to investigate the global dynamics of the Euler--Riesz system \eqref{ER_0} for $0<\sigma<2$. Specifically, we aim to address the global-in-time existence and uniqueness of regular solutions to the system \eqref{ER_0} near the equilibrium $(\rho, u)=(1, 0)$ as well as its large-time behaviors. The local-in-time existence and uniqueness of regular solutions to \eqref{ER_0} have been established in \cite{CJ22}, which also provided sufficient conditions under which finite-time singularity formation can occur. 

A significant challenge in proving the global existence of solutions to the system \eqref{ER_0} arises from the absence of a direct dissipation term in the system. This lack of dissipation complicates the analysis since it can lead to the finite-time breakdown of smoothness of solutions. To address this, some approaches have incorporated linear velocity damping. For instance, the global existence and large-time behavior of solutions to the pressureless system with linear damping were investigated in \cite{CJ23} for small initial perturbations. Additionally, recent work has explored global existence in Besov spaces with critical regularity \cite{CSX}.

In a related context, Danchin and Ducomet \cite{DD22} demonstrated that for $1\leq \sigma<2$, the system \eqref{ER_0} admits a unique global solution with zero background state under a sufficiently small, regular initial density and a dispersive spectral condition on the initial velocity that causes fluid particles to spread. The complete well-posedness of \eqref{ER_0} for $0<\sigma<1$, including both pressureless and pressure cases, has been recently resolved in \cite{CJL}. Further insights into the global dynamics have been obtained from Guo's work \cite{G} for $\sigma=2$. The Poisson interaction force introduces oscillatory behavior that results in a linearized system resembling the Klein--Gordon equation with the  dispersion relation
\[
p(r)=(1+r^2)^{1/2}.
\]
Utilizing Shatah's normal transformation and the Klein--Gordon effect, a global irrotational smooth solution to the 3D Euler--Poisson equations is constructed under small smooth perturbations. This result is complemented by similar studies in \cite{IP13, LW14} for the two-dimensional case and \cite{GHZ17} for the one-dimensional case. Guo's work highlights that dispersion due to the Poisson interaction force can mitigate the finite-time blow-up phenomenon. We also refer to \cite{DIP 17, GIP14, GIP16, GN14, GMP13, GP, HZG12} for the related Euler models. 

From this perspective, the Euler system without interaction force presents a completely different scenario in contrast to the Euler--Poisson system. Following Makino's symmetrization \cite{M86}, the pure Euler system is understood as quadratically nonlinear wave equations with a dispersion relation $p(r)=r$.  For these nonlinear wave equations, the lifespan $T_{\varepsilon}$ of classical solutions is known to be of the order of $ \exp(C/\varepsilon)$ for small initial data of size $\varepsilon$. Notably, Sideris \cite{S85, S91} demonstrated that the lifespan of classical solutions satisfies $T_{\varepsilon}<\exp(C/\varepsilon^2)$ and $T_{\varepsilon}>\exp(C/\varepsilon)$ for small smooth perturbation $(\rho_{\varepsilon}, u_{\varepsilon})=(\varepsilon n_0 +1, \varepsilon u_0)$, where $(n_0, u_0)$ is compactly supported around the constant state $(\rho, u)=(1, 0)$. The sharp asymptotic behavior of $T_{\varepsilon} \sim \exp(C/\varepsilon)$ for spherically symmetric Eulerian flows was obtained in \cite{G05}. These results indicate that singularity formation is inevitable in the pure Euler system, even with sufficiently small initial perturbations.

The central question addressed in the present work is whether a global irrotational solution exists for the 3D Euler--Riesz system \eqref{ER_0}, extending the work of Guo \cite{G05}. Even though the Riesz interaction $0 < \sigma < 2$ is more singular than the Coulombian force ($\sigma=2$), it is expected to play a significant role in ensuring global-in-time stability near the equilibrium $(\rho, u)=(1, 0)$. Throughout this paper, we consider an irrotational initial flow:
\[
\nabla \times u \equiv0
\]
which is preserved in time. For the convenience of presentation, we set $c_{\gamma}=1/3$ and $\gamma=3$ (other cases $\gamma \in (1, 1+2/3]$ can be treated similarly), and then we let $n=\rho-1$ to obtain the symmetrized form of \eqref{ER_0}:
\begin{align}\label{ER_1}
	\begin{aligned}
		&\pa_t n   +\nabla \cdot u + \nabla \cdot (n u)= 0, \quad (x,t)\in \R^3 \times \R_+,  \cr
		&\pa_t u + \nabla n + \nabla (\frac{1}{2}n^2 +\frac{1}{2} |u|^2 ) =  -\frac{\nabla }{|\nabla|^{\sigma}} n
	\end{aligned}
\end{align}
subject to the initial data
\[
(n(x, 0), u(x, 0)) =(n_0(x):=\rho_0(x)-1, u_0(x)), \quad x \in \R^3.
\]

We study the global well-posedness of the reformulated system \eqref{ER_1} with small initial data, demonstrating that \eqref{ER_0} admits a smooth global solution under a small irrotational initial perturbation. To facilitate our analysis, we introduce the following function spaces: for $s\ge 16$
 \[
  \begin{aligned}
 \| f \|_{X} := \||\nabla|^{-\frac{2-\sigma}{2}}  f\|_{H^{2s}}  +(1+t)^{\beta_{\sigma}} \|f \|_{W^{s, p}} \quad \mbox{and} \quad  \|f\|_{Y} := \||\nabla|^{-\frac{2-\sigma}{2}}f\|_{H^{2s}} 
 + \|f\|_{W^{s+6(\frac{1}{2}-\frac{1}{p}), p'}}
 \end{aligned}
 \]
where $p'$ is the H\"older conjugate of $p$. Here, $ \beta_{\sigma}$ is defined by
 \bq\label{beta}
 \beta_{\sigma}=
 \begin{cases}
 	3(\frac{1}{2}-\frac{1}{p}) \quad \text{when } 0<\sigma\leq 1,\\[1mm]
 	\frac{8}{3}(\frac{1}{2}-\frac{1}{p}) \quad \text{when } 1<\sigma<2,
 \end{cases}
 \eq 
and $p$ is chosen according to
 \bq\label{p}
 \begin{aligned}
 	&6<p<\infty \quad \text{when } 0<\sigma\leq 1,\\
 	&8<p<\infty \quad \text{when } 1<\sigma<2.
 \end{aligned}
 \eq 
Then our main result is stated as follows.
 	\begin{theorem}\label{global}
 	There exists $\varepsilon>0$ such that for any initial perturbation $(n_0+1, u_0)$ satisfying
 	$$
  \|(n_0, u_0)\|_{Y} \leq \varepsilon \quad \text{and} \quad 	\nabla \times u_0 =0, 
 	$$
	the system \eqref{ER_0} admits a global solution $(n+1, u)$ with 
 	\[
 	\sup_{0 \leq t < \infty}\|(n, u)(t)\|_{X} \leq 2\varepsilon.
 	\]
 	Moreover, there exists a constant $C>0$ independent of $t$ such that 
 	\bq \label{result_decay rate}
 	\|(n, u)(t)\|_{W_x^{s, p}} \leq
 	 C(1+t)^{-\beta_{\sigma}}, 
 	\eq
 	where $\beta_{\sigma}$ and $p$ are given as in \eqref{beta} and \eqref{p}, respectively.
 \end{theorem}

\begin{remark}
	We provide some remarks on conditions in Theorem \ref{global}.
	\begin{itemize}
		\item[(i)] The lower bound on $p$ in \eqref{p} is required to ensure that $(1+t)^{-\beta_{\sigma}}$ is integrable in time $t$.
		\item[(ii)] The assumption on initial data does not require a neutrality condition. 	
 \end{itemize}
\end{remark}

\begin{remark}
The decay rate specified in \eqref{result_decay rate} exhibits a marginal loss due to $p<\infty$. The critical case $p=\infty$ remains uncovered due to technical difficulties, e.g. \eqref{est_r}. Moreover, we would like to emphasize that the decay rate is obtained through linear flow analysis, as detailed in Proposition \ref{decay} below. Within this context, we believe this rate is optimal within our strategic framework.
\end{remark} 

\begin{remark}
 For the case $\sigma=2$, Guo \cite{G} showed that a global-in-time existence of smooth irrotational solution to \eqref{ER_0} satisfying 
\[
\|(n, u)(t)\|_{L^{p}} \leq C(1+t)^{-3(1/2-1/p)}
\]
with $6<p<\infty$. In comparison, our theorem indicates that for $1<\sigma<2$, the solution decays at a rate of $(1+t)^{-8/3(1/2-1/p)}$, which is slower than the rate for $\sigma = 2$ due to the dispersion relation of the linearized equation having zero Gaussian curvature at some point. Interestingly, for $0< \sigma\leq 1$, the decay rate recovers the same rate as in the case $\sigma = 2$. 
\end{remark}

%
%
%
%
%
%
%
%
%
 \subsection{Difficulties and strategy of the proof}
The primary challenge in analyzing the Euler--Riesz system arises from the Riesz interaction force, which causes the linearized equation to differ from the Klein--Gordon flow. By differentiating the Euler--Riesz system \eqref{ER_1} with respect to time, we obtain
\bq\label{L_equation} 
\pa_{tt} w -\Delta w +|\nabla|^{2-\sigma} w = F(w),
\eq 
where $w=(n, u)^T$ and $F$ is a vector-valued nonlinear operator. This formulation is due to the curl-free condition.  The dispersive characteristics of this equation are given by
\bq\label{dispersion}
p(r)=(r^2+r^{2-\sigma})^{1/2}.
\eq
It is observed that $p(r)$ is non-degenerate when $\sigma=2$, corresponding to the Klein--Gordon flow with a decay rate of $(1+t)^{-3/2}$. In contrast, when $\sigma=0$, the flow is degenerate, leading to a vanishing Gaussian curvature and a slower decay rate of $(1+t)^{-1}$. However, the case for $0<\sigma<2$ remains largely unexplored, since the geometric structure of $p(r)$ is uniquely determined by $\sigma$. In addition, handling the singularity in the nonlinearity of \eqref{L_equation} is challenging, especially when using the space-time resonance method. In fact, one of the nonlinearities can be viewed as a bilinear operator with a kernel of a fraction form whose denominator is 
\[
\Phi=p(|\xi|)-p(|\xi-\eta|)-p(|\eta|).
\]
The value of $\Phi$ can be zero over a significant region, and there is no null structure to address this singularity. To deal with these challenges, we need to explore the strict asymptotic behaviors of $p(r)$, $p'(r)$, and $p''(r)$ across low frequencies ($r < 1$) and high frequencies ($r \ge 1$), taking into account the value of $\sigma$. Another difficulty is the loss of regularity in the energy estimates. Using standard methods, we encounter terms like $\|\nabla \Lambda^{-\sigma} n\|_{\dot H^{2s}}$, which involves a regularity order of $2s + 1 - \sigma$. It is not feasible to bound this term using the $X$-norm of $n$ for $0 < \sigma < 1$. To overcome this, we need to develop weighted energy estimates that are comparable to classical ones, under a suitable smallness condition on the solution.
%
%
%
%
%
%
%
%
%
\subsubsection*{\bf Strategy of the proof}

We outline the proof of Theorem \ref{global} and describe some of the strategies involved. Consider the function:
 \[
 \alpha(t, x)=\frac{p(|\nabla|)}{|\nabla|} n(t, x) +i \frac{\nabla }{|\nabla|}\cdot u(t, x),
 \]
where $p$ is given as in \eqref{dispersion}. 
 From \eqref{ER_1}, it follows that
 \bq\label{ER_22}
 \begin{aligned}
 	(\pa_t -ip(|\nabla|)) \alpha =  \mathcal{Q}(\alpha),
 \end{aligned}
 \eq
where $\mathcal{Q}(\alpha)$ has quadratic forms:
\[
\begin{aligned}
	\mathcal{Q}(\alpha)
	& = \sum_{r, l=1}^2 \iint_{\R^3 \times \R^3} e^{i x\cdot \xi} m_{r, l}(\xi, \eta) \widehat{\alpha_r}(\xi-\eta) \widehat{\alpha_l} (\eta) \,d\xi d \eta
\end{aligned}
\]
with
\[
\begin{aligned}
	m_{r, l}(\xi, \eta) 
	&= c_1 p(\xi) r(\xi)r(\xi-\eta) r(\eta)+c_2 |\xi| r(\xi-\eta) r(\eta).
\end{aligned}
\]
Here, $p(\xi)=p(|\xi|)$ is defined as \eqref{dispersion}, and the map $r(x)$ is used for $x\mapsto \frac{x}{|x|}$ or $x\mapsto \frac{|x|}{p(x)}$, which obeys
\[
\| r(\nabla) f\|_{L^p} =\| \mathcal{F}^{-1} \{r \hat{f}\}\|_{L^p} \lesssim \|f\|_{L^p} \qquad 1<p<\infty,
\]
and thus it is almost negligible when we estimate it in $L^p$ spaces. By the Duhamel principle, we have
\bq\label{Duhamel}
\alpha(t, x)= e^{itp(|\nabla|)}\alpha(0, x) +  \int_0^t e^{i(t-\tau)p(|\nabla|)}\mathcal{Q}(\alpha)(\tau, x) \,d\tau.
\eq

The main ingredient in establishing energy estimates is $L^p$ decay estimates.
We need to study the decay estimate of the linear solution to \eqref{ER_22} where $\mathcal{Q} \equiv 0$. 
Based on the bootstrapping argument, we deduce a decay rate of the solution to the nonlinear problem \eqref{ER_22} from that of the linear solution in the $L^p$ sense.
For this, the decay rate is required to be integrable over $\mathbb{R}_{+}$, and the nonlinearity $\mathcal{Q}(\alpha)$ could be controlled by terms $\|\alpha\|_X^2$ of at least order $2$. However, due to the lack of regularity for the solution, it is not straightforward to estimate the nonlinear term.\newline

\textbf{Step 1} (Shatah's normal transformation).
 We apply Shatah's normal transformation (performing an integration by parts with respect to $t$ appropriately) to the Duhamel term. Then it can be expressed as
\bq\label{Duhamel2}
 \text{Duhamel term} =  g(t) -e^{itp(|\nabla|)}g(0) +h(t),
\eq
where $g$ has quadratic forms of $\alpha$ (see \eqref{bilinear} below), and $h$ is a sum of cubic forms $h_{r, l}$ for $r, l=1,2$ defined by
\bq\label{H}
\begin{aligned}
	h_{r, l}(t, x)&= -2i \int_{0}^t   \iint_{\R^3\times \R^3}  e^{ix\cdot \xi} e^{i(t-\tau) p(|\nabla|)} \frac{m_{r, l}(\xi, \eta)}{\Phi_{r, l}(\xi, \eta)} \widehat{ \alpha_r}(\tau, \xi-\eta) \widehat{\mathcal{Q}} (\tau, \eta) \, d \eta d\xi  d \tau
\end{aligned}
\eq
with
\[
\Phi_{r, l}(\xi, \eta)=p(\xi)+(-1)^rp(\xi-\eta)+(-1)^lp(\eta).
\]
Thanks to the symmetry, the argument for $(r, l)=(1,1)$ can cover the other cases as well.
It is observed that the phase $\Phi_{1, 1}$ can be zero on the significant set of 
\[
|\eta|^{\frac{2-\sigma}{2}}|\xi-\eta|^{\frac{2-\sigma}{2}}=0
\]
in a stark contrast to the case $\sigma=2$ (see Lemma \ref{Lemma_low_Phi}). This makes estimating the nonlinear term $h$ more challenging in {\bf Step 3}. \newline

\textbf{Step 2} (Dispersive estimates for the linear system). We analyze the decay estimates for the linear flow generated by the operator $e^{it p(|\nabla|)}$. Precisely, our goal is to establish:
\bq\label{Decay_0}
\|e^{it p(|\nabla|)} f\|_{L^{p}} \lesssim (1+t)^{-\beta_{\sigma}} \|f\|_{W^{6(\frac{1}{2}-\frac{1}{p}), p'}}
\eq	
for $2\leq p < \infty$ with $\beta_{\sigma}$ given as in \eqref{beta}. 
As mentioned above, the decay rate of the evolution operator $e^{itp(|\nabla|)}$ heavily relies on the geometric structure of the dispersive feature $p(r)=(r^2+r^{2-\sigma})^{1/2}$, which varies significantly with the parameter $\sigma$. 

When $0 < \sigma \leq 1$, the dispersion relation $p(r)$ is non-degenerate, leading to a decay rate comparable to that of the Klein--Gordon equation. In this regime, the non-degenerate nature of $p(r)$ ensures robust dispersive effects, allowing us to derive relatively strong decay estimates for the linear flow. However, for $1 < \sigma < 2$, the situation becomes more complex due to the presence of a degeneracy point in $p(r)$. This degeneracy results in weak dispersive effects, causing the decay rate to be slower than in the non-degenerate case. Complicating matters further, the dispersion $p(r)$ is a non-homogeneous type, making decay estimates less straightforward since $p(r)$ exhibits different behaviors at low and high frequencies.

To address these challenges, inspired by \cite{GP} and \cite{GPW}, we separate the analysis into low-and high-frequency components. Using the distinct scaling behaviors of $p(r)$ for each component, we employ a different approach on $\sigma$, because $p(r)$ for $\sigma=0$ behaves similarly to the linear wave equation, while the behavior of $p(r)$ for $\sigma=2$ resembles that of the Klein--Gordon equation. In particular, for $1<\sigma<2$, we obtain a decay rate better than that of the wave flow by isolating the degeneracy point and employing the stationary phase method (see Lemma \ref{dispersive_2}). This careful analysis is essential for understanding the dispersive properties of the system and for establishing the decay estimates necessary for the subsequent steps in the proof.\newline 

\textbf{Step 3} (Control of the singularity of the kernel). We establish the following $L^p$ decay estimates: 
\bq\label{LP decay}
\|\alpha(t)\|_{W^{s, p}}\lesssim (1+t)^{-\beta_{\sigma}} \lt(\|\alpha(0)\|_{Y}  +  \sup_{0<t<T^{\ast}}\|\alpha(t)\|_{X}^2 +\sup_{0<t<T^{\ast}}\|\alpha(t)\|_{X}^3 \rt).
\eq
By applying the decay estimate \eqref{Decay_0} to the Duhamel principle expression \eqref{Duhamel}, we focus on estimating the terms in \eqref{Duhamel2}. For clarity, we specifically consider the cubic term $h$. That is,
\[
\|h\|_{W^{s, p}} \lesssim  (1+t)^{-\beta_{\sigma}}\sup_{0<t<T^{\ast}}\|\alpha(t)\|_{X}^3.
\]
To achieve this, we employ some bilinear estimates to get $\|\alpha\|_X^3$. However,
the kernel $\frac{m_{r,l}}{\Phi_{r,l}}$ of the bilinear operator for \eqref{H} is not locally bounded due to the singularity of $\frac{1}{\Phi_{r,l}}$ as mentioned before. 
To overcome this challenge, we introduce the following modified kernel:
\[
\mathfrak{M}_{r, l}=\frac{|\xi|^{\frac{2-\sigma}{2}}|\eta|^{\frac{2-\sigma}{2}}|\xi-\eta|^{\frac{2-\sigma}{2}}}{\lal \xi-\eta\ral^{2 \lambda} \lal \eta\ral^{2 \lambda}\Phi_{r, l}},
\]
where $\lambda>0$ is large enough and $\lal \cdot \ral$ denotes the Japanese bracket defined as $\lal x\ral =(1+x^2)^{1/2}$. 

In this step, it is necessary to employ the negative Sobolev norm $\dot{H}^{-(2-\sigma)/2}$ of the solution to incorporate $\mathfrak{M}_{r, l}$ as the kernel of the bilinear operator. We then employ bilinear estimates developed by Gusafson--Nakanishi--Tsai \cite{GNT07}, where the kernel is shown to belong to the spaces $L_{\xi}^{\infty} \dot{H}_{\eta}^k$ and $L_{\eta}^{\infty} \dot{H}_{\xi}^k$. A significant portion of the analysis is dedicated to proving that $\mathfrak{M}_{r, l} \in L_{\xi}^{\infty} \dot{H}_{\eta}^k \cap L_{\eta}^{\infty} \dot{H}_{\xi}^k$, with the regularity index $k$ required to be greater than $1$ in order to derive the $L^p$ decay estimates. The proof requires a deep understanding of the interplay between the angles and magnitudes of $\xi$, $\xi-\eta$, and $\eta$, based on the asymptotic behavior of the relations $p, p'$ and $p''$. Unfortunately, these asymptotic behaviors are not only affected by the value of $\sigma$ but also by whether $\xi$ lies in the low-frequency region $|\xi|<1$ or the high-frequency region $|\xi|\ge 1$. To address these complexities, we divide the analysis into three cases, performing a detailed examination of each scenario by considering the possible configurations of frequency components in both low- and high-frequency regimes. \newline

\textbf{Step 4} (Time weighted energy estimates). 
In this step, we aim to have the a priori estimates for the solution $\alpha$ in the Sobolev spaces $H^{2s}$.
The required $\dot{H}^{-(2-\sigma)/2}$ estimates for \eqref{LP decay} can be derived since the kernel $m$ in $\mathcal{Q}$ has a gain in $p(\xi)$ or $|\xi|$, which adequately addresses the negative order of the Sobolev index. As a result, we obtain the following bound.
\[
\|\alpha(t)\|_{\dot{H}_x^{-\frac{2-\sigma}{2}}} \lesssim \|\alpha(0)\|_Y +\sup_{0<t<T^{\ast}}\|\alpha(t)\|_{X}^2.
\]
Utilizing the decay estimate \eqref{LP decay}, we can further develop the $H^{2s}$ estimates as follows:
\[
\|\alpha(t)\|_{H^{2s}}
\lesssim \|\alpha(0)\|_{Y} +\sup_{0<t<T^{\ast}}\|\alpha(t)\|_{X}^{3/2}.
\]
This estimate is derived via the application of Gr\"onwall inequality. Specifically, by using the identities
\[
{\rm Re}(\alpha) =\frac{p(|\nabla|)}{|\nabla|}n \quad \text{and} \quad {\rm Im}(\alpha)=\frac{\nabla }{|\nabla|}\cdot u,
\]
we can express the energy functional $\|\alpha\|_{\dot{H}^{2s}}$ for the higher order term as
\bq\label{EF1}
\|(n, u)(t)\|_{\dot{H}^{2s}}^2 +\| n(t)\|_{\dot{H}^{2s-\frac{\sigma}{2}}}^2. 
\eq
In a rather standard way, the following term associated with the Riesz interaction force appears 
\bq\label{Re}
\int_{\R^3}  \nabla |\nabla|^{2s-\sigma}n \cdot \nabla \lt(|\nabla|^{2s-2}\nabla \cdot(nu) \rt) dx,
\eq
where the regularity order for $n$ is approximately $1+2s-\sigma$. This term is manageable when $1 \leq \sigma < 2$, but for $0<\sigma<1$, the Riesz potential introduces a challenging loss of regularity.

To address this challenge, we employ the modified Sobolev space introduced in \cite{CJ23}:
\[
\|f\|_{\dot{\mathcal{H}}^s}:= \lt\| (n+1)^{-\frac{1}{2}}|\nabla|^{s} f\rt\|_{L^2}\qquad s>0,
\]
which satisfies $\|f\|_{\dot{\mathcal{H}}^s} \sim \|f\|_{\dot{H}^s}$ under the condition that $\sup_{0<t<T^{\ast}}\|n(t)\|_{L^{\infty}}<1/2$. It is noteworthy that Theorem \ref{global} does not require any additional assumptions since it inherently ensures the smallness of the global solution. Consequently, we adopt the energy function
\[
\|(n, u)(t)\|_{\dot{H}^{2s}}^2 + \| n(t) \|_{\dot{\mathcal{H}}^{2s-\frac{\sigma}{2}}}^2  
\]
in place of \eqref{EF1}, thereby eliminating the problematic term \eqref{Re}.  With strategic cancellations, we successfully close the higher-order estimates.

%
%
%
%
%
%
%
%
%
 \subsection{Notation and preliminaries}
 Throughout this paper, the symbols $C$ and $c_i$ for any $i\ge 1$ will represent positive universal constants, which may vary from one occurrence to another. We write $A \lesssim B$ to indicate that there exists a constant $C$, independent of time $T$, such that $A \leq CB$, and $A \sim B$ if there exists $C>1$ such that $A/C \leq B \leq CA$.  The phase function $p(r)$ is assumed to be radial, and sometimes we abuse the notation $p(r)=p(|r|)$ for convenience. Additionally, $\hat{f}$ and $\mathcal{F}(f)$ stand for the Fourier transform of $f$, and $\mathcal{F}^{-1}$ represents the inverse Fourier transform. 
 
The Sobolev space $W^{s, p}$ for $1\leq p\leq \infty$ consists of functions $f\in L^{p}$ such that 
\[
\|f\|_{W^{s, p}}:= \sum_{|k|\leq s} \| \pa^k f \|_{L^{p}} <\infty.
\]
In particular for $1<p<\infty$, an equivalent norm is given by
\[
\|f\|_{W^{s, p}}\sim \sum_{\alpha\leq s} \| |\nabla|^\alpha f \|_{L^{p}}
\]
as detailed in \cite[Appendix~A]{T06}.

The following lemma provides a useful tool that will be frequently used in this paper:
\begin{lemma}
	The following inequalities hold;
	\begin{enumerate}
		\item[$(i)$] 
		Let $s>0$, $ p \in (1, \infty)$ 
		If $f\in W^{s, p} $ and $g\in W^{s, p}$,
		then we have
		\bq\label{KP_ineq}
		\| fg\|_{\dot{W}^{s, p}} \lesssim  \| f\|_{\dot{W}^{s, p}}\|g\|_{L^{\infty}} + \|f\|_{L^{\infty}}\| g\|_{\dot{W}^{s, p}},
		\eq
		and moreover
\[
		\| fg\|_{W^{s, p}} \lesssim  \| f\|_{W^{s, p}}\|g\|_{L^{\infty}} + \|f\|_{L^{\infty}}\| g\|_{W^{s, p}}.
\]
		\item[$(ii)$] 
		Let $s>0$. If $f\in \dot{W}^{1, \infty} \cap \dot{H}^s$ and $g\in L^{\infty} \cap \dot{H}^{s-1}$, then 
		\bq\label{tech_1}
		\|[f, |\nabla|^s]g\|_{L^2} \lesssim \|f \|_{\dot{H}^s}\|g\|_{L^\infty} + \|\nabla f\|_{L^\infty}\|g\|_{\dot{H}^{s-1}}.
		\eq
	\end{enumerate}
\end{lemma}
 
We define the Littlewood--Paley operator $P_N$ by
\[
\mathcal{F}(P_ Nf) = \psi(\xi/N)\hat f-\psi(2\xi/N) \hat f=:\varphi_N(\xi) \hat{f}
\]
for any dyadic number $N>0$.
Here, $\psi(x) : \mathbb{R}^3 \rightarrow [0,1]$ is a smooth radial function such that $supp(\psi) \subseteq \{x\,:\, |x|\leq 2\}$ and $\psi(x)=1$ if $|x|\leq 1$. Additionally, we denote $P_{<N}:= \sum_{k<N} P_k$ and $ P_{\ge N}:=1-P_{<N}$. 

The Littlewood--Paley operator $P_N$ satisfies the well-known Bernstein inequality:
\[
\||\nabla|^{\pm s} P_N f\|_{L^p} \lesssim N^{\pm s}\|P_N f\|_{L^p}
\]
for all $s$ and $1\leq p\leq \infty$. The Besov space $B_{p,q}^s$ consists of functions $f\in \mathcal{S}$ such that
\[
\|f\|_{B_{p, q}^s} := \lt( \|P_{< 1} f\|_{L^p}^q +\sum_{N\ge 1}^{\infty}  \| N^s P_N f\|_{L^p}^q\rt)^{1/q} <\infty.
\]  
%
%
%
%
%
%
%
%
%
 \subsection{Organization of the paper}
The remainder of this paper is organized as follows. In Section \ref{sec_2}, we introduce Shatah's normal form transformation, which is employed to facilitate the derivation of the $L^p$ decay estimates for the solution. Section \ref{sec_3} is devoted to establishing the decay estimates for the solution. This begins with an analysis of dispersive estimates for the linear solution (Proposition \ref{decay}), followed by the derivation of $L^p$ decay estimates for the solution to the nonlinear problem. A key step in obtaining these $L^p$ estimates involves the application of bilinear estimates in which Proposition \ref{Proposition_m} plays a crucial role. In Section  \ref{sec_4}, we focus on proving Proposition \ref{Proposition_m}. Additionally, this section addresses the derivation of $\dot{H}^{-(2-\sigma)/2}$ estimates, which are necessary for the $L^p$ decay estimates. Finally, in Section \ref{sec_5}, we provide the detailed proof of Theorem \ref{global} by establishing estimates for the negative Sobolev norm, leading to the successful derivation of the $H^{2s}$, supported by the $L^p$ decay estimates.

%
%
%
%
%
%
%
%
%

\section{Normal transformation}\label{sec_2}

To facilitate the analysis of our main system, we introduce a transformation that combines the original variables $n(t,x)$ and $u(t,x)$ into a single complex-valued function $\alpha(t,x)$. This transformation is designed to simplify the structure of the equations and to make the nonlinear interactions more tractable. By expressing the system in terms of $\alpha(t,x)$, we can more effectively apply analytical techniques such as Fourier analysis and the Duhamel principle.

The transformation is defined as follows:
\bq\label{sol}
\alpha(t, x)=\frac{p(|\nabla|)}{|\nabla|} n(t, x) +i \frac{\nabla }{|\nabla|}\cdot u(t, x),
\eq
where $p(|\nabla|)=(|\nabla|^2 +|\nabla|^{2-\sigma})^{1/2}$. Using this formulation, we can express the original variables $n$ and $u$ in terms of $\alpha$ and its complex conjugate $\bar\alpha$:
\bq\label{nu_alpha}
n=\frac{|\nabla|}{p(|\nabla|)}\lt( \frac{\alpha+\bar{\alpha}}{2}\rt)\quad \text{and} \quad u=\frac{i\nabla}{|\nabla|} \lt( \frac{\alpha-\bar{\alpha}}{2}\rt).
\eq
Substituting these into the system \eqref{ER_1} transforms it into
\[
\begin{aligned}
	(\pa_t -ip(|\nabla|)) \alpha 
	& =-\frac{i}{4} \frac{p(|\nabla|)}{|\nabla|}\nabla \cdot \lt( \frac{|\nabla|}{p(|\nabla|)}(\alpha+\bar{\alpha}) \frac{\nabla }{|\nabla|}(\alpha-\bar{\alpha})\rt)\\
	&\quad  +\frac{i}{8} |\nabla| \lt( \frac{|\nabla|}{p(|\nabla|)} (\alpha +\bar{\alpha})\rt)^2-\frac{i}{8} |\nabla|\lt| \frac{\nabla}{|\nabla|} (\alpha -\bar{\alpha})\rt|^2\\
	& =:  \mathcal{Q}(\alpha).
\end{aligned}
\]
Applying the Fourier transform, we obtain
\[
\begin{aligned}
	\mathcal{Q}(\alpha)
	& =  i\sum_{r,l=1}^2 \iint_{\R^3\times \R^3} e^{i x\cdot \xi} \lt( \frac{(-1)^{l}m_1}{4} +\frac{m_2}{8} -\frac{(-1)^{l+r}m_3}{8} \rt)  \widehat{\alpha_r}(\xi-\eta) \widehat{\alpha_l} (\eta) \,d \xi d \eta
\end{aligned}
\]
where $\alpha_1=\alpha$, $\alpha_2 =\bar{\alpha}$, and
\[
\begin{aligned}
	m_1=m_1(\xi, \eta)&= p(\xi)\frac{|\xi-\eta|}{p(\xi-\eta)} \frac{\xi\cdot \eta}{|\xi||\eta|}, \cr
	m_2=m_2(\xi, \eta)&= |\xi|\frac{|\xi-\eta||\eta|}{p(\xi-\eta)p(\eta)}, \cr
	m_3=m_3(\xi, \eta)&= |\xi|\frac{(\xi-\eta)\cdot \eta}{|\xi-\eta||\eta|},
\end{aligned}
\]
with $p(\xi)=p(|\xi|)$. After symmetrizing between $\xi-\eta$ and $\eta$, we simplify this as
\[
\begin{aligned}
	\mathcal{Q}(\alpha)
	& = \sum_{r, l=1}^2 \iint_{\R^3\times \R^3} e^{i x\cdot \xi} m_{r, l}(\xi, \eta) \widehat{\alpha_r}(\xi-\eta) \widehat{\alpha_l} (\eta) \,d\xi d \eta,
\end{aligned}
\]
where 
\[
\begin{aligned}
m_{r, l}(\xi, \eta) 
&= c_1 p(\xi) \frac{|\xi-\eta|}{p(\xi-\eta)}\frac{\xi\cdot \eta}{|\xi||\eta|}
+c_2 p(\xi) \frac{|\eta|}{p(\eta)}\frac{\xi\cdot (\xi-\eta)}{|\xi||\xi-\eta|} +c_3 |\xi|\frac{(\xi-\eta)\cdot \eta}{|\xi-\eta||\eta|}
+c_4  |\xi|\frac{|\xi-\eta||\eta|}{p(\xi-\eta)p(\eta)}\\
&\sim p(\xi) r(\xi)r(\xi-\eta) r(\eta)+|\xi| r(\xi-\eta) r(\eta)
\end{aligned}
\]
which led to $m_{r, l}(\xi, \eta)=m_{r, l}(\xi, \xi-\eta)$. Here, a constant $c_j$ for $j=1,2,3, 4$ is determined by the value of $(r, l)$ as $\pm i/8$ or $0$,
but no matter what the value of the constant is.
$r(x)$ is denoted as functions $x\mapsto \frac{x}{|x|}$ or $x\mapsto \frac{|x|}{p(x)}$,
which obeys  
\bq\label{est_r}
\| r(\nabla) f\|_{L^p} =\| \mathcal{F}^{-1} \{r \hat{f}\}\|_{L^p} \lesssim \|f\|_{L^p}, \qquad 1<p<\infty.
\eq
By the Duhamel principle, we get
\bq\label{duhamel}
\alpha(t) = e^{itp(|\nabla|)} \alpha(0) +\int_0^t e^{i(t-\tau)p(|\nabla|)} \mathcal{Q}(\alpha)(\tau, x)  \,d\tau.
\eq
For the linear profile
\bq\label{b}
b(t, x)=e^{-itp(|\nabla|)} \alpha(t, x),
\eq
this implies 
\bq\label{h}
\widehat{b}(t, \xi) =\widehat{\alpha}(0, \xi) +\int_{0}^t e^{-i\tau p(\xi)} \widehat{\mathcal{Q}(\alpha)}(\tau, \xi) \,d\tau,
\eq
and thus
\bq\label{b2}
\pa_t \widehat{ b}(\xi)= e^{-it p(\xi)} \widehat{\mathcal{Q}(\alpha)}(t, \xi).
\eq
Then the Duhamel term in \eqref{h} can be further rewritten as  
\bq\label{duhamel_1}
\begin{aligned}
\int_0^t e^{-i\tau p(\xi)} \widehat{\mathcal{Q}(\alpha)}(\tau, \xi)\,d\tau
&= \sum_{r,l=1}^2 \int_0^t\int_{\R^3} e^{-i\tau p(\xi)} m_{r, l} \widehat{\alpha_r}(\xi-\eta) \widehat{\alpha_l} (\eta) \, d \eta d\tau\\
& =  \sum_{r,l=1}^2 \int_{0}^t \int_{\R^3}  e^{-i\tau \Phi_{r,l}} m_{r, l}\widehat{b_r}(\xi-\eta) \widehat{b_l} (\eta) \, d \eta d \tau,
\end{aligned}
\eq
where $\Phi_{r,l}=\Phi_{r,l}(\xi, \eta) := p(\xi)+(-1)^{r}p(\xi-\eta)+(-1)^{l}p(\eta)$.
Using the integral by parts, we have
\bq\label{duhamel_2}
\begin{aligned}
	 \int_{0}^t \int_{\R^3} e^{-i\tau \Phi_{r,l}} m_{r, l} \widehat{b_r}(\tau, \xi-\eta) \widehat{b_l} (\tau, \eta) \, d\eta d \tau
	&=  i\int_0^t\int_{\R^3} \pa_\tau \lt( e^{-i\tau \Phi_{r, l}} \rt) \frac{m_{r, l}}{\Phi_{r, l}} \widehat{b_r}(\tau, \xi-\eta) \widehat{b_l} (\tau, \eta)\, d\eta  d\tau \\
	&=   i\int_{\R^3} e^{-it \Phi_{r,l}} \frac{m_{r, l}}{\Phi_{r, l}} \widehat{b_r}(t, \xi-\eta) \widehat{b_l} (t, \eta) \, d\eta\\
	&\quad  - i\int_{\R^3}\frac{m_{r, l}}{\Phi_{r, l}} \widehat{b_r}(0, \xi-\eta) \widehat{b_l} (0, \eta) \, d\eta \\
	&\quad -2i\int_{0}^t\int_{\R^3}  e^{-i\tau \Phi_{r,l}} \frac{m_{r, l}}{\Phi_{r, l}} \widehat{ b_r}(\tau, \xi-\eta) \widehat{\pa_\tau b_l} (\tau, \eta)\, d\eta d \tau\\
\end{aligned}
\eq
with summation over $r, l=1,2$. Here we used for the last term above the symmetry property $m_{r, l}(\xi, \eta)=m_{r, l}(\xi, \xi-\eta)$ and $\Phi_{r, l}(\xi, \eta)=\Phi_{r,l}(\xi, \xi-\eta)$.
We combine \eqref{duhamel_1} and \eqref{duhamel_2} along with \eqref{b} and \eqref{b2} and then we obtain  
\[
\begin{aligned}
\int_0^t e^{-i\tau p(\xi)} \widehat{\mathcal{Q}(\alpha)}(\tau, \xi)\,d\tau & = e^{-it p(\xi)} \hat{g}(t, \xi)-\hat{g}(0, \xi) +e^{-it p(\xi)}\hat{h}(t, \xi),
\end{aligned}
\]
where
\bq\label{bilinear}
\begin{aligned}
	g (t, x) &=i \sum_{r,l=1}^2 \mathcal{F}^{-1} \lt( \int_{\R^3}  \frac{m_{r, l}}{\Phi_{r, l}} \widehat{\alpha_r}(t, \xi-\eta) \widehat{\alpha_l} (t, \eta) \, d \eta\rt),\\
	h(t, x)&= -2i  \sum_{r,l=1}^2 \mathcal{F}^{-1} \lt( \int_{0}^t \int_{\R^3}  e^{i(t-\tau) p(\xi)} \frac{m_{r, l}}{\Phi_{r, l}} \widehat{ \alpha_r}(\tau, \xi-\eta) \widehat{\mathcal{Q}} (\tau, \eta) \, d \eta d \tau\rt).
\end{aligned}
\eq
From \eqref{h}, hence it yields that 
\bq\label{sol_2}
\begin{aligned}
 \alpha(t, x)= e^{itp(|\nabla|)}\alpha(0, x) + g(t, x) -e^{itp(|\nabla|)}g(0, x) +h(t, x).
\end{aligned}
\eq

%
%
%
%
%
%
%
%
%
\section{$L^p$ decay estimates}\label{sec_3}
In this section, we present the $L^p$ decay estimate for $\alpha$ given as in \eqref{sol} within $W^{s,p}$ for some $s$ and $p$, which will be specified below. We first analyze the linear flow part, the first term on the right-hand side of \eqref{sol_2}. Subsequently, we estimate the remaining terms, constituting the other three terms.
%
%
%
%
%
%
%
%
%
\subsection{Linear flow}

We investigate time decay estimates for the linear flow $e^{itp(|\nabla|)}$ associated with the dispersion relation $p(r)=(r^2+r^{2-\sigma})^{1/2}$ for $r>0$ with $0<\sigma<2$. Precisely, in this subsection, our goal is to prove the following proposition.
\begin{proposition}\label{decay}
Let $0< \sigma<2$. The following estimate holds;
	\bq\label{decay_0}
	\|e^{it p(|\nabla|)} f\|_{L^{p}} \lesssim (1+t)^{-\beta_{\sigma}} \|f\|_{W^{6(\frac{1}{2}-\frac{1}{p}), p'}}
	\eq	
	for $2\leq p < \infty$, where $\beta_{\sigma}$ is defined as in \eqref{beta}. 
\end{proposition}

\begin{remark}
	The case $p=\infty$ can be derived by applying the Gagliardo--Nirenberg--Sobolev inequality to \eqref{note_5} with additional regularity for $f$. However, this is not relevant to our primary objectives and thus will not be considered further.
\end{remark}

The case $\sigma=0$ corresponds to a dispersion relation of homogenous type, where the decay rate of $e^{itp(|\nabla|)}$ is known to be $|t|^{-1}$. This is related to the rank of the Hessian matrix of the dispersion relation $p$. For $\sigma>0$, the dispersion relation $p(r)$ becomes non-homogeneous, introducing complexities in obtaining decay estimates due to its distinct behavior at low and high frequencies. Despite these complications, one might expect a decay rate of $|t|^{-3/2}$ provided that the surface defined by $p$ exhibits nonzero Gaussian curvature. Guo, Peng, and Wang \cite{GPW} studied such decay estimates by separating them into high and low-frequency parts with different scales as follows.  

\begin{theorem}[\cite{GPW}]\label{decay_GPW}
Assume that $p: \mathbb{R}_{+} \rightarrow \mathbb{R}$ is smooth away from the origin. 
If there exist constants $\alpha_1, \alpha_2 >0$ such that
\[
|p'(r)| \sim r^{\alpha_1-1},\qquad |p^{(k)} (r)| \lesssim r^{\alpha_1-k} \qquad \text{for }r\ge1, 
\]
\[
|p'(r)| \sim r^{\alpha_2-1},\qquad |p^{(k)} (r)| \lesssim r^{\alpha_2-k}\qquad \text{for }r<1,
\]
and, moreover, if there exist $m_1, m_2 \in \mathbb{R}$ such that
\bq\label{replacement}
|p''(r)| \sim r^{m_1-2} \text{ for } r \ge1 \quad \text{ and } \quad |p''(r)|\sim r^{m_2-2} \text{ for } r <1,
\eq 
then the following estimates hold;
\[
\begin{aligned}
\|e^{itp(|\nabla|)} P_N f\|_{L^{\infty}} \lesssim |t|^{-\frac{3}{2}} N^{(3-\frac{3\alpha_1}{2}-\frac{m_1-\alpha_1}{2})} \|f\|_{L^1} \qquad N\ge 1\\
\|e^{itp(|\nabla|)} P_N f\|_{L^{\infty}} \lesssim |t|^{-\frac{3}{2}} N^{(3-\frac{3\alpha_2}{2}-\frac{m_2-\alpha_2}{2})} \|f\|_{L^1} \qquad N<1
\end{aligned}
\]	
where $P_N$ is the Littlewood--Paley operator.
\end{theorem}

For our dispersion relation $p(r)=(r^2+r^{2-\sigma})^{1/2}$ with $r>0$, we observe that 
\bq\label{deriv_P_first}
p'(r) =\frac{2r^{\sigma}+2-\sigma}{2 r^{\sigma/2} (r^{\sigma}+1)^{1/2}}
\eq 
where $p'(r)$ is positive for $r>0$, and
\[
p''(r)=\frac{2\sigma(\sigma-1)r^{\sigma}-\sigma(2-\sigma)}{4 r^{1+\sigma/2}(r^{\sigma}+1)^{3/2}}.
\]
This leads to the asymptotic behaviors:
\bq\label{behavior_P1}
|p'(r)| \sim 1,\quad |p^{(k)} (r)| \lesssim r^{1-k}, \qquad \text{when }r\ge1 \eq
and
\bq\label{behavior_P2}
|p'(r)| \sim r^{-\frac{\sigma}{2}},\quad |p^{(k)} (r)| \lesssim r^{\frac{2-\sigma}{2}-k},\qquad \text{when }r<1
\eq
for any integer $k \ge 2$ and $0<\sigma<2$.
In particular, for $0<\sigma\leq 1$, there is no degenerate point $r_0$ such that $p''(r_0)=0$. This implies $|p''(r)|>0$ for any $r>0$, allowing us to verify that   \eqref{replacement} holds with $(m_1, m_2)=(1-\sigma, 1-\sigma/2)$ for $0<\sigma<1$ and $(m_1, m_2)=(-1, 1/2)$ for $\sigma=1$, as summarized in Table \ref{table} below. Details are provided in Appendix \ref{appendix_3}.

\begin{table}
	\centering
	\begin{tabular}{ | c | c | c | c | c |  }
		\hline
		& $0<\sigma<1$ & $\sigma=1$ & $1<\sigma\leq \frac43$ & $\frac43<\sigma<2$\\
		\hline 
		$r\ge 1$ & $r^{-1-\sigma}$ & $r^{-3}$ & $r^{-1-2\sigma}$ for $r\neq r_0$& $r^{-1-\sigma}$ for $r\neq r_0$ \\ [1mm]
		\hline
		 $r<1$ & \multicolumn{4}{c|}{$r^{-1 - \frac\sigma2}$}  \\[1mm]
		\hline  
	\end{tabular}
	\caption{Behavior of $|p''(r)|$ for $0<\sigma<2$.}
	\label{table}
\end{table}
 
Thus, a direct application of Theorem \ref{decay_GPW} to a semi-group operator $e^{itp(|\nabla|)}$ implies the following decay estimates for $0< \sigma \leq 1$.  

\begin{lemma}\label{dispersive_1}
	Let $0<\sigma \leq 1$. Then the following estimates hold
	\begin{align*}
\|e^{it p(|\nabla|) } P_N f\|_{L^{\infty}} &\lesssim   |t|^{-\frac{3}{2}} N^{\frac{5}{2}} \|f\|_{L^1} \quad \qquad \text{for } N\ge1,\\
\|e^{it p(|\nabla|) } P_N f\|_{L^{\infty}} &\lesssim  |t|^{-\frac{3}{2}} N^{\frac{3}{2}+\frac{3\sigma}{4}}\|f\|_{L^1} \quad \text{ for } N< 1.
	\end{align*}
\end{lemma}

For $1<\sigma < 2$, a degenerate point $r_0$ exists, given by
\bq\label{degerate pt}
r_0 := \lt(\frac{2-\sigma}{2(\sigma-1)}\rt)^{1/\sigma}>0
\eq 
such that $p''(r_0)=0$. We define a frequency-localized operator $Q_{r_0}$ around the degeneracy point $r_0$ as 
\[
Q_{r_0}f = \mathcal{F}^{-1} \{ \Psi_{r_0} \hat{f} \},
\]
where $\Psi_{r_0}:\mathbb{R} \rightarrow [0, 1]$ is a radial smooth function satisfying $\Psi_{r_0}(\xi)=1$ if $||\xi|-r_0|\leq \varepsilon$ and $\Psi_{r_0}(\xi)=0$ if $||\xi|-r_0| >2\varepsilon$.
Next, we consider the operator $1-Q_{r_0}$ whose frequency lies apart from the point $r_0$, defined as
$$(1-Q_{r_0}) f = \mathcal{F}^{-1} \{\Psi_{<r_0} \hat{f} +\Psi_{>r_0} \hat{f}\}, $$
where radial functions $\Psi_{<r_0}$ and $\Psi_{>r_0}$ satisfy
\[
\Psi_{<r_0}(\xi)=1 \quad \mbox{if $|\xi|<r_0-2\varepsilon$} \quad \mbox{and} \quad \Psi_{<r_0}(\xi)=0 \quad \mbox{if $|\xi|>r_0$},
\] 
and 
\[
\Psi_{>r_0}(\xi)=1 \quad \mbox{if $|\xi|>r_0+2\varepsilon$} \quad \mbox{and} \quad \Psi_{>r_0}(\xi)=0\quad \mbox{if $|\xi|<r_0$}
\]
such that
\[
\Psi_{<r_0}(\xi) +\Psi_{r_0}(\xi) +\Psi_{>r_0}(\xi)=1.
\]

The following lemma presents the dispersive estimates for frequencies localized around the degeneracy $r_0$ and for those apart from $r_0$.
\begin{lemma}\label{dispersive_2}
	Let $1<\sigma<2$. Then the following estimates hold:
	\begin{align}
		\label{degeneracy}
		\|e^{it p(|\nabla|) } Q_{r_0} f\|_{L^{\infty}} &\lesssim   |t|^{-\frac{4}{3}}  \|f\|_{L^1}, \\
		\label{high1}
		\|e^{itp(|\nabla|)} P_N (1-Q_{r_0}) f\|_{L^{\infty}}& \lesssim   |t|^{-\frac{3}{2}} N^{\frac{17}{6}} \|f\|_{L^1} \quad \quad \mbox{for } N\ge1, \\
		\label{low1}
		\|e^{it p(|\nabla|) } P_N (1-Q_{r_0})f\|_{L^{\infty}}& \lesssim  |t|^{-\frac{3}{2}} N^{\frac{3}{2}+\frac{3\sigma}{4}} \|f\|_{L^1} \quad \mbox{for }   N< 1,
	\end{align}
where $r_0$ is given as in \eqref{degerate pt}. 
\end{lemma}
\begin{proof} We begin by noting that for $r\neq r_0$, Table \ref{table} indicates that $p''(r)$ satisfies \eqref{replacement} with $(m_1, m_2)=(1-2\sigma, 1-\sigma/2)$ for $1<\sigma\leq 4/3$ and $(m_1, m_2)=(1-\sigma, 1-\sigma/2)$ for $4/3<\sigma<2$.
Applying Theorem \ref{decay_GPW} along with \eqref{behavior_P1} and \eqref{behavior_P2} provides the estimates for \eqref{high1} and \eqref{low1}. 

Here, we focus on proving \eqref{degeneracy}. By Young's inequality, we obtain
	\[
	\begin{aligned}
		\|e^{it p(|\nabla|) } Q_{r_0} f\|_{L^{\infty}} =
		\|\mathcal{F}^{-1} (e^{it p(|\xi|) } \Psi_{r_0})\ast f\|_{L^{\infty}} \lesssim \|\mathcal{F}^{-1} (e^{it p(|\xi|) } \Psi_{r_0} )\|_{L^{\infty}} \|f\|_{L^{1}}. 
	\end{aligned}
	\]
	Thus, it suffices to show that 
	\[
	\|\mathcal{F}^{-1} (e^{it p(|\xi|) } \Psi_{r_0} )\|_{L^{\infty}} \lesssim |t|^{-\frac{4}{3}}.
	\]
	Since $\Psi_{r_0}$ is radially symmetric, we can express its inverse Fourier transform as
	\[
	\mathcal{F}^{-1} (e^{itp(|\xi|)} \Psi_{r_0})(x)= 2\pi \int_{0}^{\infty} e^{itp(r)} \Psi_{r_0}(r) (r|x|)^{-\frac{1}{2}} J_{\frac{1}{2}} (r|x|) r^{2}\, dr,
	\]	
	where $J_{m}(s)$ is a Bessel function (see e.g. \cite{S71}). According to \cite[Chapter 1]{J}, we find
	\[
	s^{-m} J_{m}(s)=c Re(e^{is} Z(s)) =\frac{c}{2} (e^{is} Z(s)+e^{-is} \bar{Z}(s)),
	\]
	where $Z$ satisfies
	\bq\label{property_Z}
	|\pa^k Z(s)| \lesssim (1+s)^{-(1+k)}, \qquad k\ge0.
	\eq
	Then, we have
	\[
	\begin{aligned}
		\mathcal{F}^{-1}(e^{itp(|\xi|)} \Psi_{r_0}) (x)
		&= c \int_{0}^{\infty} e^{itp(r)} \Psi_{r_0}(r) r^{2} (e^{ir|x|}Z(r|x|)+e^{-ir|x|} \bar{Z}(r|x|)) \, dr\\
		&=  c \int_{0}^{2\varepsilon} e^{it\phi_{1}(r+r_0, |x|, t)}\tilde{\Psi}_{r_0}(r+r_0) Z((r+r_0)|x|)\, dr\\
		&\quad +c\int_{0}^{2\varepsilon} e^{it\phi_{2}(r+r_0, |x|, t)}\tilde{\Psi}_{r_0}(r+r_0) \bar{Z}((r+r_0)|x|)\, dr\\
		&=: I+II
	\end{aligned}
	\]
	by the change of variable $r \rightarrow r+r_0$. Here,  $\tilde{\Psi}_{r_0}(r):=\Psi_{r_0}(r)r^{2} $, and the
    phase functions $\phi_1$ and $\phi_2$ are defined by
	\[
	\phi_{1}(r, |x|, t) = p(r)+\frac{r|x|}{t} \quad \mbox{and} \quad \phi_{2}(r, |x|, t) = p(r)-\frac{r|x|}{t},
	\]
	respectively. For $I$, we observe
	\[
	\frac{d}{dr} \phi_1(r+r_0, |x|, t) =p'(r+r_0)+\frac{|x|}{t} \ge \frac{1}{4}p'(r_0)>0 \quad  \text{for } r\leq 2\varepsilon
	\]
	since  $p'(r+r_0) \ge 2^{-\sigma} p'(r_0)$ from \eqref{deriv_P_first} when $2\varepsilon<r_0$.
	Then, we use the non-stationary phase method to get
	\[
	|I| \lesssim t^{-\frac{3}{2}}.
	\]
	
	For $II$, we divide the range of $|x|$ into $|x| \leq \frac{1}{8}p'(r_0)t$ and $|x| > \frac{1}{8}p'(r_0)t$.
	When $|x| \leq \frac{1}{8}p'(r_0)t$, we get
	\[
	\frac{d}{dr} \phi_{2}(r+r_0, |x|, t)  =p'(r+r_0)-\frac{|x|}{t} \ge \frac{1}{8}p'(r_0)>0 \qquad \text{for } r\leq 2\varepsilon, 
	\]
	and then by the non-stationary phase method, we deduce 
	\[
	|II|\lesssim t^{-\frac{3}{2}}.
	\]
	
	In the case of $|x|>\frac{1}{8} p'(r_0) t$, there might be points where $\frac{d}{dr}\phi_2(r+r_0, |x|, t)=0$. Since $p''(r_0)=0$ and $p'''(r_0)>0$, by using the Var der Corput lemma (see e.g \cite{S93}), we have
	\[
	\lt| \int_{0}^{r} e^{it\phi_2(\tilde{r}+r_0, |x|, t)}\, d\tilde{r} \rt| \lesssim |t|^{-\frac{1}{3}}
	\]
	for any interval of $\tilde{r}$. 
	Using the fact 
	\[
	\frac{d^k}{dr^{k}}\tilde{\Psi}_{r_0}(r+r_0) \sim r_0^{2-k} \quad \mbox{for $k=0,1$} \quad \mbox{and} \quad 0<r<2\varepsilon, 
	\]
	and \eqref{property_Z}, we deduce 
	\[
	\begin{aligned}
		&\lt| \int_{0}^{2\varepsilon} e^{itp(r+r_0)-i(r+r_0)|x|}\tilde{\Psi}_{r_0}(r+r_0) \bar{Z}((r+r_0)|x|)\, dr \rt| \\
		&\quad = \lt|  \int_{0}^{2\varepsilon}\lt(\int_{0}^{r} e^{it\phi_2(\tilde{r}+r_0, |x|, t)}\, d\tilde{r}\rt) \, \frac{d}{dr} \lt(   \tilde{\Psi}_{r_0}(r+r_0) \bar{Z}((r+r_0)|x|)\rt)\, dr \rt|\\
		&\quad \leq c_{\varepsilon} |t|^{-\frac{1}{3}}\sup_{r<2\varepsilon}\lt( \big|\tilde{\Psi}_{r_0}' (r+r_0) \bar{Z}((r+r_0)|x|)\big| + |x|\big| \tilde{\Psi}_{r_0}(r+r_0) \bar{Z}'((r+r_0)|x|)\big| \rt)\\
		&\quad \leq c_{\varepsilon} |t|^{-\frac{1}{3}} |x|^{-1} \\
		&\quad \leq c_{\varepsilon} |t|^{-\frac{4}{3}}
	\end{aligned}
	\]
	for $|x|>\frac{1}{8} p'(r_0)t$.
	Combining these results, we obtain
	\[
	\|\mathcal{F}^{-1} (e^{it p(|\xi|) } \Psi_{r_0} )\|_{L^{\infty}} \lesssim |t|^{-\frac{4}{3}},
	\]	
	which completes the proof of  \eqref{degeneracy}.	
\end{proof}

We are now in a position to provide the detailed proof for Proposition \ref{decay}.

\begin{proof}[Proof of Proposition \ref{decay}]
For the proof, we employ an interpolation argument that involves Lemma \ref{dispersive_1} for $0<\sigma\leq 1$, Lemma \ref{dispersive_2} for $1<\sigma<2$, and some standard estimates. Since $e^{itp(|\nabla|)}$ is unitary operator on $L^2$, we obtain
\[
\|e^{itp(|\nabla|)} P_N f\|_{L^2} \lesssim \| f\|_{L^2}
\]
for any $N>0$.
For the case $1<\sigma<2$, we interpolate this result with the estimate given in \eqref{low1}. Consequently, we have
\[
\|e^{itp(|\nabla|)} P_N f\|_{L^{p}} \lesssim |t|^{-3(\frac{1}{2}-\frac{1}{p})} N^{(3+\frac{3\sigma}{2})(\frac{1}{2}-\frac{1}{p})}\| f\|_{L^{p'}}
\]
for $p \ge 2$.
This implies 
\bq\label{low_est}
\|e^{itp(|\nabla|)} P_{<1} f\|_{L^{p}}  \lesssim \sum_{N<1} \|e^{itp(|\nabla|)} P_{N} f\|_{L^{p}}  \lesssim |t|^{-3(\frac{1}{2}-\frac{1}{p})} \|f\|_{L^{p'} }.
\eq
For $N\ge1$, 
we interpolate between \eqref{high1} with $f$ replaced by $P_Nf$ and the following estimate
\[
\|e^{it p(|\nabla|) } P_Nf\|_{L^2} \lesssim  \| P_Nf\|_{L^2}.
\]
Thus, we derive the following estimate
\[
\|e^{it p(|\nabla|) } P_N f\|_{L^{p}} \lesssim   |t|^{-3(\frac{1}{2}-\frac{1}{p})}  \| |\nabla|^{\frac{17}{3}(\frac{1}{2}-\frac{1}{p})}P_Nf\|_{L^{p'}}
\]
for $p\ge 2$. 
By using this,
we obtain 
\bq\label{note_5}
\begin{aligned}
\|e^{itp(|\nabla|)} P_{\ge 1} f\|_{L^p}^2
&\lesssim \sum_{N\ge 1} \|e^{itp(|\nabla|)} P_N f\|_{L^p}^2\\
&\lesssim |t|^{-6(\frac{1}{2}-\frac{1}{p})}\sum_{N\ge 1} \|P_N |\nabla|^{\frac{17}{3}(\frac{1}{2}-\frac{1}{p})}f\|_{L^{ p'}}^2\\
&\lesssim |t|^{-6(\frac{1}{2}-\frac{1}{p})} \| f\|_{W^{\frac{17}{3}(\frac{1}{2}-\frac{1}{p}), p'}}^2\\
\end{aligned}
\eq
since $L^{p'} \subset B_{p', 2}^0$ for $2\leq p <\infty$.
From the embedding relation $B_{p,2}^0 \subset L^{p}$ for $2\leq p<\infty$, we use \eqref{low_est} and \eqref{note_5} to get 
\bq\label{t>1}
\begin{aligned}
\|e^{itp(|\nabla|)} f\|_{L^p}^2 
 \lesssim \| e^{itp(|\nabla|)}P_{<1}f\|_{L^p}^2 + \|e^{itp(|\nabla|)}P_{\ge 1} f\|_{L^p}^2  \lesssim |t|^{-6(\frac{1}{2}-\frac{1}{p})} \|f\|_{W^{\frac{17}{3}(\frac{1}{2}-\frac{1}{p}),p'}}
\end{aligned}
\eq
for $t\ge 1$. When $t<1$, we deduce
\bq\label{t<1}
\| e^{itp(|\nabla|)} f\|_{L^p} \lesssim \| e^{itp(|\nabla|)} f \|_{\dot{H}^{3(\frac{1}{2}-\frac{1}{p})}} \lesssim \|f\|_{\dot{H}^{3(\frac{1}{2}-\frac{1}{p})}} \lesssim \|f\|_{\dot{W}^{6(\frac{1}{2}-\frac{1}{p}), p'}}
\eq
by the Hardy--Littlewood--Sobolev inequality and the Plancherel theorem.
Therefore, a combination of \eqref{t>1} and \eqref{t<1} concludes the desired estimate for $1< \sigma<2$. 
The case $0<\sigma \leq 1$ can be shown in a similar way thanks to Lemma \ref{dispersive_1}. This completes the proof.
\end{proof}
%
%
%
%
%
%
%
%
%
\subsection{Decay estimates}
This subsection is dedicated to establishing $L^p$ type estimates by applying Proposition \ref{decay} to \eqref{sol_2}. We consider the bilinear pseudo-product operator defined as follows:
\[
B_m[f, g] =\mathcal{F}_{\xi}^{-1} \int_{\mathbb{R}^3} m(\xi, \eta) \hat{f}(\xi-\eta) \hat{g}(\eta) \,d\eta.
\]
To achieve the desired $L^p$ estimates, it is required to utilize bilinear estimates for $g$ and $h$ as specified in \eqref{bilinear}. However, it is important to note that our kernel $\frac{m_{r,l}}{\Phi_{r, l}}$ associated with the bilinear operator, as indicated in \eqref{bilinear}, is not locally bounded due to the singularity presented in $\frac{1}{\Phi_{r, l}}$ (e.g. see Lemma \ref{Lemma_low_Phi} for further details).  To address this issue, we construct a new kernel defined by
\bq\label{kernel_new}
\mathfrak{M}_{r, l}=\frac{|\xi|^{\frac{2-\sigma}{2}}|\eta|^{\frac{2-\sigma}{2}}|\xi-\eta|^{\frac{2-\sigma}{2}}}{\lal \xi-\eta\ral^{2 \lambda} \lal \eta\ral^{2 \lambda}\Phi_{r, l}}
\eq
for $\lambda>0$ will be chosen later. We shall employ Gustafson--Nakanishi--Tsai's estimate \cite{GNT07} by introducing the following norm:
\[
\|\mathfrak{M}\|_{M_{\xi, \eta}^{k, \infty}} =\sum_{N\in 2^{\mathbb{Z}}} \|P_N^\eta \mathfrak{M}\|_{L_{\xi}^{\infty} \dot{H}_{\eta}^k}.
\]
Then the following estimates were presented in \cite{GNT07, GP}.
\begin{theorem}\label{GNT estimate}
Let $0\leq k\leq 3/2$, and suppose $\|\mathfrak{M}\|_{M_{\eta, \xi}^{k, \infty}}, \|\mathfrak{M}\|_{M_{\xi, \eta}^{k, \infty}} < \infty$. Then, we have
	\bq\label{bilinear_1}
	\|B_{\mathfrak{M}}[f, g]\|_{L^{l_1'}} \lesssim \|\mathfrak{M}\|_{M_{\eta, \xi}^{k, \infty}} \|f\|_{L^{l_2}} \|g\|_{L^2}
	\eq
	and
\[
	\|B_{\mathfrak{M}}[f, g]\|_{L^{2}} \lesssim \|\mathfrak{M}\|_{M_{\xi, \eta}^{k, \infty}} \|f\|_{L^{l_1}} \|g\|_{L^{l_2}}
\]
for $f, g \in \mathcal{S}$ where
	\[
	2\leq l_1, l_2 \leq \frac{6}{3-2k} \quad \mbox{and} \quad \frac{1}{l_1}+\frac{1}{l_2} =1-\frac{k}{3}.
	\]
	Here, $\mathcal{S}$ is the Schwartz class.
\end{theorem}

By applying the product rule at the frequency level along with Theorem \ref{GNT estimate}, we can easily derive the following lemma. For the sake of completeness, we provide the proof in Appendix \ref{proof_Lemma_product} below.
\begin{lemma}\label{GNT_estimate_product}
	Let $\alpha_1, \alpha_2 \ge 0$.
	Under the conditions of Theorem \ref{GNT estimate}, we have
	\bq\label{bilinear_11}
	\|\lal \nabla \ral^{\alpha_1} D^{\alpha_2} B_{\mathfrak{M}}[f, g]\|_{L^{l_1'}} \lesssim  \|f \|_{L^{l_2}} \|g\|_{H^{\alpha_1+\alpha_2}} +\|g\|_{L^{l_2}} 
	\|f \|_{H^{\alpha_1+\alpha_2}}
	\eq
	and 	
	\[
	\|\lal \nabla \ral^{\alpha_1} D^{\alpha_2} B_{\mathfrak{M}}[f, g]\|_{L^{2}} \lesssim  
	\|f \|_{W^{\alpha_1+\alpha_2, l_1}}\|g\|_{L^{l_2}}  + \|g\|_{W^{\alpha_1+\alpha_2, l_1}} \|f \|_{L^{l_2}},
	\]
	where $D^{\alpha} = \lal|\nabla|^{\alpha} \ral$ or $D^{\alpha} =|\nabla|^{\alpha}$.
\end{lemma}

In order to employ Theorem \ref{GNT estimate} and Lemma \ref{GNT_estimate_product},
it is essential to demonstrate the following proposition, whose proof is provided in Section \ref{sec_4}.
\begin{proposition}\label{Proposition_m}
	Let $0< \sigma<2$ and $r, l=1,2$. For $0\leq k\leq 3/2$,
	then we have 
	\[
	\|\mathfrak{M}_{r, l}\|_{M_{\xi, \eta}^{k, \infty}} +\|\mathfrak{M}_{r, l}\|_{M_{\eta, \xi}^{k, \infty}} \lesssim 1
	\]
	with $\mathfrak{M}_{r, l}$ given as \eqref{kernel_new} with $\lambda> 7/4+\sigma k$.
\end{proposition}

We now establish the $L^p$ decay estimates in the following theorem.
\begin{theorem}\label{Theorem_decay}
	Let $0< \sigma<2$ and $p$ satisfies \eqref{p}. For $s\ge 16$ and $T^{\ast}>0$, then we have
	\[
	(1+t)^{\beta_{\sigma}}\|\alpha(t)\|_{W^{s, p}}\lesssim \|\alpha(0)\|_{Y}  +  \sup_{0<t<T^{\ast}}\|\alpha(t)\|_{X}^2 +\sup_{0<t<T^{\ast}}\|\alpha(t)\|_{X}^3,
	\]
	where $\beta_{\sigma}$ is given as in \eqref{beta}.
\end{theorem}

\begin{remark}
A condition $s>\max\{13/2+4\sigma-3/p, 13/3+35\sigma/6\}$ is required in the proof. For simplicity of presentation, we impose the condition $s\ge 16$ when $0<\sigma<2$.
\end{remark}

\begin{proof}[Proof of Theorem \ref{Theorem_decay}]
	From \eqref{sol_2}, we first see that
	\[
	\alpha(t)=e^{itp(|\nabla|)} \alpha(0)+g(t)-e^{itp(|\nabla|)}g(0) +h(t) .
	\]
	Then, it suffices to show the following estimates:
	\[
		 \|h(t)\|_{W^{s, p} }\lesssim (1+t)^{-\beta_{\sigma}}\|\alpha\|_{X}^3, \quad \|g(t)\|_{W^{s, p}}\lesssim (1+t)^{-\beta_{\sigma}}\|\alpha\|_{X}^2,\quad \mbox{and} \quad \|g(0)\|_{W^{s+6(\frac{1}{2}-\frac{1}{p}), p'}} \lesssim \|\alpha\|_{X}^2.
	\]
	Once these estimates are obtained, it follows from Proposition \ref{decay_0} that
	\[
	\begin{aligned}
		\|\alpha(t)\|_{W^{s, p}} 
		&\lesssim (1+t)^{-\beta_{\sigma}}(\|\alpha(0)\|_{W^{s+6(\frac{1}{2}-\frac{1}{p}), p'}}  +  \|g(0)\|_{W^{s+6(\frac{1}{2}-\frac{1}{p}), p'}} ) + \|h(t)\|_{W^{s, p}}  +\|g(t)\|_{W^{s, p}}\\
		&\lesssim  (1+t)^{-\beta_{\sigma}}(\|\alpha(0)\|_{Y}  +  \|\alpha\|_{X}^2 +\|\alpha\|_{X}^3 )
	\end{aligned}
	\]
	which deduces the desired result.
	
	$\bullet$ Estimates for $h(t)$.
	We begin with the expression for $h(t)$ given by  
	\[
	h(t)=-2i  \mathcal{F}_{\xi}^{-1}\lt( \int_{0}^t \int_{\R^3} e^{i(t-\tau)p(\xi)} \frac{m(\xi, \eta)}{\Phi(\xi, \eta)} \widehat{\alpha}(\xi-\eta)\widehat{\mathcal{Q}(\alpha)}(\tau, \eta) \,d\eta d\tau\rt)
	\]
	omitting subscripts $r, l$ for simplicity in our analysis, where
	\[
	\widehat{\mathcal{Q}(\alpha)}(\tau, \eta)=\int_{\R^3} m(\eta, \zeta) \widehat{\alpha}(\eta-\zeta)\widehat{\alpha}(\zeta) \,d\zeta.
	\]
	We note that
	\[
	\begin{aligned}
		\frac{m(\xi, \eta)}{\Phi(\xi, \eta)} 
		&=c\,\mathfrak{M}(\xi, \eta) (1+|\xi|^{\sigma})^{\frac{1}{2}}r(\xi) \frac{r(\xi-\eta)\lal \xi-\eta\ral^{2\lambda}}{|\xi-\eta|^{(2-\sigma)/2}} \frac{r(\eta)\lal \eta\ral^{2\lambda}}{|\eta|^{(2-\sigma)/2}} \\
		&\quad +c\,\mathfrak{M}(\xi, \eta) |\xi|^{\frac{\sigma}{2}} \frac{r(\xi-\eta)\lal \xi-\eta\ral^{2\lambda}}{|\xi-\eta|^{(2-\sigma)/2}} \frac{r(\eta)\lal \eta\ral^{2\lambda}}{|\eta|^{(2-\sigma)/2}}, 
	\end{aligned}	
	\]
	and then by using Proposition \ref{decay}, the relation \eqref{est_r}, and Lemma \ref{GNT_estimate_product} with $k=3/2-\delta$ for sufficiently small $\delta>0$ in turn, we obtain
	\[
	\begin{aligned}
		\|h(t) \|_{W^{s, p}}
		&\lesssim \int_0^t (1+t-\tau)^{-\beta_{\sigma}} \lt\| \mathcal{F}^{-1}_{\xi} \lt( \int_{\R^3} \lal \xi\ral^{s+6(\frac{1}{2}-\frac{1}{p})}\frac{m(\xi, \eta)}{\Phi(\xi,\eta)} \widehat{\alpha}(\tau,\xi-\eta) \widehat{\mathcal{Q}(\alpha)} (\tau,\eta) d\eta \rt) \rt\|_{L^{p'}} d\tau\\
		&  \lesssim \int_0^t (1+t-\tau)^{-\beta_{\sigma}}
		\lt\| \mathcal{F}^{-1}_{\xi} \lt( \int_{\R^3} \lal \xi\ral^{s+6(\frac{1}{2}-\frac{1}{p})}(D_1^{\frac{\sigma}{2}}+D_2^{\frac{\sigma}{2}}) \mathfrak{M}(\xi, \eta)  \widehat{A}(\tau,\xi-\eta) \widehat{B} (\tau,\eta) d\eta \rt) \rt\|_{L^{p'}} d\tau\\
		& \lesssim  \int_0^t (1+t-\tau)^{-\beta_{\sigma}} \lt( \| A\|_{H^{s+\frac{\sigma}{2}+6(\frac{1}{2}-\frac{1}{p})}} \|B\|_{L^{\frac{1}{\frac{1}{2}-\frac{1}{p}+\frac{\delta}{3}}}} +\|A\|_{L^{\frac{1}{\frac{1}{2}-\frac{1}{p}+\frac{\delta}{3}}}}\|B\|_{H^{s+\frac{\sigma}{2}+6(\frac{1}{2}-\frac{1}{p})}}\rt) d\tau,
	\end{aligned}
	\]
	where we let $D_1^{\sigma/2}=(1+|\xi|^{\sigma})^{1/2}$, $D_2^{\sigma/2}=|\xi|^{\sigma/2}$, and
	\bq\label{set_A}
	\widehat{A}(\xi-\eta) = \frac{r(\xi-\eta) \lal \xi-\eta\ral^{2\lambda}}{|\xi-\eta|^{(2-\sigma)/2}} \widehat{\alpha}(\xi-\eta), \quad 
	\widehat{B}(\eta) = \frac{r(\eta) \lal \eta\ral^{2\lambda}}{|\eta|^{(2-\sigma)/2}} \widehat{\mathcal{Q}(\alpha)}(\eta).
	\eq
	We observe that for given $\alpha \in X$, the following inequalities hold:
	\[
	\begin{aligned}
		&\|A\|_{H^{s+\frac{\sigma}{2}+6(\frac{1}{2}-\frac{1}{p})}}
		= \| |\nabla|^{-\frac{2-\sigma}{2}} \alpha \|_{H^{s+\frac{\sigma}{2}+2\lambda +6(\frac{1}{2}-\frac{1}{p})}}\lesssim \|\alpha\|_{\dot{H}^{-\frac{2-\sigma}{2}} \cap H^{s+\frac{\sigma}{2}+2\lambda +6(\frac{1}{2}-\frac{1}{p})} }\lesssim \|\alpha\|_{X},\\
		&\|A\|_{L^{\frac{1}{\frac{1}{2}-\frac{1}{p}+\frac{\delta}{3}}}}
		\lesssim \| |\nabla|^{-\frac{2-\sigma}{2}} \alpha \|_{H^{2\lambda+ \frac{3}{p}-\delta}} 	
		\lesssim \lt\|  \alpha \rt\|_{\dot{H}^{-\frac{2-\sigma}{2}} \cap H^{ 2\lambda+ \frac{3}{p}-\delta} }\lesssim \|\alpha\|_{X}
	\end{aligned}
	\]
	by using the embedding relation $H^{\frac{3}{p}-\delta} \subset  \dot{H}^{\frac{3}{p}-\delta} \subset {L^{\frac{1}{\frac{1}{2}-\frac{1}{p}+\frac{\delta}{3}}}}$.
	In a similar way, we have
	\bq\label{B}
	\begin{aligned}
		&\|B\|_{H^{s+\frac{\sigma}{2}+6(\frac{1}{2}-\frac{1}{p})}} 
		\lesssim \||\nabla|^{-\frac{2-\sigma}{2}} \mathcal{Q}(\alpha)\|_{H^{s+\frac{\sigma}{2}+2\lambda+6(\frac{1}{2}-\frac{1}{p})}},\\
		&\|B\|_{L^{\frac{1}{\frac{1}{2}-\frac{1}{p}+\frac{\delta}{3}}}} \lesssim \||\nabla|^{-\frac{2-\sigma}{2}} \mathcal{Q}(\alpha)\|_{H^{2\lambda+\frac{3}{p}-\delta}}\lesssim \||\nabla|^{-\frac{2-\sigma}{2}} \mathcal{Q}(\alpha)\|_{H^{s+\frac{\sigma}{2}+2\lambda+6(\frac{1}{2}-\frac{1}{p})}}.
	\end{aligned}
	\eq
	Employing the Plancherel theorem twice, we show for any $\tilde{s}\ge 0$  that 
	\[
	\begin{aligned}
		\||\nabla|^{-\frac{2-\sigma}{2}} \mathcal{Q}(\alpha)\|_{H^{\tilde{s}}} 
		&\sim  \lt\| \int_{\R^3} \frac{\lal \eta\ral^{\tilde{s}}m(\eta, \zeta)}{|\eta|^{(2-\sigma)/2}} \widehat{\alpha}(\eta-\zeta) \widehat{\alpha}(\zeta) \,d\zeta\rt\|_{L_{\eta}^2}\\
		&\lesssim  \lt\| \int_{\R^3} \lal \eta\ral^{\tilde{s}}r(\eta-\zeta) \widehat{\alpha}(\eta-\zeta) r(\zeta) \widehat{\alpha}(\zeta) \,d\zeta\rt\|_{L_{\eta}^2}\\ 
		&\quad +\lt\| \int_{\R^3} \lal \eta\ral^{\tilde{s}}|\eta|^{\frac{\sigma}{2}} r(\eta-\zeta)\widehat{\alpha}(\eta-\zeta) r(\zeta)\widehat{\alpha}(\zeta) \,d\zeta\rt\|_{L_{\eta}^2}\\
		&\lesssim \|  (r(\nabla)\alpha) (\lal \nabla \ral^{\tilde{s}} r(\nabla) \alpha) \|_{L^2}
		+ \|  (\lal \nabla\ral^{\tilde{s}} r(\nabla)\alpha) (|\nabla|^{\frac{\sigma}{2}}r(\nabla) \alpha) \|_{L^2}\\
		&\quad + \|  (r(\nabla)\alpha) (\lal \nabla \ral^{\tilde{s}}|\nabla|^{\frac{\sigma}{2}}r(\nabla) \alpha) \|_{L^2}
	\end{aligned}
	\]
	due to the symmetry.
	By the H\"older inequality, \eqref{est_r}, and the embedding $\dot{H}^{\frac{3}{p}} \subset L^{\frac{1}{\frac{1}{2}-\frac{1}{p}}}$, we arrive at
	\bq\label{Q_B}
	\begin{aligned}
		\||\nabla|^{-\frac{2-\sigma}{2}} \mathcal{Q}(\alpha)\|_{H^{\tilde{s}}} 	&\lesssim \|\alpha\|_{L^p} \|\alpha\|_{W^{\tilde{s}, \frac{1}{\frac{1}{2}-\frac{1}{p}}}}
		+\|\alpha\|_{W^{\tilde{s}, \frac{1}{\frac{1}{2}-\frac{1}{p}}}} \||\nabla|^{\frac{\sigma}{2}}\alpha\|_{L^p}
		+\|\alpha\|_{L^p} \||\nabla|^{\frac{\sigma}{2}}\alpha\|_{W^{\tilde{s},\frac{1}{\frac{1}{2}-\frac{1}{p}}}}\\
		&\lesssim \|\alpha\|_{W^{\frac{\sigma}{2}, p}} \|\alpha\|_{H^{\tilde{s}+\frac{3}{p}+\frac{\sigma}{2}}}
	\end{aligned}
	\eq
	for any $\tilde{s}\ge 0$. Putting \eqref{Q_B} into \eqref{B} with $\tilde{s}=s+\sigma/2+2\lambda+6(1/2-1/p)$,  we deduce 
	\[
	\begin{aligned}
		&\| A\|_{H^{s+\frac{\sigma}{2}+6(\frac{1}{2}-\frac{1}{p})}} \|B\|_{L^{\frac{1}{\frac{1}{2}-\frac{1}{p}+\frac{\delta}{3}}}} +\|A\|_{L^{\frac{1}{\frac{1}{2}-\frac{1}{p}+\frac{\delta}{3}}}}\|B\|_{H^{s+\frac{\sigma}{2}+6(\frac{1}{2}-\frac{1}{p})}}  \lesssim \|\alpha\|_{W^{s, p}} \|\alpha\|_{X}^2  
		\lesssim (1+\tau)^{-\beta_{\sigma}}\|\alpha\|_{X}^3
	\end{aligned}
	\]
for
	\bq\label{s_condition1}
s>2\lambda+3 -\frac{3}{p}+\sigma.
\eq
Consequently, we have
	\[
	\begin{aligned}
		\| h(t) \|_{W^{s, p}} 
		&\lesssim \|\alpha\|_{X}^3 \lt( \int_0^{\frac{t}{2}} +\int_{\frac{t}{2}}^t \rt)  (1+t-\tau)^{-\beta_{\sigma}} (1+\tau)^{-\beta_{\sigma}} \,d\tau \lesssim \lt(1+\frac{t}{2} \rt)^{-\beta_{\sigma}}\|\alpha\|_{X}^3
		\lesssim (1+t)^{-\beta_{\sigma}}\|\alpha\|_{X}^3. 
	\end{aligned}
	\]
	
	$\bullet$ Estimates for $g(t)$. We recall the expression for $g$ as follows:
	\[
	g= i\mathcal{F}_{\xi}^{-1} \lt( \int_{\R^3} \frac{m(\xi, \eta)}{\Phi(\xi, \eta)} \widehat{\alpha}(\xi-\eta)\widehat{\alpha}(\eta) \,d\eta\rt).
	\]
	When $|\xi|\leq 1$, we apply the Hardy--Littlewood--Sobolev embedding theorem, the Plancherel theorem and Lemma \ref{GNT_estimate_product} to obtain
	\[
	\begin{aligned}
		\|g\|_{W^{s, p}} 
		&\lesssim \lt\| \lal\xi \ral^s |\xi|^{3(\frac{1}{2}-\frac{1}{p})}  \int_{\R^3} \frac{m(\xi, \eta)}{\Phi(\xi, \eta)} \widehat{\alpha}(\xi-\eta)\widehat{\alpha}(\eta) \, d\eta \rt\|_{L_{\xi}^2}\\
		&\lesssim  \lt\| | \xi|^{\frac{2-\sigma}{2}} \int_{\R^3} \mathfrak{M}(\xi, \eta) \widehat{A}(\xi-\eta)\widehat{A}(\eta) \, d\eta\rt\|_{L_{\xi}^2}\\
		& \lesssim \||\nabla|^{\frac{2-\sigma}{2}}A\|_{L^p} \|A\|_{L^{\frac{1}{\frac{1}{2}-\frac{1}{p}+\frac{\delta}{3}}}}\\
		&\lesssim \| \alpha\|_{W^{2\lambda, p}} \lt\|  \alpha \rt\|_{\dot{H}^{-\frac{2-\sigma}{2}} \cap H^{ 2\lambda+ \frac{3}{p}-\delta} }\\
		&\lesssim \| \alpha\|_{W^{s, p}} \|\alpha\|_{X},
	\end{aligned}
	\]
	where $A$ is given as in \eqref{set_A}.
	Here, for the third and fourth inequalities above, we used 
	\[
	|\xi|^{\frac{2-\sigma}{2}} \le |\xi-\eta|^{\frac{2-\sigma}{2}}+|\eta|^{\frac{2-\sigma}{2}} 
	\]
	and the embedding property  $\dot{H}^{\frac{3}{p}-\delta} \subset {L^{\frac{1}{\frac{1}{2}-\frac{1}{p}+\frac{\delta}{3}}}}$, respectively.
	
	For $|\xi|\ge 1$, we use Lemma \ref{GNT_estimate_product} with $k=3/2-\delta$,
	\[
	\frac{1}{l_1}=\nu \lt(\frac{1}{2}-\frac{1}{p}+\frac{2-\sigma}{6}\rt)+\frac{2\delta}{3}, \quad \mbox{and} \quad \frac{1}{l_2} = \frac{1}{2}-\nu\lt(\frac{1}{2}-\frac{1}{p}+\frac{2-\sigma}{6}\rt)-\frac{\delta}{3},
	\]
	where $\nu=3/4-\delta$. Here, $\delta>0$ is chosen sufficiently small.
	Then, we get
	\[
	\begin{aligned}
		\|g\|_{W^{s, p}} 
		&\lesssim \lt\| |\xi|^{s+3(\frac{1}{2}-\frac{1}{p})} \int_{\R^3} \frac{m(\xi, \eta)}{\Phi(\xi, \eta)} \widehat{\alpha}(\xi-\eta)\widehat{\alpha}(\eta) \, d\eta\rt\|_{L_{\xi}^2}\\
		&\lesssim  \lt\| | \xi|^{s+3(\frac{1}{2}-\frac{1}{p})} \int_{\R^3} \mathfrak{M}(\xi, \eta) \widehat{A}(\xi-\eta)\widehat{A}(\eta) \, d\eta\rt\|_{L_{\xi}^2}  + \lt\| | \xi|^{s+3(\frac{1}{2}-\frac{1}{p})+\frac{\sigma}{2}}\int_{\R^3} \mathfrak{M}(\xi, \eta) \widehat{A}(\xi-\eta)\widehat{A}(\eta) \, d\eta\rt\|_{L_{\xi}^2}\\
		& \lesssim \||\nabla|^{s+3(\frac{1}{2}-\frac{1}{p})}A\|_{L^{l_1}} \|A\|_{L^{l_2}} +\||\nabla|^{s+3(\frac{1}{2}-\frac{1}{p})+\frac{\sigma}{2}}A\|_{L^{l_1}} \|A\|_{L^{l_2}}
	\end{aligned}
	\]
	due to the symmetry. We further estimate
	\[
	\begin{aligned}
		\||\nabla|^{s+3(\frac{1}{2}-\frac{1}{p})+\frac{\sigma}{2}}A\|_{L^{l_1}}
		&=\| |\nabla|^{s-1+\sigma+3(\frac{1}{2}-\frac{1}{p})}\lal \nabla\ral^{2\lambda} \alpha\|_{L^{l_1}} \\
		&\lesssim \|\lal \nabla \ral^{2\lambda} \alpha\|_{W^{s-2\lambda, p}}^{1-\nu} \| \lal \nabla\ral^{2\lambda}\alpha\|_{H^{s_1-2\lambda}}^{\nu}	\\
		&\lesssim \|\alpha\|_{W^{s, p}}^{1-\nu} \| \alpha\|_{H^{2s}}^{\nu}	\\
		&\lesssim \|\alpha\|_{W^{s, p}}^{1-\nu} \| \alpha\|_{X}^{\nu}	
	\end{aligned}
	\]
	for
		\bq \label{s_condition2}
	s_1= s -\frac{2-\sigma}{2}+\frac{1}{\nu} \lt(2\lambda -2\delta +\sigma+\frac{1}{2}\rt)<2s.
	\eq
	Analogously, we deduce
	\[
	\| |\nabla|^{s+3(\frac{1}{2}-\frac{1}{p})} A\|_{L^{l_1}} \lesssim \|\alpha\|_{W^{s, p}}^{1-\nu}\|\alpha\|_X^{\nu}.
	\]
	On the other hand, we use the Bernstein inequality to get
	\[
	\begin{aligned}
		\|P_{<1}A\|_{L^{l_2}} 
		&\lesssim \sum_{N<1} N^{-\frac{2-\sigma}{2}}\|P_N \alpha\|_{L^{l_2}} \\
		&\lesssim \sum_{N<1}N^{-\frac{2-\sigma}{2}} N^{3(\frac{1}{l_3}-\frac{1}{l_2})} \|P_N \alpha\|_{L^{l_3}} \\
		&\lesssim \sum_{N<1}N^{\delta} (N^{-\frac{2-\sigma}{2}} \|P_N \alpha\|_{L^2})^{1-\nu} \|P_N \alpha\|_{L^{p}}^{\nu}\\
		&\lesssim \|\alpha\|_X^{1-\nu} \|\alpha\|_{W^{s, p}}^{\nu}, 
	\end{aligned}
	\]
	where
	\[
	\frac{1}{l_3}=\frac{1-\nu}{2}+\frac{\nu}{p}. 
	\]
We also observe that
	\[
	\begin{aligned}
		\|P_{\ge1}A\|_{L^{l_2}} 
		&\lesssim \sum_{N\ge 1}N^{-\frac{2-\sigma}{2}} \|P_N \lal \nabla\ral^{2\lambda}\alpha\|_{L^{l_2}} \lesssim \|\lal \nabla \ral^{2\lambda} \alpha\|_{L^{l_2}} \lesssim  \|\lal \nabla\ral^{2\lambda}\alpha\|_{L^2}^{1-\nu} \| \lal \nabla\ral^{2\lambda}\alpha\|_{W^{s_2, p}}^{\nu} \lesssim  \|\alpha\|_{X}^{1-\nu} \| \alpha\|_{W^{s, p}}^{\nu}.
	\end{aligned}
	\]
	Here,  
	\[
	s_2 = \frac{2-\sigma}{2} +\frac{\delta}{\nu}<s-2\lambda
	\]
	whose inequality holds true under condition \eqref{s_condition1}.
	Consequently, we have
	\[
	\|g\|_{W^{s, p}}\lesssim \|\alpha\|_X \|\alpha\|_{W^{s, p}}\lesssim (1+t)^{-\beta_{\sigma}}\|\alpha\|_X^2
	\]	 
	provided that conditions \eqref{s_condition1} and \eqref{s_condition2} are satisfied with $\lambda=7/4+3\sigma/2$. 
	
	$\bullet$ Estimates for $g(0)$. 
	Let $A$ be defined as in \eqref{set_A}, and we use Lemma \ref{GNT_estimate_product} to estimate
	\[
	\begin{aligned}
		\|g(0)\|_{W^{s+6(\frac{1}{2}-\frac{1}{p}), p'}}
		&\sim \lt\| \mathcal{F}^{-1} \lt(  \lal\xi \ral^{s+6(\frac{1}{2}-\frac{1}{p})}(D_1^{\frac{\sigma}{2}}+D_2^{\frac{\sigma}{2}}) \int_{\R^3} \mathfrak{M} \widehat{A}(\xi-\eta)\widehat{A}(\eta)\, d\eta \rt)\rt\|_{L^{p'}}\\
		&\lesssim \|A\|_{H^{s+6(\frac{1}{2}-\frac{1}{p})+\frac{\sigma}{2}}} \|A\|_{L^{\frac{1}{\frac{1}{2}-\frac{1}{p}+\frac{\delta}{3}}}}\\
		&\lesssim \|\alpha\|_{\dot{H}^{-\frac{2-\sigma}{2}}\cap H^{s+2\lambda+6(\frac{1}{2}-\frac{1}{p})+\frac{\sigma}{2}}} \|\alpha\|_{\dot{H}^{-\frac{2-\sigma}{2}}\cap H^{2\lambda+\frac{3}{p}-\delta}}\cr
		&\lesssim \|\alpha\|_{X}^2,
	\end{aligned}
	\]
	where $D_1^{\frac{\sigma}{2}}=(1+|\xi|^{\sigma})^{1/2}$ and $D_2^{\frac{\sigma}{2}}=|\xi|^{\frac{\sigma}{2}}$. 
	
	Combining all of the above estimates concludes the desired result.
\end{proof}

%
%
%
%
%
%
%
%
%

\section{$M_{\xi, \eta}^{k, \infty} \cap M_{\eta, \xi}^{k, \infty}$ estimates of the kernel $\mathfrak{M}$}\label{sec_4}
Throughout this section, our analysis revolves around a phase function denoted as $\Phi(\xi, \eta):= p(|\xi|)-p(|\xi-\eta|)-p(|\eta|)$, where 
$p(r)=(r^2+r^{2-\sigma})^{1/2}$ for $0<  \sigma<2$.
We carefully investigate the charateristics of the phase function $\Phi$ and demonstrate that the kernel $\mathfrak{M}_{r,l}$ belongs to the function spaces $M_{\xi, \eta}^{k, \infty}$ and $M_{\eta, \xi}^{k, \infty}$. 

%
%
%
%
%
%
%
%
%

\subsection{Estimates of phase function $\Phi$}

To facilitate the analysis, we introduce the following lemma, a useful tool for obtaining lower bounds of $\Phi$.
\begin{lemma}\label{Lemma_low_Phi}
	Let $0< \sigma<2$, $a=b+c \in \mathbb{R}^3$, and $|c| \leq \min \{|a|, |b|\}$.
	Then, we have 
	\bq\label{low_Phi_1}
	|p(a)-p(b)-p(c)| \gtrsim 
	\frac{|c|}{1+(|a||c|)^{\sigma}} +|c| (1-\cos [a, c]+1-\cos[a, b])
	\eq
	for $|b|\ge 1$.
	Moreover, we obtain
	\bq\label{low_Phi_2}
	|p(a)-p(b)-p(c)| \gtrsim \frac{|c|^{(2-\sigma)/2}}{(1+|b|^{\sigma})^{1/2}}
	\eq
	for $|c|<1$. Here, $[a, b]$ represents the angle between two vectors $a$ and $b$.
\end{lemma}
\begin{proof}
Note that the case $|a|\leq |b|$ implies $p(a)\leq p(b)$. 
	Since $p(r)$ is a non-negative and increasing function (e.g. see \eqref{deriv_P_first}),
 it follows that
	\[
	|p(a)-p(b)-p(c)| \ge p(c) = (|c|^2 +|c|^{2-\sigma})^{1/2} 
	\]
	which implies the desired inequalities.

	We now focus on the case $|a| > |b|$. Consider the function
\[
	Q(r)=\frac{p(r)}{r}=\lt(\frac{1+r^{\sigma}}{r^{\sigma}}\rt)^{1/2}.
\]
	We observe that $Q$ is a decreasing and non-negative function with properties such as
	\[
	Q(r)>1, \qquad Q'(r) = - \frac{\sigma}{2r^{\sigma/2+1} (1+r^{\sigma})^{1/2}} \quad \text{for any}\quad r\ge0,
	\] 
	\[
	Q'(1) = -\frac{\sigma}{2\sqrt{2}} ,\quad \text{and} \quad Q'(r)=-\frac{\sigma}{2r^{\sigma+1}} \quad \text{as} \quad r\rightarrow \infty,
	\]
where $Q'(r)$ is an increasing function.
	Utilizing the relation $|a|=|b| \cos[a, b]+|c| \cos[a, c]$, 
	we write
	\bq\label{represent_Phi}
	\begin{aligned}
		p(b)+p(c)-p(a)
		&=(|b|+|c|-|a|)Q(a)+|b|(Q(b)-Q(a))+|c|(Q(c)-Q(a))\\
		&= \lt\{|b|(1-\cos[a, b])+|c|(1-\cos[a, c]) \rt\} Q(a)\\
		& \quad +|b|(Q(b)-Q(a))+|c|(Q(c)-Q(a)).
	\end{aligned}
	\eq
	It is noteworthy that each term on the right-hand side above is positive since $Q$ is a decreasing function.
	
	To prove \eqref{low_Phi_1}, we first assume $\cos[a, c]\leq 9/10$.
	Then, \eqref{represent_Phi} implies  
	\bq\label{low_phi_1}
	\begin{aligned}
		p(b)+p(c)-p(a) 
		&= |b| (Q(b)-\cos[a, b]Q(a)) +|c|(Q(c)-\cos[a, c]Q(a)) \\
		&\ge|c|(Q(c)-\cos[a, c]Q(a))\\
		&\ge \frac{|c|Q(c)}{10}+\frac{9|c|}{10}(Q(c)-Q(a))\\
		&\ge \frac{p(c)}{10}
	\end{aligned}
	\eq
	due to the decreasing property of $Q(r)$. Thus, it is clear to deduce \eqref{low_Phi_1} when $\cos[a, c]\leq 9/10$.
	 
	Meanwhile, it is observed that $|c|\leq |b|$ implies $\cos [a, c]\leq \cos[a, b]$, leading to $|a| \ge 2 |c|\cos[a, c]$. When $\cos[a, c]> 9/10$, it follows that
	$$|a| \ge \frac{4|c|}{3}.$$ 
	For $1 < |a|<\delta^{-1}$ for $\delta>0$ small enough, we have 
	\[
	|a|\ge \frac{|a|}{4}+|c| \ge \frac{1}{4}+|c|. 
	\]
	Consequently,
	\[
	Q(c)-Q(a)=-\int_{|c|}^{|a|} Q'(r) \,dr \ge -Q'(\delta^{-1}) \int_{|c|}^{|a|} \,dr \gtrsim \delta^{\sigma+1} \int_{|c|}^{|c|+\frac{1}{4}} \,dr \gtrsim 1,
	\]
	yielding
	\bq\label{b>1_1}
	|c|(Q(c)-Q(a)) \gtrsim |c| \gtrsim \frac{|c|}{1+|a|^{\sigma}|c|^{\sigma}}.
	\eq
For $|a| \ge\delta^{-1}$, this can be simplified to $|c|> \delta^{-1}$ since $-Q'(r)$ is decreasing. Subsequently,
	\[
	Q(c)-Q(a)=-\int_{|c|}^{|a|} Q'(r)\,dr  \sim \int_{|c|}^{|a|} \frac{\sigma}{r^{\sigma+1}}\,dr 
	=\frac{|a|^{\sigma}-|c|^{\sigma}}{|a|^{\sigma}|c|^{\sigma}} \gtrsim  \frac{|a|^{\sigma}}{1+|a|^{\sigma}|c|^{\sigma}}
	\]
	due to $|a| \ge 4|c|/3$ and $|a|> 1$. This yields
	\bq\label{b>1_2}
	|c|(Q(c)-Q(a)) \gtrsim \frac{|c|}{1+|a|^{\sigma}|c|^{\sigma}}.
	\eq
By combining \eqref{b>1_1} and \eqref{b>1_2} along with \eqref{represent_Phi}, we conclude \eqref{low_Phi_1} for $\cos [a,c]>9/10$.
	
Next, we aim to demonstrate \eqref{low_Phi_2}. Similar to the above, it suffices to deal with the case $|b|<|a|$. By utilizing \eqref{low_phi_1}, the case $\cos [a, b] \leq 9/10$ can be easily handled as follows:
	\[
	p(b)+p(c)-p(a) \ge \frac{p(c)}{10}  \gtrsim \frac{|c|^{(2-\sigma)/2}}{(1+|b|^{\sigma})^{1/2}}. 
	\]
When $\cos[a, b] >9/10$, we find
	\bq\label{c<1_1}
	\begin{aligned}
		p(b)+p(c)-p(a) 
		&\ge |b|Q(b) +|c|Q(c)-|a|Q(b)\\
		&=(|b|+|c|-|a|)Q(b)	+|c|(Q(c)-Q(b))\\
		& \ge |c|(Q(c)-Q(b))
	\end{aligned}
	\eq
	due to the monotonicity of $Q(r)$ and $|a|\leq |b|+|c|$.
	If $|c|^{\sigma}<3|b|^{\sigma}/4$, direct computation gives  
	\bq\label{c<1_2}
	\begin{aligned}
		Q(c)-Q(b)
		&=\frac{(1+|c|^{\sigma})^{1/2}}{|c|^{\sigma/2}}-\frac{(1+|b|^{\sigma})^{1/2}}{|b|^{\sigma/2}} \\
		&= \frac{|b|^{\sigma/2}(1+|c|^{\sigma})^{1/2}-|c|^{\sigma/2}(1+|b|^{\sigma})^{1/2}}{|b|^{\sigma/2}|c|^{\sigma/2}}\\
		&= \frac{|b|^{\sigma}-|c|^{\sigma}}{|b|^{\sigma/2}|c|^{\sigma/2}\lt(|b|^{\sigma/2}(1+|c|^{\sigma})^{1/2}+|c|^{\sigma/2}(1+|b|^{\sigma})^{1/2}\rt)}\\
		&> \frac{|b|^{\sigma/2}}{4|c|^{\sigma/2}\lt(|b|^{\sigma/2}(1+|c|^{\sigma})^{1/2}+|c|^{\sigma/2}(1+|b|^{\sigma})^{1/2}\rt)}
	\end{aligned}
	\eq
	which implies
	\[
	|c| (Q(c)-Q(b))  
	\gtrsim \frac{|c|^{(2-\sigma)/2}}{ (1+|b|^{\sigma})^{1/2}}
	\]
	whenever $|c|<1$.
	Thus, by combining \eqref{c<1_1} and \eqref{c<1_2}, we derive \eqref{low_Phi_2}. 
	
Moreover, for $|c|^{\sigma}\ge 3|b|^{\sigma}/4$,  we have
\[
|a| \leq |b|+|c|\leq \Big(1+\Big(\frac{4}{3}\Big)^{\frac{1}{\sigma}} \Big)|c| =: \nu |c|.
\]
This leads to
	\[
	\begin{aligned}
		Q(b)-Q(a)= - \int_{|b|}^{|a|} Q'(r)\,dr \gtrsim -Q'\lt(\nu |c|\rt) \int_{|b|}^{\frac{4|b|}{3}}\, dr = -\frac{|b|}{3}Q'\lt( \nu |c|\rt) 
	\end{aligned}
	\]
	since $|a|\ge 4|b|/3$ when $\cos[a, b] > 9/10$. Consequently,
	\[
	|b|(Q(b)-Q(a)) =-\frac{|b|^2}{3}Q'\lt( \nu |c| \rt) =\frac{\sigma|b|^2}{6  \nu^{1+\sigma/2}|c|^{1+\sigma/2} (\nu^{\sigma}|c|^{\sigma}+1)^{1/2}} \gtrsim \frac{|c|^{(2-\sigma)/2}}{(1+|b|^{\sigma})^{1/2}}
	\]
	for $|b|\ge |c|$. By combining this with \eqref{represent_Phi}, we conclude \eqref{low_Phi_2} as desired.
\end{proof}

The following lemma presents useful estimates involving the first-order and second-order derivatives of $\Phi$.
\begin{lemma}\label{Lemma_deriv_Phi}
	Let $0< \sigma<2$. The following estimates hold;
	\[
	\begin{aligned}
		|\nabla_{\xi} \Phi| &\lesssim 
		\begin{cases}
			  |\eta|+2 \sin \frac{[\xi, \xi-\eta]}{2} \hspace{144pt} \text{if } \displaystyle \min\{|\xi|, |\xi-\eta|\}\ge 1\\
			  \frac{1}{\min\{|\xi|, |\xi-\eta| \}^{\sigma/2}} + 2\max\lt\{1, \frac{1}{|\xi-\eta|^{\sigma/2}}\rt\} \sin \frac{[\xi, \xi-\eta]}{2}  \quad  \text{if } \min\{|\xi|, |\xi-\eta|\}< 1,
		\end{cases}\\
		|\nabla_{\eta} \Phi| &\lesssim 
		\begin{cases}
			|\xi|+2 \sin \frac{[\eta, \xi-\eta]}{2} \hspace{144pt}  \text{if } \min\{|\eta|, |\xi-\eta|\}\ge 1\\
			\frac{1}{\min\{|\eta|, |\xi-\eta| \}^{\sigma/2}} + 2\max\lt\{1, \frac{1}{|\xi-\eta|^{\sigma/2}}\rt\} \sin \frac{[\eta, \xi-\eta]}{2} \quad \text{if } \min\{|\eta|, |\xi-\eta|\}< 1,
		\end{cases}\\
		|\Delta_{\xi} \Phi| &\lesssim 
		\begin{cases}
			|\eta| \hspace{85pt} \text{if } \min\{|\xi|, |\xi-\eta|\}\ge 1\\
			\frac{1}{\min\{|\xi|, |\xi-\eta|\}^{1+\sigma/2}}  
			\qquad \text{if } \min\{|\xi|, |\xi-\eta|\}< 1,\\
		\end{cases}\\
		|\Delta_{\eta} \Phi| &\lesssim 
		\begin{cases}
			|\xi|\hspace{85pt} \text{if } \min\{|\eta|, |\xi-\eta|\}\ge 1\\
			\frac{1}{\min\{|\eta|, |\xi-\eta|\}^{1+\sigma/2}}  
			\qquad \text{if } \min\{|\eta|, |\xi-\eta|\}< 1.\\
		\end{cases}
	\end{aligned}
	\]
	In addition, we have 
	\bq\label{replacement_Phi}
	\begin{aligned}
		&|\nabla_{\xi} \Phi| \lesssim 	\frac{|\eta|}{\min\{|\xi|, |\xi-\eta|\}^{1+\sigma/2}} + 2\max\lt\{1, \frac{1}{|\xi-\eta|^{\sigma/2}}\rt\} \sin \frac{[\xi, \xi-\eta]}{2}, \\
		&|\Delta_{\xi} \Phi| \lesssim \frac{|\eta|}{\min\{|\xi|, |\xi-\eta|\}^{2+\sigma/2}}+\frac{|\eta|^{1+\sigma/2}}{\min\{|\xi|, |\xi-\eta|\}^{2+\sigma}}
	\end{aligned}
	\eq
	for $\min\{|\xi|, |\xi-\eta|\}<1$.
\end{lemma}
\begin{proof}
	We observe that 
	\[
	\begin{aligned}
		|\nabla_{\xi} \Phi |
		&= \lt|p'(\xi) \frac{\xi}{|\xi|}-p'(\xi-\eta) \frac{\xi-\eta}{|\xi-\eta|}\rt|\\ 
		&= \lt|p'(\xi) \frac{\xi}{|\xi|}-p'(\xi-\eta)\frac{\xi}{|\xi|} + p'(\xi-\eta) \frac{\xi}{|\xi|}-p'(\xi-\eta) \frac{\xi-\eta}{|\xi-\eta|}\rt|	\\
		&\leq \lt|p'(\xi) -p'(\xi-\eta)\rt| + | p'(\xi-\eta)| \lt| \frac{\xi}{|\xi|}- \frac{\xi-\eta}{|\xi-\eta|}\rt|.
	\end{aligned}
	\]
	If $\min\{ |\xi|, |\xi-\eta|\}\ge1$, it follows that
	\[
	\begin{aligned}
		\lt|p'(\xi) -p'(\xi-\eta)\rt| 
		\leq \int_{\min\{|\xi|, |\xi-\eta|\}}^{\max\{|\xi|, |\xi-\eta|\}} |p''(r)| \,dr \lesssim \int_{\min\{|\xi|, |\xi-\eta|\}}^{\max\{|\xi|, |\xi-\eta|\}} \frac{1}{r}\,dr 
		 \lesssim \frac{||\xi|-|\xi-\eta||}{\min\{|\xi|, |\xi-\eta|\}} \leq |\eta|
	\end{aligned}
	\]
	by utilizing \eqref{behavior_P1}. 
	When $\min\{|\xi|, |\xi-\eta| \}<1$, we get 
\[
\begin{aligned}
	|p'(\xi)-p'(\xi-\eta)| 
	&\leq  \lt|p'(\min\{|\xi|, |\xi-\eta|\})\rt|+\lt|p'(\max\{|\xi|, |\xi-\eta|\})\rt|\\
	&\lesssim \frac{1}{\min\{|\xi|, |\xi-\eta|\}^{\sigma/2}} +\max\lt\{ 1, \frac{1}{\max\{|\xi|, |\xi-\eta|\}^{\sigma/2}}\rt\}\\
	&\lesssim \frac{1}{\min \{|\xi|, |\xi-\eta|\}^{\sigma/2}}
\end{aligned}
\]
by using \eqref{behavior_P1} and \eqref{behavior_P2}. Additionally, we have
	\bq\label{note_8}
	|p'(\xi-\eta)| \lesssim \max\lt\{1, \frac{1}{|\xi-\eta|^{\sigma/2}}\rt\}. 
	\eq
Since 
	\bq\label{note_9}
	\lt|\frac{\xi}{|\xi|}-\frac{\xi-\eta}{|\xi-\eta|}\rt| =2 \sin \frac{[\xi, \xi-\eta]}{2},
	\eq
we deduce
	\[
	|\nabla_{\xi} \Phi| \lesssim 
	\begin{cases}
		|\eta|+2 \sin \frac{[\xi, \xi-\eta]}{2} \hspace{138pt} \text{if } \min\{|\xi|, |\xi-\eta|\}\ge 1,\\
		\frac{1}{\min\{|\xi|, |\xi-\eta| \}^{\sigma/2}} + 2 \max\{1, \frac{1}{|\xi-\eta|^{\sigma/2}}\} \sin \frac{[\xi, \xi-\eta]}{2}
		\quad \text{if } \min\{|\xi|, |\xi-\eta|\}<1.\\
	\end{cases}
	\]
	
	Next, we observe that
	\bq\label{Delta_Phi}
	\begin{aligned}
		|\Delta_{\xi}\Phi|
		&= |\Delta_{\xi} p(\xi)-\Delta_{\xi}p(\xi-\eta)|\\
		&\leq |p''(\xi)-p''(\xi-\eta)|+2\lt|\frac{p'(\xi)}{|\xi|}-\frac{p'(\xi-\eta)}{|\xi-\eta|} \rt|	\\
		&\lesssim |p''(\xi)-p''(\xi-\eta)|+\frac{1}{|\xi|} |p'(\xi)-p'(\xi-\eta)| +\lt| \frac{1}{|\xi|}-\frac{1}{|\xi-\eta|}\rt| |p'(\xi-\eta)|\\
		&=:I+II+III.
	\end{aligned}
	\eq
	If $\min\{|\xi|, |\xi-\eta|\} \ge 1$, a similar argument as before gives
	\[
	I \lesssim  \int^{\max\{|\xi|, |\xi-\eta|\}}_{\min \{|\xi|, |\xi-\eta|\}} \frac{1}{r^2}\, dr  \lesssim \frac{||\xi|-|\xi-\eta||}{|\xi| |\xi-\eta|} \lesssim \frac{|\eta|}{\min\{|\xi|, |\xi-\eta|\}^2}, \quad II\lesssim \frac{|\eta|}{|\xi|} \lesssim |\eta|,
	\]
and
\[
III\lesssim \lt| \frac{1}{|\xi|}-\frac{1}{|\xi-\eta|}\rt| |p'(\xi-\eta)|\leq \frac{||\xi|-|\xi-\eta||}{|\xi||\xi-\eta|} \leq |\eta|
\]
leading to 	the desired result for $\min\{|\xi|, |\xi-\eta|\}\ge1$.

We now consider the case $\min\{|\xi|, |\xi-\eta|\}<1$.  It is clear to get
	\[
	\begin{aligned}
		I&\lesssim \lt| p''(\min\{|\xi|, |\xi-\eta|\})\rt| +\lt|p''(\max\{|\xi|, |\xi-\eta|\}) \rt|\\
		&\lesssim \frac{1}{\min\{|\xi|, |\xi-\eta|\}^{1+\sigma/2} } + \frac{1}{\max\{|\xi|, |\xi-\eta|\}^{1+\sigma/2}}\\
		&\lesssim \frac{1}{\min\{|\xi|, |\xi-\eta|\}^{1+\sigma/2} },\\
		 II&\lesssim \frac{1}{|\xi|\min\{|\xi|, |\xi-\eta|\}^{\sigma/2} } \lesssim \frac{1}{ \min\{|\xi|, |\xi-\eta|\}^{1+\sigma/2} },
	\end{aligned}
	\]
	and
\[
		III \lesssim \frac{1}{\min\{|\xi|, |\xi-\eta|\} } |p'(\xi-\eta)| \lesssim \frac{1}{\min\{|\xi|, |\xi-\eta|\} } \max\{1, \frac{1}{|\xi-\eta|^{\sigma/2}}\}  \lesssim \frac{1}{\min\{|\xi|,|\xi-\eta|\}^{1+\sigma/2}}.
	\]
	Hence, we arrive at the following upper-bound:
	\[
	|\Delta_{\xi} \Phi| \lesssim 
	\frac{1}{\min\{|\xi|, |\xi-\eta|\}^{1+\sigma/2}} 
	\]
	for $\min\{|\xi|, |\xi-\eta|\}<1$.
	
	In order to derive \eqref{replacement_Phi}, if $\max\{|\xi|, |\xi-\eta|\}<1$, we find that
	\[
	\begin{aligned}
		\lt|p'(\xi) -p'(\xi-\eta)\rt| 
		&\lesssim \int_{\min\{|\xi|, |\xi-\eta|\}}^{\max\{|\xi|, |\xi-\eta|\}} \frac{1}{r^{1+\sigma/2}} \,dr  \leq \frac{||\xi|-|\xi-\eta||}{\min\{|\xi|, |\xi-\eta|\}^{1+\sigma/2}} \leq \frac{|\eta|}{\min\{|\xi|, |\xi-\eta|\}^{1+\sigma/2}}.
	\end{aligned}
	\]
	If $\max\{|\xi|, |\xi-\eta|\}\ge 1$, then
	\bq\label{note_1}
	\begin{aligned}
		\lt|p'(\xi) -p'(\xi-\eta)\rt| 
		&\lesssim \int_{\min\{|\xi|, |\xi-\eta|\}}^{1} \frac{1}{r^{1+\sigma/2}} \,dr +\int_{1}^{\max\{|\xi|, |\xi-\eta|\}} \frac{1}{r^{1+\sigma}}\, dr\\
		&\lesssim \frac{1-\min\{|\xi|, |\xi-\eta|\}}{\min\{|\xi|, |\xi-\eta|\}^{1+\sigma/2}} +\max\{|\xi|, |\xi-\eta|\}-1\\
		&\lesssim \frac{||\xi|-|\xi-\eta||}{\min\{|\xi|, |\xi-\eta|\}^{1+\sigma/2}}\\
		&\lesssim \frac{|\eta|}{\min\{|\xi|, |\xi-\eta|\}^{1+\sigma/2}}
	\end{aligned}
	\eq
	whenever $\min\{|\xi|, |\xi-\eta|\} <1$. By using this along with \eqref{note_8} and \eqref{note_9}, we deduce that 
	\[
	|\nabla_{\xi} \Phi| \lesssim 	\frac{|\eta|}{\min\{|\xi|, |\xi-\eta|\}^{1+\sigma/2}} + 2 \max\{1, \frac{1}{|\xi-\eta|^{\sigma/2}}\} \sin \frac{[\xi, \xi-\eta]}{2}
	\]
	for $\min\{|\xi|, |\xi-\eta|\}<1$.
	
	Meanwhile, recalling that 
	\[
	\begin{aligned}
		|\Delta_{\xi}\Phi|\lesssim I+II+III
	\end{aligned}
	\]
	from \eqref{Delta_Phi}, we estimate 
	\[
	I\lesssim \int_{\min\{|\xi|, |\xi-\eta|\}}^{\max\{|\xi|, |\xi-\eta|\}} \frac{1}{r^{2+\sigma/2}}\, dr 
	\lesssim \frac{||\xi|^{1+\sigma/2}-|\xi-\eta|^{1+\sigma/2}|}{|\xi|^{1+\sigma/2}|\xi-\eta|^{1+\sigma/2}}
	\lesssim \frac{|\eta|^{1+\sigma/2}}{\min\{|\xi|, |\xi-\eta|\}^{2+\sigma}}
	\]
	for $\max\{ |\xi|, |\xi-\eta|\} <1$, and
	\[
	\begin{aligned}
		I&\lesssim  \int_{\min\{|\xi|, |\xi-\eta|\}}^1 \frac{1}{r^{2+\sigma/2}} \,dr +\int_1^{\max\{|\xi|, |\xi-\eta|\}} \frac{1}{r^2} \,dr \\
		&\lesssim \frac{1-\min\{|\xi|, |\xi-\eta|\}}{\min\{|\xi|, |\xi-\eta|\}^{2+\sigma/2}} +\max\{|\xi|, |\xi-\eta|\}-1\\
		&\lesssim \frac{|\eta|}{\min\{|\xi|, |\xi-\eta|\}^{2+\sigma/2}}
	\end{aligned}
	\]
	similarly to \eqref{note_1} for $\max\{|\xi|, |\xi-\eta|\}\ge 1$.
Moreover, we find
	\[
	\begin{aligned}
		II& \lesssim \frac{|\eta|}{\min\{|\xi|, |\xi-\eta|\}^{1+\sigma/2}|\xi| } \lesssim \frac{|\eta|}{\min\{|\xi|, |\xi-\eta|\}^{2+\sigma/2} },\\ 
		III&\lesssim \frac{||\xi|-|\xi-\eta||}{|\xi||\xi-\eta|} \max\{1, \frac{1}{|\xi-\eta|^{\sigma/2}}\} \leq \frac{|\eta|}{\min\{|\xi|,|\xi-\eta|\}^{2+\sigma/2}}.
	\end{aligned}
	\]
	Combining these, we conclude that 
	\[
	|\Delta_{\xi}\Phi| \lesssim  \frac{|\eta|}{\min\{|\xi|, |\xi-\eta|\}^{2+\sigma/2}}+\frac{|\eta|^{1+\sigma/2}}{\min\{|\xi|, |\xi-\eta|\}^{2+\sigma}}.
	\]
It can be verified that similar upper bounds for the derivative with respect to $\eta$ can be obtained in a similar manner, thereby completing the proof.
\end{proof}

%
%
%
%
%
%
%
%
%

\subsection{Proof of Proposition \ref{Proposition_m}}
In this subsection, we provide the proof of Proposition \ref{Proposition_m} using Lemmas \ref{Lemma_low_Phi} and  \ref{Lemma_deriv_Phi}, dealing with $\Phi=\Phi_{1,1}$ and $\mathfrak{M}=\mathfrak{M}_{1,1}$.

It is worth noting that a similar argument can be applied to $\mathfrak{M}_{1,2}$ (respectively $\mathfrak{M}_{2,1}$) by interchanging $\xi$ and $\xi-\eta$ (respectively $\eta$), and it can be extended to $\mathfrak{M}_{2,2}$
since it can be easily verified that $\Phi_{2,2}=p(|\xi|)+p(|\xi-\eta|)+p(|\eta|)$ satisfies Lemmas \ref{Lemma_low_Phi} and \ref{Lemma_deriv_Phi}.

To demonstrate $\mathfrak{M}\in M_{\xi, \eta}^{k, \infty} \cap M_{\eta, \xi}^{k, \infty}$, we categorize the proof into three cases:
\begin{itemize}
	\item (Case A) $\{|\xi| \leq \frac{|\eta|}{2}\}$,\\[-3mm]
	\item (Case B) $\{|\eta| \leq \frac{|\xi|}{2}\}$,\\[-3mm]
	\item (Case C) $\{\frac13<\frac{|\xi|}{|\eta|}<3\}$.
\end{itemize}
Essentially, these three cases respectively correspond to situations where $|\xi-\eta|\sim |\eta|$, $|\xi-\eta|\sim |\xi|$, and $|\eta|\sim |\xi|$.

From Lemma \ref{Lemma_low_Phi}, we expect that $\Phi$ can be zero when $\min\{|\xi-\eta|, |\eta|\}=0$, leading to a singularity in $\mathfrak{M}$. 
Thus, it is crucial to investigate situations where $|\xi-\eta|$ or $|\eta|$ are small enough.
Utilizing symmetry, we streamline the proof by focusing on (Case B), where $|\eta|$ is the smallest among the three cases (proofs for the other cases are provided in Appendix \ref{appendix_1}).

Employing a slight abuse of notation, we set $\theta := [\xi, \eta], \gamma:=[\xi, \xi-\eta]$, and $\pi-\beta:=[\eta, \eta-\xi]$.
We observe the fundamental inequalities:
\bq\label{basic1}
\frac{|\xi-\eta|}{\sin \theta}=\frac{|\xi|}{\sin \beta} =\frac{|\eta|}{\sin \gamma}, \quad  1-\cos y =2 \sin^2 (\frac{y}{2})
\eq
for any $0\leq y \leq 2\pi$
and 
\bq\label{basic2}
\frac{x}{2} \leq \sin x \leq x ,\quad \frac{x^2}{4} \leq \sin^2 x \leq x^2 \quad \text{for } 0\leq x\leq \frac{\pi}{2}.
\eq

In (Case B), we note that $|\eta|$ is the smallest value and $|\xi-\eta|\sim |\xi|$. For $\mathfrak{M}\in M_{\eta, \xi}^{k, \infty}$, we further subdivide it into the following cases; 
\begin{itemize}
	\item (B1)  $\min\{|\xi|, |\xi-\eta|\}<1 $,\\[-4mm]
	\item (B2) $\min\{|\xi|, |\xi-\eta|\} \ge 1$.\\[-4mm]
\end{itemize}

In subcase (B1), we get 
\[
|\Phi| \gtrsim \frac{|\eta|^{(2-\sigma)/2}}{\lal \xi \ral^{\sigma/2}}
\] 
by applying \eqref{low_Phi_2} along with $|\xi|\sim |\xi-\eta|$.
By \eqref{replacement_Phi}, we obtain
\[
\begin{aligned}
	|\nabla_{\xi}\Phi|
	& \lesssim \frac{|\eta|}{| \xi|^{1+\sigma/2}}+ \max\{1, \frac{1}{|\xi-\eta|^{\sigma/2}}\} \sin \frac{\gamma}{2} \lesssim
	\frac{|\eta|}{|\xi|^{1+\sigma/2}},\\
	|\Delta_{\xi}\Phi| 
	&\lesssim \frac{|\eta|}{|\xi|^{2+\sigma/2}}+\frac{|\eta|^{1+\sigma/2}}{|\xi|^{2+\sigma}} =\frac{|\eta|}{|\xi|^{2+\sigma/2}} \lt( 1+\frac{|\eta|^{\sigma/2}}{|\xi|^{\sigma/2}}\rt)\lesssim \frac{|\eta|}{{|\xi|}^{2+\sigma/2}}
\end{aligned}
\]
since $|\eta| \leq |\xi| \sim |\xi-\eta|$. In the second inequality of the first line above, we used the fact that 
\bq\label{sin_gamma}
\sin \frac{\gamma}{2} \leq \frac{\sin \gamma}{2}=\frac{|\eta|}{2|\xi|} \sin \beta \leq \frac{|\eta|}{|\xi|}
\eq
due to $\gamma<\pi/2$. Thus, we have
\bq\label{note_6}
\begin{aligned}
	\lt|\nabla_{\xi} (\frac{1}{\Phi})\rt| =	&\lt|-\frac{\nabla_{\xi}\Phi}{\Phi^2}\rt| 
	\lesssim \frac{\lal \xi \ral^{\sigma}}{|\xi|^{1+\sigma/2}|\eta|^{1-\sigma}},\\
	\lt|\Delta_{\xi}(\frac{1}{\Phi}) \rt| 
	= &\lt|-\frac{\Delta_\xi \Phi}{\Phi^2}+\frac{2|\nabla_{\xi} \Phi|^2}{\Phi^3} \rt| 
	\lesssim  \frac{\lal \xi \ral^{\sigma}}{|\xi|^{2+\sigma/2}|\eta|^{1-\sigma}} \lt( 1+\frac{|\eta|^{\sigma/2}\lal\xi\ral^{\sigma/2}}{|\xi|^{\sigma/2}}\rt).
\end{aligned}
\eq
Let us consider 
\bq\label{M_g}
\mathfrak{M}(\xi, \eta)=\frac{|\xi|^{(2-\sigma)/2}|\xi-\eta|^{(2-\sigma)/2}|\eta|^{(2-\sigma)/2}}{\lal \xi-\eta\ral^{2 \lambda} \lal \eta\ral^{2 \lambda}\Phi(\xi, \eta)} = \frac{g(\xi, \eta)}{\Phi(\xi, \eta)},
\eq
where we find
\[
|\mathfrak{M}| \lesssim \frac{|\xi|^{2-\sigma} }{\lal \xi \ral^{2\lambda-\sigma/2} \lal\eta\ral^{2\lambda}}
\]
due to $|\xi-\eta| \sim |\xi|$. We express
\bq\label{Del_M}
\Delta_{\xi}\mathfrak{M} = g \Delta_{\xi}(\frac{1}{\Phi}) +2 \nabla_{\xi}g \cdot \nabla_{\xi}(\frac{1}{\Phi})  +\frac{\Delta_{\xi} g }{\Phi},
\eq
and using \eqref{note_6}, we deduce
\[
\begin{aligned}
	|\Delta_{\xi}\mathfrak{M} |
	\lesssim \frac{1}{|\xi|^{\sigma}\lal \xi\ral^{2\lambda} \lal \eta\ral^{2\lambda}} \lt( \frac{\lal \xi \ral^{3\sigma/2}|\eta|^{\sigma}}{|\xi|^{\sigma}} +\frac{\lal \xi\ral^{\sigma}|\eta|^{\sigma/2}}{|\xi|^{\sigma/2}}+\lal\xi \ral^{\sigma/2}\rt) 
	\lesssim \frac{\lal \xi\ral^{3\sigma/2}}{|\xi|^{\sigma}\lal \xi \ral^{2\lambda} \lal \eta \ral^{2\lambda}}
\end{aligned}
\]
given that $|\eta|\leq |\xi|$. This implies
\[
\begin{aligned}
	\int_{\R^3} |\mathfrak{M} \varphi ( \frac{\xi}{N}) |^2  \,d\xi &\lesssim \int_{\rho\sim N}  \frac{\rho^{6-2\sigma} }{\lal \rho \ral^{4\lambda-\sigma}\lal \eta \ral^{4\lambda} }  \,d\rho 
	\lesssim     \frac{N^{7-2\sigma}}{\lal N \ral^{4\lambda-\sigma}\lal \eta\ral^{4\lambda} },\\
	\int_{\R^3} |\Delta_{\xi}  \mathfrak{M} \varphi(\frac{\xi}{N}) |^2 \,d\xi 
	& \lesssim \int_{\rho \sim N} \frac{\rho^{2-2\sigma}}{ \lal \rho\ral^{4\lambda-3\sigma}\lal \eta \ral^{4\lambda}} \,d\rho 
	\lesssim   \frac{N^{3-2\sigma}}{ \lal N\ral^{4\lambda-3\sigma}\lal \eta \ral^{4\lambda}}  .
\end{aligned}
\]
Consequently, we have
\[
\begin{aligned}
	\|\mathfrak{M} \varphi_N \|_{\dot{H}_{\xi}^{k}} 
	&\lesssim   \lt(  \frac{N^{(7-2\sigma)/2}}{\lal N \ral^{2\lambda-\sigma/2}\lal \eta\ral^{2\lambda} }\rt)^{1-k/2}  \lt(  \frac{N^{(3-2\sigma)/2}}{ \lal N\ral^{2\lambda-3\sigma/2}\lal \eta \ral^{2\lambda} }  \rt)^{k/2} \\
	&= \frac{N^{7/2-\sigma-k}}{\lal N \ral^{2\lambda-\sigma/2-k(\sigma)/2}\lal \eta\ral^{2\lambda} }
	\lesssim  \frac{1}{\lal N \ral^{2\lambda-7/2+\sigma/2-k(\sigma-2)/2} }
\end{aligned}
\]
for $0\leq k\leq 3/2$.
It is clearly summable over $N$ if $\lambda> 7/4-\sigma/4 +k(\sigma-2)/4$, leading to $\mathfrak{M}\in M_{\eta, \xi}^{k, \infty}$ for $0\leq k\leq 3/2$.

In subcase (B2) where $\min\{|\xi|, |\xi-\eta|\} \ge 1$, we get
\[
|\Phi| \gtrsim |\eta|(\lal \xi\ral^{-2\sigma}+\theta^2),
\]
which is evidently obtained by \eqref{low_Phi_1}  
due to the second identity in \eqref{basic1} and the first inequality in \eqref{basic2}. 
We use Lemma \ref{Lemma_deriv_Phi} and \eqref{sin_gamma}
to deduce
\[
|\nabla_{\xi} \Phi| \lesssim |\eta| +\frac{|\eta|}{|\xi|} \lesssim |\eta| \quad \text{and}\quad |\Delta_{\xi} \Phi|\lesssim |\eta|. 
\]
Subsequently, it follows that
\[
\begin{aligned}
	|\nabla_{\xi} (\frac{1}{\Phi})| 
	\lesssim \frac{1}{|\eta|(\lal \xi \ral^{-2\sigma}+\theta^2)^2} \quad \text{and}\quad
	|\Delta_{\xi}(\frac{1}{\Phi})| 
	&\lesssim \frac{1}{|\eta|(\lal \xi \ral^{-2\sigma}+\theta^2)^3}.
\end{aligned}
\]
Given $\mathfrak{M}$ as defined in \eqref{M_g}, we obtain
\[
\begin{aligned}
	|\mathfrak{M}|&\lesssim \frac{|\xi|^{2-\sigma}}{|\eta|^{\sigma/2} \lal \xi\ral^{2\lambda}\lal \eta \ral^{2\lambda}(\lal \xi \ral^{-2\sigma}+\theta^2)} \lesssim\frac{|\xi|^{2-\sigma}}{ \lal \xi\ral^{2\lambda}\lal \eta \ral^{2\lambda}(\lal \xi \ral^{-2\sigma}+\theta^2)} , \\
	|\Delta_{\xi}\mathfrak{M} |
	&\lesssim 
	\frac{|\xi|^{2-\sigma}}{|\eta|^{\sigma/2}\lal \xi\ral^{2\lambda}\lal \eta\ral^{2\lambda}}  \lt( \frac{1}{(\lal \xi \ral^{-2\sigma}+\theta^2)^3} + \frac{1}{|\xi|(\lal \xi \ral^{-2\sigma}+\theta^2)^2} + \frac{1}{|\xi|^2(\lal \xi \ral^{-2\sigma}+\theta^2)}\rt) \\
	&\lesssim 
	\frac{|\xi|^{2-\sigma}}{ \lal \xi\ral^{2\lambda} \lal \eta \ral^{2\lambda}(\lal \xi \ral^{-2\sigma}+\theta^2)^3}
\end{aligned}
\]
due to $\lal \xi \ral^{-2}+\theta^2 \lesssim 1$ when $|\xi| \ge |\eta|\ge 1$. Upon transforming into spherical polar coordinates with $-\eta/|\eta|$ as the north pole, it yields that
\[
\begin{aligned}
	\int_{\R^3} |\mathfrak{M} \varphi ( \frac{\xi}{N}) |^2 \, d\xi 
	& \lesssim \int_{\rho\sim N}  \frac{\rho^{6-2\sigma}}{\lal \rho \ral^{4\lambda}\lal \eta \ral^{4\lambda} }   \lt(\int_{\theta^2 <\lal \rho\ral^{-2\sigma}} \frac{1}{\lal \rho\ral^{-4\sigma}}\sin\theta \,d\theta +\int_{\theta^2 \ge\lal \rho\ral^{-2\sigma}} \frac{1}{\theta^4}\sin\theta \,d\theta \rt)d\rho \\
	&  \lesssim     \frac{N^{7-2\sigma}}{\lal N \ral^{4\lambda-2\sigma}\lal \eta\ral^{4\lambda} }
\end{aligned}
\]
and	
\[
\begin{aligned}
	\int_{\R^3} |\Delta_{\xi}  \mathfrak{M} \varphi(\frac{\xi}{N}) |^2 \,d\xi 
	&  \lesssim \int_{\rho \sim N} \frac{\rho^{6-2\sigma}}{\lal \rho\ral^{4\lambda}\lal \eta \ral^{4\lambda} } \lt( \int_{\theta^2 <\lal \rho\ral^{-2\sigma}}	\frac{1}{\lal \rho \ral^{-12\sigma}}  \sin \theta \,d\theta + \int_{\theta^2 \ge \lal \rho\ral^{-2\sigma}}	\frac{1}{ \theta^{12}} \sin \theta \,d\theta \rt) d\rho \\
	&  \lesssim \int_{\rho \sim N} \frac{\rho^{6-2\sigma}}{\lal \rho\ral^{4\lambda-10\sigma} \lal \eta \ral^{4\lambda} } \, d\rho \\
	& \lesssim  \frac{N^{7-2\sigma}}{\lal N \ral^{4\lambda-10\sigma}\lal \eta\ral^{4\lambda} }.
\end{aligned}
\]

Thus, we deduce that
\[
\begin{aligned}
	\|\mathfrak{M} \varphi_N \|_{\dot{H}_{\xi}^{k}} 
	&\lesssim \lt( \frac{N^{(7-2\sigma)/2}}{\lal N \ral^{2\lambda-\sigma}\lal \eta\ral^{2\lambda} } \rt)^{1-k/2} \lt( \frac{N^{(7-2\sigma)/2}}{\lal N \ral^{2\lambda-5\sigma}\lal \eta\ral^{2\lambda} } \rt)^{k/2} \lesssim \frac{N^{7/2-\sigma}}{\lal N \ral^{2\lambda-\sigma(1+2k)}\lal \eta\ral^{2\lambda} } \lesssim \frac{1}{\lal N \ral^{2\lambda-7/2-2\sigma k}}
\end{aligned}
\]
which is summable over $N\ge1$ if $\lambda> 7/4+\sigma k$ for any $0\leq k\leq 2$. Hence, we establish that $\mathfrak{M}\in M_{\eta, \xi}^{k, \infty}$ when $|\eta|\ge 1$.

In the case $|\eta|<1$, similarly as before, we observe that 
\[
|\Phi| \gtrsim  \frac{|\eta|^{(2-\sigma)/2}}{\lal \xi\ral^{\sigma/2}}, \quad |\nabla_{\xi} \Phi| \lesssim  |\eta|,  \quad \mbox{and} \quad  |\Delta_{\xi} \Phi|\lesssim |\eta|
\]
This yields
\[
\begin{aligned}
	&|\nabla_{\xi} (\frac{1}{\Phi})| 
	\lesssim \frac{\lal \xi\ral^{\sigma}}{|\eta|^{1-\sigma}} \quad \mbox{and} \quad |\Delta_{\xi}(\frac{1}{\Phi})| 
	\lesssim \frac{\lal \xi\ral^{\sigma}}{|\eta|^{1-\sigma}} +\frac{\lal \xi\ral^{3\sigma/2}}{|\eta|^{1-3\sigma/2}} \lesssim \frac{\lal \xi\ral^{3\sigma/2}}{|\eta|^{1-\sigma}}
\end{aligned}
\]
owing to $|\eta|<1$. Consequently, it follows that
\[
\begin{aligned}
	|\mathfrak{M}|&\lesssim \frac{|\xi|^{2-\sigma}}{\lal \xi\ral^{2\lambda-\sigma/2}\lal \eta \ral^{2\lambda}},  \\
	|\Delta_{\xi}\mathfrak{M}|
	&\lesssim 
	\frac{|\xi|^{2-\sigma}|\eta|^{(2-\sigma)/2}}{\lal \xi\ral^{2\lambda}\lal \eta\ral^{2\lambda}}  \lt(  \frac{\lal \xi\ral^{3\sigma/2}}{|\eta|^{1-\sigma}} 
	+\frac{\lal \xi\ral^{\sigma}}{|\xi||\eta|^{1-\sigma}} 
	+\frac{\lal \xi\ral^{\sigma/2}}{|\xi|^2|\eta|^{(2-\sigma)/2}}\rt) \lesssim 
	\frac{|\xi|^{2-\sigma}}{\lal \xi\ral^{2\lambda-3\sigma/2}\lal \eta\ral^{2\lambda}}  
\end{aligned}
\]
for $|\xi| \ge 1 > |\eta|$. This deduces that
\[
\begin{aligned}
	\int_{\R^3} |\mathfrak{M} \varphi ( \frac{\xi}{N}) |^2 \, d\xi 
	&\lesssim \int_{\rho\sim N}  \frac{\rho^{6-2\sigma}}{\lal \rho \ral^{4\lambda-\sigma}\lal \eta \ral^{4\lambda} }  \,d\rho 
	\lesssim     \frac{N^{7-2\sigma}}{\lal N \ral^{4\lambda-\sigma}\lal \eta\ral^{4\lambda} },\\
	\int_{\R^3} |\Delta_{\xi}  \mathfrak{M} \varphi(\frac{\xi}{N}) |^2 \,d\xi 
	&\lesssim \int_{\rho \sim N} \frac{\rho^{6-2\sigma}}{\lal \rho\ral^{4\lambda-3\sigma}\lal \eta \ral^{4\lambda} } \,d\rho 
	\lesssim  \frac{N^{7-2\sigma} }{\lal N \ral^{4\lambda-3\sigma}\lal \eta\ral^{4\lambda} }
\end{aligned}
\]
which implies 
\[
\begin{aligned}
	\|\mathfrak{M} \varphi_N \|_{\dot{H}_{\xi}^{k}} 
	&\lesssim \lt( \frac{N^{(7-2\sigma)/2}}{\lal N \ral^{2\lambda-\sigma/2}\lal \eta\ral^{2\lambda} } \rt)^{1-k/2} \lt( \frac{N^{(7-2\sigma)/2}}{\lal N \ral^{2\lambda-3\sigma/2}\lal \eta\ral^{2\lambda} } \rt)^{k/2}\\
	&\lesssim \frac{N^{(7-2\sigma)/2}}{\lal N \ral^{2\lambda-\sigma(k+1)/2}\lal \eta\ral^{2\lambda} }\\
	&\lesssim \frac{1}{\lal N \ral^{2\lambda-7/2-\sigma(k-1)/2} }
\end{aligned}
\]
which is summable over $N\ge1$ if $\lambda> 7/4+\sigma(k-1)/4$ for any $0\leq k\leq 2$, thus establishing $\mathfrak{M}\in M_{\eta, \xi}^{k, \infty}$.

We next establish $\mathfrak{M} \in M_{\xi, \eta}^{k, \infty}$. 
Let us consider the following cases:
\begin{itemize}
	\item (B$'$1) $|\eta|<1 $,\\[-4mm]
	\item (B$'$2) $|\eta|\ge 1$.\\[-4mm]
\end{itemize}

In the case (B$'$1), we observe that
\[
|\Phi| \gtrsim \frac{|\eta|^{(2-\sigma)/2}}{\lal \xi\ral^{\sigma/2}}
\]
by Lemma \ref{Lemma_low_Phi} and $|\xi-\eta|\sim |\xi|$.
Given that $\min\{|\eta|, |\xi-\eta|\}=|\eta|$, we also derive
\[ |\nabla_{\eta}\Phi| \lesssim \frac{1}{|\eta|^{\sigma/2}}  \quad \mbox{and} \quad |\Delta_{\eta} \Phi| \lesssim \frac{1}{|\eta|^{1+\sigma/2}}
\]
by Lemma \ref{Lemma_deriv_Phi}. Then, we obtain
\[
|\nabla_{\eta} (\frac{1}{\Phi})| \lesssim \frac{\lal  \xi\ral^{\sigma}}{|\eta|^{2-\sigma/2}}  \quad \mbox{and} \quad
|\Delta_{\eta}(\frac{1}{\Phi})| \lesssim \frac{\lal \xi\ral^{3\sigma/2}}{|\eta|^{3-\sigma/2}},
\]
which leads to
\[
\begin{aligned}
	|\mathfrak{M}|&\lesssim \frac{|\xi|^{2-\sigma}}{\lal \eta \ral^{2\lambda} \lal \xi\ral^{2\lambda-\sigma/2}},\\
	|\Delta_{\eta}\mathfrak{M}|& \lesssim \frac{|\xi|^{2-\sigma}}{\lal \eta \ral^{2\lambda} \lal \xi \ral^{2\lambda}} \lt( \frac{\lal \xi\ral^{3\sigma/2}}{|\eta|^2} +\frac{\lal \xi\ral^{\sigma}}{|\eta|^2} +\frac{\lal \xi\ral^{\sigma/2}}{|\eta|^2} \rt)  
	\lesssim \frac{|\xi|^{2-\sigma}}{|\eta|^{2}\lal \eta\ral^{2\lambda} \lal\xi \ral^{2\lambda-3\sigma/2}}.
\end{aligned}
\]
Consequently, we deduce that
\[
\begin{aligned}
	\int_{\R^3} |\mathfrak{M} \varphi ( \frac{\eta}{N}) |^2  \,d\eta 
	&\lesssim \int_{\rho\sim N}  \frac{\rho^2 |\xi|^{4-2\sigma}}{\lal \rho \ral^{4\lambda}\lal \xi \ral^{4\lambda-\sigma} }  \, d\rho 
	\lesssim     \frac{N^3 |\xi|^{4-2\sigma}}{\lal N \ral^{4\lambda}\lal \xi\ral^{4\lambda-\sigma} },\\
	\int_{\R^3} |\Delta_{\eta}  \mathfrak{M} \varphi(\frac{\eta}{N}) |^2 \,d\eta 
	& \lesssim \int_{\rho \sim N} \frac{|\xi|^{4-2\sigma}}{\rho^2 \lal \rho\ral^{4\lambda}\lal \xi \ral^{4\lambda-3\sigma} } \,d\rho 
	\lesssim  \frac{|\xi|^{4-2\sigma}}{N \lal N \ral^{4\lambda}\lal \xi\ral^{4\lambda-3\sigma} }.
\end{aligned}
\]
Hence, we arrive at
\[
\begin{aligned}
	\|\mathfrak{M} \varphi_N \|_{\dot{H}_{\eta}^{k}} 
	&\lesssim   \lt(  \frac{N^{3/2}|\xi|^{2-\sigma}}{\lal N \ral^{2\lambda}\lal \xi \ral^{2\lambda-\sigma/2} }\rt)^{1-k/2}  \lt(  \frac{|\xi|^{2-\sigma}}{N^{1/2}\lal N \ral^{2\lambda} \lal \xi \ral^{2\lambda-3\sigma/2} }  \rt)^{k/2} \\
	&= \frac{N^{3/2-k}|\xi|^{2-\sigma}}{\lal N \ral^{2\lambda}\lal \xi\ral^{2\lambda-\sigma(k+1)/2} }
	\lesssim  \frac{1}{\lal N \ral^{2\lambda-3/2+k}\lal \xi\ral^{2\lambda-2-\sigma(k-1)/2}},
\end{aligned}
\]
where $0\leq k\leq 3/2$. When $\lambda>1+\sigma(k-1)/4$, this implies $\mathfrak{M} \in M_{ \xi, \eta}^{k, \infty}$ for $0\leq k\leq 3/2$.

In the case (B$'$2), it is straightforward to get
\[
|\Phi| \gtrsim |\eta|(\lal \xi\ral^{-2\sigma}+\theta^2), \quad
|\nabla_{\eta}\Phi| \lesssim |\xi|+ \sin \frac{\beta}{2}\lesssim |\xi|, \quad |\nabla_{\eta} \Phi| \lesssim |\xi|
\]
using Lemma \ref{Lemma_low_Phi} and Lemma \ref{Lemma_deriv_Phi} together with the following fact:
\bq\label{sin_beta}
\sin \frac{\beta}{2} \leq \frac{\sin \beta}{2}=\frac{|\xi|}{2|\eta|} \sin \gamma \leq|\xi|
\eq
for $|\eta|\ge1$.
Then we have 
\[
\begin{aligned}
	|\nabla_{\eta} (\frac{1}{\Phi})|  \lesssim \frac{|\xi|}{|\eta|^2(\lal \xi\ral^{-2\sigma}+\theta^2)^2} \quad \mbox{and} \quad 
	|\Delta_{\eta} (\frac{1}{\Phi})|  \lesssim \frac{|\xi|^2}{|\eta|^3(\lal \xi\ral^{-2\sigma}+\theta^2)^3} 
\end{aligned}
\]
due to $|\xi|\ge|\eta|$ and $\lal \xi \ral^{-2\sigma}+\theta^2 \lesssim 1$.
This leads to
\[
\begin{aligned}
	|\mathfrak{M}| &\lesssim \frac{|\xi|^{2-\sigma}}{|\eta|^{\sigma/2}\lal \eta\ral^{2\lambda} \lal \xi \ral^{2\lambda} (\lal \xi \ral^{-2\sigma}+\theta^2)}
	\lesssim \frac{|\xi|^{2-\sigma}}{\lal \eta\ral^{2\lambda} \lal \xi \ral^{2\lambda} (\lal \xi \ral^{-2\sigma}+\theta^2)} ,  \\
	|\Delta_{\eta} \mathfrak{M}| 
	&\lesssim \frac{|\xi|^{2-\sigma}}{ |\eta|^{2+\sigma/2}\lal \eta \ral^{2\lambda} \lal \xi\ral^{2\lambda} } \lt( \frac{|\xi|^2 }{(\lal\xi \ral^{-2\sigma}+\theta^2)^3}+\frac{|\xi| }{(\lal\xi \ral^{-2\sigma}+\theta^2)^2}+\frac{1}{\lal\xi \ral^{-2\sigma}+\theta^2}\rt) \\
	&\lesssim \frac{|\xi|^{2-\sigma}}{|\eta|^{2+\sigma/2}\lal \eta\ral^{2\lambda} \lal\xi \ral^{2\lambda-2}(\lal \xi\ral^{-2\sigma}+\theta^2)^3}\\
	&\lesssim \frac{|\xi|^{2-\sigma}}{\lal \eta\ral^{2\lambda} \lal\xi \ral^{2\lambda-2}(\lal \xi\ral^{-2\sigma}+\theta^2)^3}
\end{aligned}
\]
since $|\eta|\ge 1$. Thus, we obtain
\[
\begin{aligned}
	\int_{\R^3} |\mathfrak{M} \varphi ( \frac{\eta}{N}) |^2 \,d\eta 
	&\lesssim \int_{\rho\sim N}  \frac{\rho^2 |\xi|^{4-2\sigma}}{\lal \rho \ral^{4\lambda}\lal \xi \ral^{4\lambda} }   
	\lt(\int_{\theta^2 <\lal \xi\ral^{-2\sigma}} \frac{\theta}{\lal \xi\ral^{-4\sigma}} \,d\theta +\int_{\theta^2 \ge\lal \xi\ral^{-2\sigma}} \frac{1}{\theta^3} \,d\theta \rt)d\rho \cr
	& \lesssim     \frac{N^3 |\xi|^{4-2\sigma}}{\lal N \ral^{4\lambda}\lal \xi\ral^{4\lambda-2\sigma} }, \\
	\int_{\R^3} |\Delta_{\eta}  \mathfrak{M} \varphi(\frac{\eta}{N}) |^2 \,d\eta 
	&\lesssim \int_{\rho \sim N} \frac{\rho^2|\xi|^{4-2\sigma}}{ \lal \rho\ral^{4\lambda}\lal \xi \ral^{4\lambda-4} } \lt( \int_{\theta^2 <\lal \xi\ral^{-2\sigma}}	\frac{\theta}{\lal \xi \ral^{-12\sigma}}\,d\theta + \int_{\theta^2 \ge \lal \xi\ral^{-2\sigma}}	\frac{1}{ \theta^{11}}\,d\theta \rt) d\rho \\
	&\lesssim  \frac{N^3|\xi|^{4-2\sigma}}{\lal N \ral^{4\lambda}\lal \xi\ral^{4\lambda-4-10\sigma} }.
\end{aligned}
\]
Hence, we have
\[
\begin{aligned}
	\|\mathfrak{M} \varphi_N \|_{\dot{H}_{\eta}^{k}} 
	&\lesssim   \lt(  \frac{N^{3/2}|\xi|^{2-\sigma}}{\lal N \ral^{2\lambda}\lal \xi \ral^{2\lambda-\sigma} }\rt)^{1-k/2}  \lt(  \frac{N^{3/2}|\xi|^{2-\sigma}}{\lal N \ral^{2\lambda} \lal \xi \ral^{2\lambda-2-5\sigma} }  \rt)^{k/2} \\
	&= \frac{N^{3/2}|\xi|^{2-\sigma}}{\lal N \ral^{2\lambda}\lal \xi\ral^{2\lambda-\sigma-k(1+2\sigma)} }\\ 
	&\lesssim  \frac{1}{\lal N \ral^{2\lambda-3/2}\lal \xi\ral^{2\lambda-2-k(1+2\sigma)} }
\end{aligned}
\]
for $0\leq k\leq 2$.
If $\lambda>1+k(1+2\sigma)/2$, we easily check that it is summable over $N$, and thus $\mathfrak{M} \in M_{\xi, \eta}^{k, \infty}$ for $0\leq k\leq 2$.
This completes the proof for (Case B).
%
%
%
%
%
%
%
%
%
\section{Global-in-time well-posedness}\label{sec_5}
In this section, we provide estimates for the negative Sobolev norm, which facilitate the successful derivation of the $H^{2s}$ estimates, supported by the $L^p$ decay estimates. Thereby, we complete the proof of Theorem \ref{global}.
%
%
%
%
%
%
%
%
%
\subsection{Energy estimates}
We first provide the $\dot{H}^{-\frac{2-\sigma}{2}}$ estimate which is required for $L^p$ decay estimates as stated in Theorem \ref{Theorem_decay}. 
\begin{lemma}\label{H^{-1/2}_norm}
Let $0< \sigma<2$. For $T^{\ast}>0$, we have
	\[
\|\alpha(t)\|_{\dot{H}_x^{-\frac{2-\sigma}{2}}} \lesssim \|\alpha(0)\|_Y +\sup_{0<t<T^{\ast}}\|\alpha(t)\|_{X}^2.
	\]
\end{lemma}
\begin{proof}
	From \eqref{duhamel}, we get
	\bq\label{est_negative_1}
	\| \alpha(t)\|_{\dot{H}^{-\frac{2-\sigma}{2}}} \lesssim  \|  \alpha(0)\|_{\dot{H}^{-\frac{2-\sigma}{2}}} +  \int_0^t \lt\| \mathcal{Q}(\alpha)(\tau) \rt\|_{\dot{H}^{-\frac{2-\sigma}{2}}} d\tau
	\eq
	given the isometric property of the operator $e^{itp(|\nabla|)}$ on $\dot{H}^{-(2-\sigma)/2}$.
Recall that 
\[
\mathcal{Q}(\alpha)
= \iint_{\R^3\times \R^3}  e^{i x\cdot \xi} m_{r, l}(\xi, \eta) \widehat{\alpha_r}(\xi-\eta) \widehat{\alpha_l} (\eta) \,d\xi d \eta,
\]
where 
\[
m_{r, l} \sim p(\xi) r(\xi)r(\xi-\eta) r(\eta)+|\xi| r(\xi-\eta) r(\eta).
\]
Further, we obtain
	\[
	\begin{aligned}
		\frac{\widehat{\mathcal{Q}(\alpha)}(\xi)}{|\xi|^{(2-\sigma)/2}} 
		&=  \int_{\R^3} \frac{m_{r,l}(\xi, \eta)}{|\xi|^{(2-\sigma)/2}} \widehat{\alpha}_{r}(\xi-\eta) \widehat{\alpha}_{l}(\eta) \,d\eta \\
		& \sim  \int_{\R^3}  (1+|\xi|^{\sigma})^{\frac{1}{2}} r(\xi )r(\xi-\eta)\widehat{\alpha}_{r}(\xi-\eta) r(\eta) \widehat{\alpha}_{l}(\eta)\, d\eta   +  \int_{\R^3} |\xi|^{\frac{\sigma}{2}}  r(\xi-\eta)\widehat{\alpha}_{r}(\xi-\eta)  r(\eta) \widehat{\alpha}_{l}(\eta)\, d\eta.  
	\end{aligned}
	\]
By using the basic relations
	\[
	(1+|\xi|^{\sigma})^{\frac{1}{2}} \leq 1+|\xi|^{\frac{\sigma}{2}} \quad \text{and} \quad |\xi|^{\frac{\sigma}{2}} \leq |\xi-\eta|^{\frac{\sigma}{2}}+|\eta|^{\frac{\sigma}{2}},
	\]
we deduce that 
	\[
	\begin{aligned}
	\|\mathcal{Q}(\alpha)\|_{\dot{H}^{-\frac{2-\sigma}{2}}} 
	&\lesssim 
\lt\|  (r(\nabla)\alpha_{r}) (r(\nabla)\alpha_{l}) \rt\|_{L^2}  +\| (r(\nabla)\alpha_{r}) ( |\nabla|^{\sigma/2}r(\nabla)\alpha_{l}) \|_{L^2}\\
&\lesssim \| \alpha\|_{L^2}    \|r(\nabla)\alpha\|_{L^{\infty}}  
+ \| \alpha\|_{\dot{H}^{\sigma/2}}\|   r(\nabla)\alpha\|_{L^{\infty}} 
	\end{aligned}
	\]
by virtue of \eqref{est_r}, the Plancherel theorem, and symmetry.
By employing the embedding property $W^{\frac{3}{p}+\delta, p} \subset L^{\infty} $ for small $\delta>0$, we utilize \eqref{est_r} to obtain   
	\bq\label{est_negative_2}
	\begin{aligned}
		\int_0^t \lt\| \mathcal{Q}(\alpha)(\tau) \rt\|_{\dot{H}^{-\frac{2-\sigma}{2}}} d\tau 
		&\lesssim \sup_{0<t<T^{\ast}} (\| \alpha(t)\|_{L_x^2}+\| \alpha(t)\|_{\dot{H}_x^{\sigma/2}} ) \int_{0}^t   \|\alpha(\tau)\|_{W^{\frac{3}{p}+\delta, p}}  \,d\tau \\
		&\lesssim \sup_{0<t<T^{\ast}}\| \alpha(t)\|_{X}^2  \int_0^t \frac{1}{(1+\tau)^{\beta_{\sigma}}}\, d\tau\\
		&\lesssim \sup_{0<t<T^{\ast}} \| \alpha(t)\|_{X}^2  
	\end{aligned}
	\eq
	if $\beta_{\sigma}>1$ and $p<\infty$. 
	By combining \eqref{est_negative_1} and \eqref{est_negative_2}, we conclude the desired result.
\end{proof}

Next, we present $H^{2s}$ estimates in the lemma below.
\begin{lemma}\label{H^s_norm}
Let $s>1+\frac{3}{p}$.
For $T^{\ast}>0$,
	we have
	\[
\|\alpha(t)\|_{H^{2s}}
	\lesssim \|\alpha(0)\|_{Y} +\sup_{0<t<T^{\ast}}\|\alpha(t)\|_{X}^{3/2},
	\]
	where $\alpha$ is given as in \eqref{sol}, with 
	\[
	\|(n_0, u_0)\|_{Y}<\varepsilon \quad \mbox{and} \quad \sup_{0<t<T^{\ast}} \|n(t)\|_{L^{\infty}}<\frac{1}{2}.
	\]
\end{lemma}
\begin{proof}
Note that the system \eqref{ER_1} can be reformulated as
	\bq\label{ER_1_simple}
	\pa_t w +A_j(w)\pa_j w = \lt(0, \,	-\nabla |\nabla|^{-\sigma} n \rt)^T,
	\eq
	where $w=(w^0, w^1, w^2, w^3)=(n, u)^T$ and matrix $A_j$ for $j=1,2,3$ is given by
	$$
	A_j(w)=	\begin{pmatrix} 
		w^j   & ( w^0+1) e_j^T \\
		(w^0 +1) e_j & w^j \mathbf{I}_3  \\
	\end{pmatrix}_{4\times 4} .
	$$
Since
	\[
	{\rm Re}(\alpha) =\frac{p(|\nabla|)}{|\nabla|}n \quad \text{and} \quad {\rm Im}(\alpha)=\frac{\nabla }{|\nabla|}\cdot u,
	\]
it suffices to show 
	\bq\label{H^s_norm_sub}
	\|w(t)\|_{H^{2s}}^2 +\||\nabla|^{-\frac{\sigma}{2}} n(t)\|_{H^{2s}}^2 \lesssim   \|w(0) \|_{Y}^2  +\sup_{0<t<T^{\ast}}\|w(t)\|^{3}_{X}
	\eq
	for $0<t<T^{\ast}$.
	In fact, once \eqref{H^s_norm_sub} is established, we deduce
	\[
	\begin{aligned}
		\|\alpha(t)\|_{H^{2s}}^2
		&=\lt\|\frac{p(|\nabla|)}{|\nabla|}n(t) \rt\|_{H^{2s}}^2 +\lt\|\frac{\nabla }{|\nabla|}\cdot u(t) \rt\|_{H^{2s}}^2 \\
		&\lesssim \|n(t)\|_{H^{2s} \cap \dot{H}^{-\sigma/2}}^2 +\|u(t)\|_{H^{2s}}^2\\
		&\lesssim \|w(t)\|_{{H}^{2s}}^2 +\|n(t)\|_{\dot{H}^{-\sigma/2}\cap H^{2s}}^2 \\
		&\lesssim \|w(0)\|_{Y}^2 +\sup_{0<t<T^{\ast}}\|w(t)\|_{X}^{3}\\
		&\lesssim \|\alpha(0)\|_{Y}^2 +\sup_{0<t<T^{\ast}}\|\alpha(t)\|_{X}^{3}
	\end{aligned}
	\]
	by \eqref{nu_alpha}, thereby concluding the desired result.
	
	Now, we prove \eqref{H^s_norm_sub}.
	From \eqref{ER_1_simple}, we derive
	\bq\label{energy0}
	\begin{aligned}
		\frac{1}{2}\frac{d}{dt} \int_{\R^3} \lt| |\nabla|^{2s} w \rt|^2  dx 
		&= -\int_{\R^3}  \lt[|\nabla|^{2s}, A_{j}(w) \rt]\pa_j w \cdot |\nabla|^{2s} w  \,dx \\
		&\quad +\frac{1}{2}\int_{\R^3}   \pa_j A_{j}(w) \lt||\nabla|^{2s}w \rt|^2   dx 
		- \int_{\R^3}   \nabla |\nabla|^{2s-\sigma} n \cdot |\nabla|^{2s} u   \,dx\\
		& \lesssim \|\nabla w\|_{L^{\infty}} \|w\|_{\dot{H}^{2s}}^2- \int_{\R^3}   \nabla |\nabla|^{2s-\sigma} n \cdot |\nabla|^{2s} u \, dx
	\end{aligned}
	\eq
by using \eqref{tech_1} for $s>0$.
To estimate the last term on the right-hand side of the above, we consider two cases: $0<\sigma<1$ and $1\leq \sigma<2$. 
In the case $1\leq \sigma<2$, we observe that 
	\bq\label{Riesz_1}
	\begin{aligned}
		 \frac{1}{2} \frac{d}{dt}\int_{\R^3} |\nabla |\nabla|^{2s-1-\frac{\sigma}{2}} n|^2 \,dx  =\int_{\R^3}  \nabla |\nabla|^{2s-\sigma}  n \cdot |\nabla|^{2s} u \,dx 
		-\int_{\R^3}  \nabla |\nabla|^{2s-\sigma}n \cdot \nabla \lt(|\nabla|^{2s-2}\nabla \cdot(nu) \rt) dx.
	\end{aligned}
	\eq
Using the Sobolev inequality and \eqref{KP_ineq}, we estimate
	\[
	\begin{aligned}
	\lt| \int_{\R^3}  \nabla |\nabla|^{2s-\sigma}n \cdot \nabla \lt(|\nabla|^{2s-2}\nabla \cdot(nu) \rt) dx \rt|	
	 \lesssim  \|n\|_{\dot{H}^{2s+1-\sigma}} \|nu\|_{\dot{H}^{2s}} \lesssim  \|w\|_{L^{\infty}}\|w\|_{\dot{H}^{2s-1}}^{\sigma-1} \|w\|_{\dot{H}^{2s}}^{3-\sigma} \lesssim \|w\|_{L^{\infty}} \|w\|_{H^{2s}}^2	
	\end{aligned}
	\]
for $1\leq \sigma<2$. Hence, combining \eqref{energy0} and \eqref{Riesz_1}, we arrive at
\bq\label{energy1}
\frac{1}{2}\frac{d}{dt} (\|w\|_{\dot{H}^{2s}} ^2 +  \| \nabla n\|_{\dot{H}^{2s-1-\frac{\sigma}{2}}}^2)  \lesssim \| w\|_{W^{1, \infty}} \|w\|_{H^{2s}}^2
\eq
for $1\leq \sigma<2$.

To deal with the case $0<\sigma<1$, we shall use the modification argument developed in \cite{CJ23}.
Let us introduce a modified space:
\[
\|f\|_{\dot{\mathcal{H}}^s}:= \| (n+1)^{-\frac{1}{2}}|\nabla|^{s} f\|_{L^2}, \qquad s>0.
\]
If $\sup_{0<t<T^{\ast}}\|n(t)\|_{L^{\infty}}<1$, we observe that
\[
\lt(1-\sup_{0<t<T^{\ast}}\|n(t)\|_{L^{\infty}}\rt)^{1/2}\|f\|_{\dot{\mathcal{H}}^{s}} \lesssim \|f\|_{\dot{H}^s} \lesssim \lt(1+\sup_{0<t<T^{\ast}}\|n(t)\|_{L^{\infty}}\rt)^{1/2}\|f\|_{\dot{\mathcal{H}}^s}
\]
which signifies that  $\|f\|_{\dot{\mathcal{H}}^s} \sim \|f\|_{\dot{H}^s}$ for $s>0$. Then, we estimate
\[
\begin{aligned}
&\frac{1}{2} \frac{d}{dt} \int_{\R^3} \frac{1}{n+1} \lt||\nabla|^{2s-\frac{\sigma}{2}}n \rt|^2 dx\\ 
&\quad = \frac{1}{2} \int_{\R^3} \frac{1}{n+1} \nabla \cdot u \lt||\nabla|^{2s-\frac{\sigma}{2}}n\rt|^2 dx
 +\frac{1}{2} \int_{\R^3} \frac{1}{(n+1)^2} u\cdot \nabla n \lt||\nabla|^{2s-\frac{\sigma}{2}}n \rt|^2 dx\\
&\qquad -\int_{\R^3} \frac{1}{n+1}  |\nabla|^{2s-\frac{\sigma}{2}} n \lt[|\nabla|^{2s-\frac{\sigma}{2}}, n+1\rt] \nabla \cdot u \,dx 
-\int_{\R^3}  |\nabla|^{2s-\frac{\sigma}{2}} n |\nabla|^{2s-\frac{\sigma}{2}}\nabla  \cdot u \,dx\\
&\qquad -\int_{\R^3} \frac{1}{n+1} |\nabla|^{2s-\frac{\sigma}{2}}n \lt[|\nabla|^{2s-\frac{\sigma}{2}}, u\rt] \cdot \nabla n \,dx -\frac{1}{2}\int_{\R^3} \frac{1}{n+1} u \cdot \nabla \lt| |\nabla|^{2s-\frac{\sigma}{2}}  n\rt|^2 dx\\
&\quad =: \sum_{i=1}^6 I_i,
\end{aligned}
\]
where we easily get
\[
I_1 \lesssim \frac{1}{1-\sup_{0<t<T^{\ast}}\|n(t)\|_{L^{\infty}}} \|\nabla w\|_{L^{\infty}} \|w\|_{\dot{H}^{2s-\frac{\sigma}{2}}}^2 \lesssim \|\nabla w\|_{L^{\infty}} \|w\|_{\dot{H}^{2s-\frac{\sigma}{2}}}^2.
\]
By virtue of \eqref{tech_1}, it follows that
\[
\begin{aligned}
I_3 &=-\int_{\R^3} \frac{1}{n+1}  |\nabla|^{2s-\frac{\sigma}{2}} n \lt[|\nabla|^{2s-\frac{\sigma}{2}}, n\rt] \nabla \cdot u \,dx \cr
& \lesssim \lt\||\nabla|^{2s-\frac{\sigma}{2}} n\rt\|_{L^2} \lt\|[|\nabla|^{2s-\frac{\sigma}{2}}, n] \nabla \cdot u\rt\|_{L^2} \cr
&\lesssim \|\nabla w\|_{L^{\infty}} \|w\|_{\dot{H}^{2s-\frac{\sigma}{2}}}^2,\\
I_5 &=-\int_{\R^3} \frac{1}{n+1} |\nabla|^{2s-\frac{\sigma}{2}}n [|\nabla|^{2s-\frac{\sigma}{2}}, u] \cdot \nabla n \,dx \lesssim \|\nabla w\|_{L^{\infty}} \|w\|_{\dot{H}^{2s-\frac{\sigma}{2}}}^2.
\end{aligned}
\]
For $I_6$, we find
\[
\begin{aligned}
I_6 
&= 
\frac{1}{2} \int_{\R^3} \nabla (\frac{1}{n+1}) \cdot u \lt||\nabla|^{2s-\frac{\sigma}{2}}n\rt|^2 dx +J_1\\
&= -\frac{1}{2} \int_{\R^3} \frac{1}{(n+1)^2} \nabla n \cdot u \lt||\nabla|^{2s-\frac{\sigma}{2}} n\rt|^2 dx +J_1\\
&=-I_2+I_1.
\end{aligned}
\]
Finally, we use
\[
I_4 =\int_{\R^3}  \nabla |\nabla|^{2s-\sigma} n  \cdot |\nabla|^{2s} u \,dx,
\]
to deduce that
\bq\label{Riesz_2}
\begin{aligned}
	\frac{1}{2} \frac{d}{dt} \|n\|_{\dot{\mathcal{H}}^{2s-\frac{\sigma}{2}}}^2 -\int_{\R^3}  \nabla |\nabla|^{2s-\sigma} n  \cdot |\nabla|^{2s} u \,dx
	\lesssim \|\nabla w\|_{L^{\infty}} \|w\|_{\dot{H}^{2s-\frac{\sigma}{2}}}^2.
\end{aligned}
\eq
It follows from \eqref{energy0} and \eqref{Riesz_2} that
\bq\label{energy11}
\frac{1}{2} \frac{d}{dt} \lt( \|w\|_{\dot{H}^{2s}}^2 +\|n\|_{\dot{\mathcal{H}}^{2s-\frac{\sigma}{2}}}^2\rt) \lesssim \|\nabla w\|_{L^{\infty}} \lt(\|w\|_{\dot{H}^{2s}}^2+\|w\|_{\dot{H}^{2s-\frac{\sigma}{2}}}^2\rt)
\eq
for $0<\sigma<1$.

The subsequent $L^2$ estimate can be obtained as
	\bq\label{energy2}
	\begin{aligned}
		&\frac{1}{2} \frac{d}{dt} (\|w\|_{L^2}^2+ 
		\| \nabla n\|_{\dot{H}^{-1-\frac{\sigma}{2}}}^2)  \lesssim \|\nabla w\|_{L^{\infty}} \|w\|_{L^2}^2.
	\end{aligned}
	\eq
 Combining \eqref{energy1}, \eqref{energy11}, and \eqref{energy2} yields that  
\bq\label{energy41}
\begin{aligned}
\frac{1}{2}\frac{d}{dt} \Big(\|w\|_{H^{2s}} ^2 & + \Big\| \frac{\nabla}{|\nabla|^{1+\sigma/2}}n \Big\|_{H^{2s}}^2 
 \Big) 
  \lesssim \| w\|_{W^{1, \infty}} \|w\|_{H^{2s}}^2
 \lesssim \| w\|_{W^{1+\frac{3}{p}+\delta, p}} \|w\|_{X}^2
\end{aligned}
\eq
for $1\leq \sigma<2$ and sufficiently small $\delta>0$. Similarly,
\bq\label{energy42}
\begin{aligned}
	\frac{1}{2}\frac{d}{dt} \Big(\|w\|_{H^{2s}} ^2 &+ \sum_{l=1}^s \| n \|_{\dot{\mathcal{H}}^{2l-\frac{\sigma}{2}}}^2 + \Big\| \frac{\nabla}{|\nabla|^{1+\sigma/2}}n \Big\|_{L^2}^2 
	\Big) 
	\lesssim \| w\|_{W^{1+\frac{3}{p}+\delta, p}} \|w\|_{X}^2
\end{aligned}
\eq
for $0<\sigma<1$.
By applying the Gr\"onwall's lemma for \eqref{energy41} and \eqref{energy42}, and using the equivalences $\|n\|_{\dot{H}^s} \sim \|n\|_{\dot{\mathcal{H}}^s}$ and $
\lt\|\nabla  n(t)\rt\|_{\dot{H}^{-1-\frac{\sigma}{2}}} \sim \| |\nabla|^{-\sigma/2}n(t)\|_{L^2}$, we deduce
\[
\begin{aligned}
	\|w(t)\|_{H^{2s}} ^2 +  \| |\nabla|^{-\frac{\sigma}{2}}  n(t)\|_{H^{2s}}^2 
	& \lesssim \|w(0)\|_{Y}^2
	+ \|w(t)\|_X^2  \int_0^t \|w(\tau)\|_{W^{1+\frac{3}{p}+\delta, p}} \,d\tau\\
	&\lesssim \|w(0)\|_{Y}^2  +  \lt(\sup_{0<t<T^{\ast}}\|w(t)\|_{X}^3\rt) \int_0^t (1+\tau)^{-\beta_\sigma} \,d\tau \\
	&\lesssim \|w(0)\|_{Y}^2  + \sup_{0<t<T^{\ast}}\|w(t)\|_{X}^3 
\end{aligned}
\]
for $\beta_{\sigma}>1$ since $s>1+\frac{3}{p}$. 
It leads to the estimate \eqref{H^s_norm_sub}, thereby concluding the proof.
\end{proof}

%
%
%
%
%
%
%
%
%
\subsection{Proof of Theorem \ref{global}}
The local-in-time existence and uniqueness of a solution $\alpha \in \mathcal{C}([0, T^{\ast});X)$ can be established through the standard argument detailed in \cite{K75}. By Lemma \ref{H^{-1/2}_norm}, Lemma \ref{H^s_norm}, and Theorem \ref{Theorem_decay} we deduce that 
\[
\||\nabla|^{-\frac{2-\sigma}{2}}\alpha(t)\|_{H^{2s}}+(1+t)^{\beta_{\sigma}}\|\alpha(t)\|_{W^{s, p}} \lesssim \|\alpha(0)\|_{Y} +\sup_{0<t<T^{\ast}}\|\alpha\|_{X}^{3/2}
\]
holds provided that $\sup_{0<t<T^{\ast}} \|\alpha(t)\|_X \leq 1$. This inequality allows us to derive
\bq\label{uniform_estimate}
\sup_{0<t<T^{\ast}} \|\alpha (t)\|_{X} \leq C^{\ast}\|\alpha(0)\|_{Y} 
\eq
for any $T^{\ast}>0$.
We set 
\bq\label{vare_1}
\|\alpha(0)\|_{Y}^2=\varepsilon_2^2 := \frac{\varepsilon_1^2}{2(1+C^{\ast})}
\eq
for $\varepsilon_1 > 0$ and $C^{\ast} > 0$ given in \eqref{uniform_estimate}.
By the local well-posedness result, the following set is nonempty:
\bq\nonumber
\mathcal{S}:=\lt\{ T>0\, :\, \sup_{0\leq t\leq T} \lt( \| \alpha(t)\|_{\dot{H}^{-\frac{2-\sigma}{2}}\cap H^{2s}}^2 +(1+t)^{2\beta_{\sigma}}\|\alpha(t)\|_{W^{s, p}}^2\rt)\leq \varepsilon_1^2\rt\}.
\eq  
Assuming, for contradiction, that $\sup\mathcal{S}=: T^{\ast}<\infty$, it follows from \eqref{uniform_estimate} and \eqref{vare_1} that
\bq\nonumber
\varepsilon_1^2 =\sup_{0<t<T^{\ast}} \| \alpha(t)\|_{X}^2 \leq C^{\ast} \| \alpha(0)\|_{Y}^2  \leq C^{\ast} \varepsilon_2^2=\frac{C^{\ast}\varepsilon_1^2}{2(1+C^{\ast})}  <\varepsilon_1^2.
\eq 
This contradiction demonstrates that the solution indeed belongs to $\mathcal{C}(\mathbb{R}_{+}; X)$.

%
%
%
%
%
%
%
%
%

\section*{Acknowledgments}
The first and last authors are supported by NRF grants (No. 2022R1A2C1002820 and RS-2024-00406821.)
The second author was supported by NRF grant (No. RS-2022-00165600).
The last author is supported by NRF grant (No. NRF-2022R1I1A1A01068481).

\appendix

%
%
%
%
%
%
%
%
%
\section{Proof of Proposition \ref{Proposition_m}: Cases A \& C} \label{appendix_1}
This appendix is devoted to proving Proposition \ref{Proposition_m} for cases A and C.
\subsection{Case A}
In the case $|\xi|\leq|\eta|/2$,
we note that $|\xi|$ is the smallest value, and $|\xi-\eta| \sim |\eta|$ since
\[
\frac{|\eta|}2 \leq |\eta|-|\xi|\leq |\xi-\eta| \leq |\eta|+|\xi|\leq \frac{3|\eta|}2.
\]
Given that $p(r)=r^{(2-\sigma)/2}(1+r^{\sigma})^{1/2}$ is an increasing function, a direct computation shows 
\bq\label{low_Phi_3}
\begin{aligned}
	|\Phi|=|p(\xi)-p(\xi-\eta)-p(\eta)|
	& \gtrsim \max \{ p(\xi-\eta), p(\eta)\}\\
	&\gtrsim \lt\{\begin{array}{lcl} 
		|\eta|^{(2-\sigma)/2} & \mbox{for} & \max\{|\xi-\eta|, |\eta|\}<1,\\
		|\eta| & \mbox{for} & \max\{|\xi-\eta|, |\eta|\}\ge1,
	\end{array}\rt.
\end{aligned}
\eq
based on the equivalence $|\xi-\eta| \sim |\eta|$.

First, we claim $\mathfrak{M} \in  M_{\eta, \xi}^{k, \infty}$. For this, we consider two cases:
\begin{itemize}
	\item (A1)  $\max\{|\xi-\eta|, |\eta|\} < 1$, \\[-4mm]
	\item (A2) $\max\{|\xi-\eta|, |\eta|\} \ge 1$. \\[-4mm]
\end{itemize}

In case (A1), it follows that
\[
|\Phi| \gtrsim |\eta|^{\frac{2-\sigma}{2}}
\]
by \eqref{low_Phi_3}. Additionally, we find
\[
\begin{aligned}
	|\nabla_{\xi} \Phi| \lesssim \frac{1}{|\xi|^{\sigma/2}} +\frac{1}{|\xi-\eta|^{\sigma/2}} \sin \frac{\gamma}{2} \lesssim \frac{1}{|\xi|^{\sigma/2}} \quad \mbox{and} \quad 
	|\Delta_{\xi}\Phi|\lesssim \frac{1}{|\xi|^{1+\sigma/2}}
\end{aligned}
\]
by Lemma \ref{Lemma_deriv_Phi} along with the fact $|\xi|\leq |\xi-\eta|$. 
Then, we obtain
\bq\label{deri1_1/Phi_1}
\begin{aligned}
	\lt|\nabla_{\xi} (\frac{1}{\Phi})\rt|
	=\lt|-\frac{\nabla_{\xi}\Phi}{\Phi^2}\rt| 
	\lesssim \frac{1}{ |\eta|^{2-\sigma}|\xi|^{\sigma/2}}  
\end{aligned}
\eq
and 
\bq\label{deri2_1/Phi_1}
\begin{aligned}
	\lt|\Delta_{\xi} (\frac{1}{\Phi})\rt| &= \lt|-\frac{\Delta_\xi \Phi}{\Phi^2}+\frac{2|\nabla_{\xi} \Phi|^2}{\Phi^3}\rt| 
	\lesssim  \frac{1}{ |\eta|^{2-\sigma} |\xi|^{1+\sigma/2}}  +\frac{1}{ |\eta|^{3(2-\sigma)/2}|\xi|^{\sigma}} 
	\lesssim \frac{1}{ |\eta|^{2-\sigma} |\xi|^{1+\sigma/2}}	
\end{aligned}
\eq
due to $|\xi|\leq |\eta|$. Considering
\[
\mathfrak{M}(\xi, \eta)=\frac{|\xi|^{(2-\sigma)/2}|\xi-\eta|^{(2-\sigma)/2}|\eta|^{(2-\sigma)/2}} {\lal \xi-\eta\ral^{2 \lambda} \lal \eta\ral^{2 \lambda}\Phi(\xi, \eta)}= \frac{g(\xi, \eta)}{\Phi(\xi, \eta)}
\]
we deduce that
\[
|\mathfrak{M}| \lesssim \frac{|\xi|^{(2-\sigma)/2} |\eta|^{(2-\sigma)/2}}{\lal\eta\ral^{4\lambda}}
\]
due to $|\xi-\eta| \sim |\eta|$. By using \eqref{Del_M}, we apply \eqref{deri1_1/Phi_1} and \eqref{deri2_1/Phi_1} to derive
\[
\begin{aligned}
	|\Delta_{\xi}\mathfrak{M} |
	&\lesssim 
	\frac{|\xi|^{(2-\sigma)/2}|\eta|^{2-\sigma}}{\lal \eta\ral^{4\lambda}} \lt( 
	\frac{1}{ |\eta|^{2-\sigma} |\xi|^{1+\sigma/2}}  + \frac{1}{|\xi|^{2}|\eta|^{(2-\sigma)/2}} \rt)  \lesssim
	\frac{|\eta|^{(2-\sigma)/2}}{|\xi|^{1+\sigma/2}\lal \eta\ral^{4\lambda}} 
\end{aligned}
\]
for $|\xi|\leq |\eta|$. Consequently, we have
\[
\begin{aligned}
	\int_{\R^3} |\mathfrak{M} \varphi( \frac{\xi}{N})  |^2  \,d\xi
	\lesssim \int_{|\xi|\sim N}  \frac{|\eta|^{2-\sigma}|\xi|^{2-\sigma}}{\lal \eta \ral^{8\lambda} }  \,d\xi
	\lesssim \int_{\rho\sim N}  \frac{\rho^{4-\sigma}}{\lal \rho \ral^{8\lambda-2+\sigma} } \, d\rho
	\lesssim \frac{N^{5-\sigma}}{ \lal N\ral^{8\lambda-2+\sigma}}
\end{aligned}
\]
and
\[
\begin{aligned}
	\int_{\R^3} |\Delta_{\xi} \mathfrak{M} \varphi(\frac{\xi}{N})  |^2 \, d\xi 
	 \lesssim \int_{|\xi|\sim N}
	\frac{1}{|\xi|^{2+\sigma}\lal \eta\ral^{8\lambda-2+\sigma} } \, d\xi \lesssim \int_{\rho\sim N}
	\frac{1}{\rho^{\sigma}\lal \rho\ral^{8\lambda-2+\sigma} } \, d\rho
	\lesssim \frac{N^{1-\sigma}}{ \lal N\ral^{8\lambda-2+\sigma}}.
\end{aligned}
\]
By employing an interpolation argument, we derive
\[
\begin{aligned}
	\|\mathfrak{M} \varphi_N \|_{\dot{H}_{\xi}^{k}} 
	 =\| \mathfrak{M} \varphi_N   \|_{L_{\xi}^2}^{1-k/2}
	\| \mathfrak{M} \varphi_N  \|_{\dot{H}_{\xi}^2}^{k/2} \lesssim 
	\lt( \frac{N^{(5-\sigma)/2} }{ \lal N\ral^{4\lambda-1+\sigma/2}}\rt)^{1-k/2}\lt(\frac{N^{(1-\sigma)/2}}{ \lal N\ral^{4\lambda-1+\sigma/2}}\rt)^{k/2} = \frac{N^{(5-\sigma)/2-k} }{\lal N\ral^{4\lambda-1+\sigma/2}}
\end{aligned}
\]
for $k\leq 2$.
Hence, we arrive at
\[
\Big\| \sum_{N<1} \| \mathfrak{M} \varphi_N  \|_{\dot{H}_{\xi}^k} \Big\|_{L_{\eta}^{\infty}} \leq  \sum_{N<1} \frac{1}{  \lal N\ral^{4\lambda-7/2+\sigma+k}}<\infty
\]
if $\lambda>7/8-\sigma/4-k/4$ with $0\leq k\leq 2$.

In case (A2) where $\max\{|\xi-\eta|, |\eta|\} \ge 1$ holds, we first observe
\[
\begin{aligned}
	|\Phi| \gtrsim |\eta|,\quad 
	|\nabla_{\xi} \Phi| \lesssim \frac{1}{|\xi|^{\sigma/2}} +\frac{1}{|\xi-\eta|^{\sigma/2}} \lesssim \frac{1}{|\xi|^{\sigma/2}}, \quad \text{and} \quad
	|\Delta_{\xi}\Phi|\lesssim \frac{1}{|\xi|^{1+\sigma/2}}
\end{aligned}
\]
by \eqref{low_Phi_3} and Lemma \ref{Lemma_deriv_Phi} when $|\xi| <1 $.
These estimates lead to
\[
\begin{aligned}
	\lt|\nabla_{\xi} (\frac{1}{\Phi})\rt|
	&\lesssim \frac{1}{ |\eta|^2|\xi|^{\sigma/2}} \quad \mbox{and} \quad 
	\lt|\Delta_{\xi} (\frac{1}{\Phi})\rt| 
	&\lesssim 
	\frac{1}{ |\eta|^2 |\xi|^{1+\sigma/2}} +\frac{1}{ |\eta|^{3}|\xi|^{\sigma}}.  
\end{aligned}
\]
By using \eqref{M_g}, we get
\[
|\mathfrak{M}| \lesssim \frac{|\xi|^{(2-\sigma)/2}|\eta|^{1-\sigma}}{\lal \eta\ral^{4\lambda}}
\]
and
\[
\begin{aligned}
	|\Delta_{\xi}\mathfrak{M} |
	&\lesssim 
	\frac{|\xi|^{(2-\sigma)/2}|\eta|^{2-\sigma}}{\lal \eta\ral^{4\lambda}}
	\lt( \frac{1}{ |\eta|^{3}|\xi|^{\sigma}} +
	\frac{1}{ |\eta|^2 |\xi|^{1+\sigma/2}}  
	+ \frac{1}{ |\eta| |\xi|^{2}} 
	\rt)  
	\lesssim 
	\frac{|\eta|^{1-\sigma}}{ |\xi|^{1+\sigma/2}\lal \eta\ral^{4\lambda}}
\end{aligned}
\]
since $|\eta| \gtrsim 1$ and $|\xi|< 1$. Subsequently, we obtain
\bq\label{A2_m}
	\int_{\R^3} |\mathfrak{M} \varphi ( \frac{\xi}{N} )  |^2  \, d\xi
	 \lesssim \int_{|\xi|\sim N}  \frac{|\xi|^{2-\sigma}}{\lal \eta \ral^{8\lambda-2+2\sigma} } \, d\xi  \lesssim \int_{\rho\sim N}  \frac{\rho^{4-\sigma}}{\lal \eta \ral^{8\lambda-2+2\sigma} } \, d\rho
	\lesssim \frac{N^{5-\sigma}}{ \lal N\ral^{8\lambda-2+2\sigma}}, 
\eq
\[
	\int_{\R^3} |\Delta_{\xi} \mathfrak{M} \varphi(\frac{\xi}{N})  |^2 \, d\xi 
	\lesssim \int_{|\xi|\sim N}
	\frac{1}{|\xi|^{2+\sigma} \lal \eta\ral^{8\lambda-2+2\sigma}}   \,d\xi  \lesssim \int_{\rho\sim N}
	\frac{1}{ \rho^{\sigma}\lal \rho\ral^{8\lambda-2+2\sigma}} \, d\rho
	\lesssim \frac{N^{1-\sigma}}{\lal N \ral^{8\lambda-2+2\sigma}} 
\]
which yields that 
\[
\begin{aligned}
	\|\mathfrak{M} \varphi_N \|_{\dot{H}_{\xi}^{k}} 
	 =\| \mathfrak{M} \varphi_N   \|_{L_{\xi}^2}^{1-k/2}
	\| \mathfrak{M} \varphi_N  \|_{\dot{H}_{\xi}^2}^{k/2}  \lesssim  \lt(\frac{N^{(5-\sigma)/2} }{ \lal N\ral^{4\lambda-1+\sigma}} \rt)^{1-k/2} \lt(\frac{N^{(1-\sigma)/2}}{\lal N \ral^{4\lambda-1+\sigma}}\rt)^{k/2}
	\lesssim 
	\frac{N^{(5-\sigma)/2-k} }{\lal N\ral^{4\lambda-1+\sigma}}
\end{aligned}
\]
for $k\leq 2$ and $|\xi|<1$. Hence, we have
\bq \label{A2_1}
\begin{aligned}
	\sum_{N<1} \| \mathfrak{M} \varphi_N  \|_{L_{\eta}^{\infty}\dot{H}_{\xi}^k} 
	\lesssim \sum_{N<1}  \frac{1}{ \lal N\ral^{4\lambda-7/2+3\sigma/2+k}} <\infty
\end{aligned}
\eq
for $0\leq k\leq 3/2$ and $\lambda>7/8-3\sigma/8-k/4$.

When $|\xi|\ge1$, we obtain
\[
|\Phi| \gtrsim |\eta|,\quad
|\nabla_{\xi} \Phi| \lesssim  |\eta|+2\sin \frac{\gamma}{2}  \lesssim |\eta|, \quad \text{and} \quad
|\Delta_{\xi}\Phi| \lesssim |\eta| 
\]
by using \eqref{low_Phi_3} and Lemma \ref{Lemma_deriv_Phi} along with \eqref{sin_gamma}. This implies
\[
\begin{aligned}
	|\nabla_{\xi} (\frac{1}{\Phi})| 
	\lesssim  \frac{1}{|\eta|} \quad \mbox{and} \quad 
	|\Delta_{\xi}(\frac{1}{\Phi})| \lesssim \frac{1}{|\eta|} 
\end{aligned}
\]
leading to
\[
\begin{aligned}
	|\Delta_{\xi}\mathfrak{M} |
	\lesssim 
	\frac{|\xi|^{(2-\sigma)/2}|\eta|^{1-\sigma}}{\lal \eta\ral^{4\lambda}} \lt( 1 + \frac{1}{|\xi| } +\frac{1}{|\xi|^{2}} \rt)  
	\lesssim \frac{|\xi|^{(2-\sigma)/2}|\eta|^{1-\sigma}}{\lal \eta \ral^{4\lambda}}
\end{aligned}
\]
due to $ |\xi| \ge 1$. Then, it is easy to derive
\[
\begin{aligned}
	\int_{\R^3} |\Delta_{\xi}  \mathfrak{M} \varphi(\frac{\xi}{N}) |^2 \, d\xi 
	\lesssim 
	\int_{|\xi|\sim N} \frac{|\xi|^{2-\sigma}}{\lal \eta \ral^{8\lambda-2+2\sigma}} \,   d\xi 
	\lesssim  \frac{N^{5-\sigma}}{\lal N \ral^{8\lambda-2+2\sigma}} .
\end{aligned}
\]
By using this along with \eqref{A2_m}, we deduce that
\[
\begin{aligned}
	\|\mathfrak{M} \varphi_N \|_{\dot{H}_{\xi}^{k}} 
	\lesssim \frac{N^{(5-\sigma)/2}}{\lal N \ral^{4\lambda-1+\sigma} } 
\end{aligned}
\]
for $0\leq k\leq 2$ when $|\xi|\ge 1$. Thus, we have
\[
\begin{aligned}
	\sum_{N\ge1} \| \mathfrak{M} \varphi_N  \|_{L_{\eta}^{\infty}\dot{H}_{\xi}^k} 
	\lesssim  \sum_{N\ge1} \frac{1}{ \lal N\ral^{4\lambda-7/2+3\sigma/2}} <\infty
\end{aligned}
\]
for $\lambda>7/8-3\sigma/8$ and $0\leq k\leq 2$.
Consequently, by \eqref{A2_1} and the above inequality, we conclude that $\mathfrak{M} \in M_{\eta, \xi}^{k, \infty}$ for $\lambda>7/8-3\sigma/8$ for $0\leq k\leq 3/2$.

Next, we show $\mathfrak{M} \in M_{\xi, \eta}^{k, \infty}$ by dividing it into the following cases:
\begin{itemize}
	\item (A$'$1) $\max\{|\xi-\eta|, |\eta|\} < 1$, \\[-4mm]
	\item (A$'$2) $\min\{|\xi-\eta|, |\eta|\} \ge 1$, \\[-4mm]
	\item (A$'$3) $\max\{|\xi-\eta|, |\eta| \}\ge 1$ and $\min\{|\xi-\eta|, |\eta|\} < 1$.\\[-4mm]
\end{itemize}

In the case (A$'$1), we observe
\[
|\Phi| \gtrsim |\eta|^{\frac{2-\sigma}{2}}, \quad  |\nabla_{\eta}\Phi| \lesssim \frac{1}{|\eta|^{\sigma/2}} , \quad \mbox{and} \quad |\Delta_{\eta} \Phi| \lesssim \frac{1}{|\eta|^{1+\sigma/2}}
\]
leading to
\[
|\nabla_{\eta} (\frac{1}{\Phi})| \lesssim \frac{1}{|\eta|^{2-\sigma/2}} \quad \mbox{and} \quad |\Delta_{\eta}(\frac{1}{\Phi})| \lesssim \frac{1}{|\eta|^{3-\sigma/2}}.
\]
It is easy to verify that
\[
\begin{aligned}
	|\mathfrak{M}|&\lesssim \frac{|\xi|^{(2-\sigma)/2}|\eta|^{(2-\sigma)/2}}{\lal \xi\ral^{2\lambda} \lal \eta \ral^{2\lambda} } \quad \mbox{and} \quad 
	|\Delta_{\eta}\mathfrak{M}|  \lesssim \frac{|\xi|^{(2-\sigma)/2}}{ |\eta|^{1+\sigma/2}\lal \xi\ral^{2\lambda} \lal \eta\ral^{2\lambda} }
\end{aligned}
\]
which subsequently implies
\[
\begin{aligned}
	\int_{\R^3} |\mathfrak{M} \varphi ( \frac{\eta}{N}) |^2 \, d\eta 
	&\lesssim \int_{\rho\sim N}  \frac{\rho^{4-\sigma} |\xi|^{2-\sigma}}{\lal \rho \ral^{4\lambda}\lal \xi \ral^{4\lambda} }  \, d\rho 
	\lesssim     \frac{N^{5-\sigma} }{\lal N \ral^{4\lambda}\lal \xi\ral^{4\lambda-2+\sigma} }, \\
	\int_{\R^3} |\Delta_{\eta}  \mathfrak{M} \varphi(\frac{\eta}{N}) |^2 \, d\eta 
	& \lesssim \int_{\rho \sim N} \frac{|\xi|^{2-\sigma} }{\rho^{\sigma} \lal \rho\ral^{4\lambda}\lal \xi \ral^{4\lambda} } \, d\rho 
	\lesssim  \frac{N^{1-\sigma}}{\lal N \ral^{4\lambda}\lal \xi\ral^{4\lambda-2+\sigma} }.
\end{aligned}
\]
By using an interpolation argument, we obtain 
\[
\begin{aligned}
	\|\mathfrak{M} \varphi_N \|_{\dot{H}_{\eta}^{k}} 
	 \lesssim   \lt(  \frac{N^{(5-\sigma)/2} }{\lal N \ral^{2\lambda}\lal \xi\ral^{2\lambda-1+\sigma/2} } \rt)^{1-k/2}  \lt(  \frac{N^{(1-\sigma)/2}}{\lal N \ral^{2\lambda}\lal \xi\ral^{2\lambda-1+\sigma/2} } \rt)^{k/2} = \frac{N^{(5-\sigma)/2-k}}{\lal N \ral^{2\lambda}\lal \xi\ral^{2\lambda-1+\sigma/2} }.
\end{aligned}
\]
For $\lambda> (5-\sigma)/4-k/2$ for $0\leq k\leq 3/2$, it is summable over $N$, establishing $\mathfrak{M}\in M_{\xi, \eta}^{k, \infty}$.

In the case (A$'$2), we get
\[
|\Phi| \gtrsim |\eta|, \quad  |\nabla_{\eta}\Phi| \lesssim |\xi|+ \sin \frac{\beta}{2} \lesssim |\xi| ,\quad \mbox{and} \quad |\Delta_{\eta} \Phi| \lesssim |\xi|,
\]
due to \eqref{sin_beta} for $|\eta|\ge1$. Subsequently, it follows that
\[
|\nabla_{\eta} (\frac{1}{\Phi})| \lesssim \frac{|\xi|}{|\eta|^{2}} \quad \mbox{and} \quad |\Delta_{\eta}(\frac{1}{\Phi})| \lesssim \frac{|\xi|}{|\eta|^2}
\]
owing to $|\xi| \leq |\eta|$. Now, we use $|\eta| \ge 1$ and $|\xi|\leq |\eta|$ to derive
\[
\begin{aligned}
	|\mathfrak{M}| \lesssim \frac{|\eta|^{1-\sigma}|\xi|^{(2-\sigma)/2}}{\lal \eta \ral^{2\lambda} \lal \xi\ral^{2\lambda}} \quad \mbox{and} \quad
	|\Delta_{\eta}\mathfrak{M}|  \lesssim \frac{|\eta|^{1-\sigma}|\xi|^{(2-\sigma)/2}}{\lal \eta \ral^{2\lambda} \lal \xi \ral^{2\lambda}} \lt( \frac{|\xi|}{|\eta|}+\frac{|\xi|}{|\eta|^2}+\frac{1}{|\eta|^2}\rt) \lesssim \frac{|\eta|^{1-\sigma}|\xi|^{(2-\sigma)/2}}{\lal \eta \ral^{2\lambda} \lal \xi \ral^{2\lambda}}.
\end{aligned}
\]
This leads to
\[
\begin{aligned}
	\int_{\R^3} |\mathfrak{M} \varphi ( \frac{\eta}{N}) |^2\,d\eta 
	&\lesssim \int_{\rho\sim N}  \frac{\rho^{4-2\sigma} |\xi|^{2-\sigma}}{\lal \rho \ral^{4\lambda}\lal \xi \ral^{4\lambda} }\, d\rho 
	\lesssim     \frac{N^{5-2\sigma} |\xi|^{2-\sigma}}{\lal N \ral^{4\lambda}\lal \xi\ral^{4\lambda} }, \\
	\int_{\R^3} |\Delta_{\eta}  \mathfrak{M} \varphi(\frac{\eta}{N}) |^2 \, d\eta 
	& \lesssim \int_{\rho \sim N} \frac{\rho^{4-2\sigma}|\xi|^{2-\sigma} }{ \lal \rho\ral^{4\lambda}\lal \xi \ral^{4\lambda} } \, d\rho 
	\lesssim  \frac{N^{5-2\sigma}|\xi|^{2-\sigma}}{\lal N \ral^{4\lambda}\lal \xi\ral^{4\lambda} },
\end{aligned}
\]
and consequently, we arrive at 
\[
\begin{aligned}
	\|\mathfrak{M} \varphi_N \|_{\dot{H}_{\eta}^{k}} 
	&\lesssim   \frac{N^{5/2-\sigma}|\xi|^{1-\sigma/2}}{\lal N \ral^{2\lambda}\lal \xi\ral^{2\lambda} }\lesssim  \frac{1}{\lal N \ral^{2\lambda-5/2+\sigma}\lal \xi\ral^{2\lambda-1+\sigma/2} }
\end{aligned}
\]
for $0\leq k\leq 2$.
It is summable over $N$ whenever $\lambda> 5/4-\sigma/2$. Hence, we have $\mathfrak{M}\in M_{\xi, \eta}^{k, \infty}$. 

We now consider the case (A$'$3). In this case, we get
\[
|\Phi| \gtrsim |\eta|, \quad  |\nabla_{\eta}\Phi| \lesssim \frac{1}{|\eta|^{\sigma/2}},\quad \mbox{and} \quad |\Delta_{\eta} \Phi| \lesssim \frac{1}{|\eta|^{1+\sigma/2}},
\]
which leads to
\[
|\nabla_{\eta} (\frac{1}{\Phi})| \lesssim \frac{1}{|\eta|^{2+\sigma/2}} ,\quad |\Delta_{\eta}(\frac{1}{\Phi})| \lesssim \frac{1}{|\eta|^{3+\sigma/2}}
\]
due to $|\eta| \gtrsim 1$.
Moreover, we estimate
\[
\begin{aligned}
	|\mathfrak{M}| \lesssim \frac{|\eta|^{1-\sigma}|\xi|^{(2-\sigma)/2}}{ \lal \eta \ral^{2\lambda}\lal \xi\ral^{2\lambda}} \quad \mbox{and} \quad
	|\Delta_{\eta}\mathfrak{M}|  \lesssim \frac{|\xi|^{(2-\sigma)/2}}{\lal \eta \ral^{2\lambda} \lal \xi \ral^{2\lambda}}  \lt( \frac{1}{|\eta|^{1+3\sigma/2}}+\frac{1}{|\eta|^{1+\sigma}}\rt) \lesssim  \frac{|\xi|^{(2-\sigma)/2}}{\lal \eta \ral^{2\lambda} \lal \xi \ral^{2\lambda}}. 
\end{aligned}
\]
Consequently,
\[
\begin{aligned}
	\int_{\R^3} |\mathfrak{M} \varphi ( \frac{\eta}{N}) |^2 \,d\eta 
	&\lesssim \int_{\rho\sim N}  \frac{\rho^{4-2\sigma} |\xi|^{2-\sigma}}{\lal \rho \ral^{4\lambda}\lal \xi \ral^{4\lambda} } \,d\rho 
	\lesssim     \frac{N^{5-2\sigma}|\xi|^{2-\sigma}}{\lal N \ral^{4\lambda}\lal \xi\ral^{4\lambda} }, \\
	\int_{\R^3} |\Delta_{\eta}  \mathfrak{M} \varphi(\frac{\eta}{N}) |^2 \, d\eta 
	& \lesssim \int_{\rho \sim N} \frac{ \rho^{2}|\xi|^{2-\sigma} }{ \lal \rho\ral^{4\lambda}\lal \xi \ral^{4\lambda} }\,d\rho 
	\lesssim  \frac{N^3|\xi|^{2-\sigma}}{\lal N \ral^{4\lambda}\lal \xi\ral^{4\lambda} },
\end{aligned}
\]
and we have
\[
\begin{aligned}
\|\mathfrak{M} \varphi_N \|_{\dot{H}_{\eta}^{k}} 
&\lesssim   \lt(  \frac{N^{(5-2\sigma)/2} |\xi|^{(2-\sigma)/2}}{\lal N \ral^{2\lambda}\lal \xi\ral^{2\lambda} } \rt)^{1-k/2}  \lt(  \frac{N^{3/2} |\xi|^{(2-\sigma)/2}}{\lal N \ral^{2\lambda}\lal \xi\ral^{2\lambda} } \rt)^{k/2} \\
&\lesssim   \frac{N^{5/2-\sigma+(\sigma-1)k/2} |\xi|^{(2-\sigma)/2}}{\lal N \ral^{2\lambda}\lal \xi\ral^{2\lambda} }\cr
&\lesssim  \frac{1}{\lal N \ral^{2\lambda-5/2+\sigma-(\sigma-1)k/2} \lal \xi\ral^{2\lambda-(2-\sigma)/2} }
\end{aligned}
\]
which obviously implies $\mathfrak{M}\in M_{\xi, \eta}^{k, \infty} $ for $0\leq k\leq 2$ if $\lambda>5/4-\sigma/2+(\sigma-1)k/4$.

%
%
%
%
%
%
%
%
%
\subsection{Case C}
We now consider the case $|\eta|/3 <|\xi|<3|\eta|$, implying that $|\xi|\sim |\eta|$. To make use of Lemmas \ref{Lemma_low_Phi} and  \ref{Lemma_deriv_Phi}, we further subdivide (Case C) into the subsequent cases:
\begin{itemize}
	\item (C1)  $|\xi-\eta|$ is the smallest and $|\xi-\eta| <1$,\\[-4mm]
	\item (C2)  $|\xi-\eta|$ is the smallest and $|\xi-\eta|\ge 1$,\\[-4mm]
	\item (C3)  $|\xi|$ is the smallest and $\max\{|\xi-\eta|, |\eta|\} < 1$, \\[-4mm]
	\item (C4)  $|\xi|$ is the smallest and $\max\{|\xi-\eta|, |\eta|\} \ge 1$, \\[-4mm]
	\item (C5)  $|\eta|$ is the smallest and $\min\{|\xi-\eta|, |\xi|\}<1$, \\[-4mm]
	\item (C6)  $|\eta|$ is the smallest and $\min\{|\xi-\eta|, |\xi|\}\ge 1$. \\[-4mm]
\end{itemize}
We first consider the case (C1). We observe
\[
	|\Phi|\gtrsim \frac{|\xi-\eta|^{(2-\sigma)/2}}{\lal \eta\ral^{\sigma/2}} , \quad 
	|\nabla_{\xi}\Phi| 
	\lesssim \frac{1}{|\xi-\eta|^{\sigma/2}} + \frac{1}{|\xi-\eta|^{\sigma/2}} \sin \frac{\gamma}{2}\lesssim   \frac{1}{|\xi-\eta|^{\sigma/2}},
	\]
	and
	\[
	|\Delta_{\xi}\Phi| \lesssim \frac{1}{|\xi-\eta|^{1+\sigma/2}}
\]
by using \eqref{low_Phi_2} and Lemma \ref{Lemma_deriv_Phi}.
Consequently,
\[
\begin{aligned}
	|\nabla_{\xi} (\frac{1}{\Phi})| 
	\lesssim \frac{\lal \eta \ral^{\sigma}}{ |\xi-\eta|^{2-\sigma/2} }  
	\quad \text{and} \quad
	|\Delta_{\xi}(\frac{1}{\Phi})| 
	\lesssim \frac{\lal \eta \ral^{3\sigma/2}}{|\xi-\eta|^{3-\sigma/2}}
\end{aligned}
\]
followed by
\[
\begin{aligned}
	|\mathfrak{M}|
	 \lesssim \frac{|\eta|^{2-\sigma}}{ \lal \xi-\eta \ral^{2\lambda}\lal \eta \ral^{2\lambda-\sigma/2} } \quad \text{and} \quad
	|\Delta_{\xi}\mathfrak{M} |
	  \lesssim \frac{|\eta|^{2-\sigma}}{ |\xi-\eta|^{2}\lal \xi-\eta \ral^{2\lambda} \lal \eta \ral^{2\lambda-3\sigma/2} }.
\end{aligned}
\]
We then deduce that
\[
\begin{aligned}
	\int_{\R^3} |\mathfrak{M} \varphi ( \frac{\xi-\eta}{N}) |^2  \,d\xi 
	&\lesssim \int_{\rho\sim N}  \frac{\rho^2|\eta|^{4-2\sigma}}{\lal \rho \ral^{4\lambda} \lal \eta\ral^{4\lambda-\sigma}}  \, d\rho
	\lesssim    \frac{N^3}{\lal N \ral^{4\lambda} \lal \eta\ral^{4\lambda-4+\sigma}}, \\
	\int_{\R^3} |\Delta_{\xi}  \mathfrak{M} \varphi(\frac{\xi-\eta}{N}) |^2  \, d\xi 
	&\lesssim \int_{\rho \sim N} \frac{ |\eta|^{4-2\sigma}}{ \rho^2 \lal \rho \ral^{4\lambda} \lal \eta \ral^{4\lambda-3\sigma} }  \, d\rho
	\lesssim  \frac{1}{N\lal N\ral^{4\lambda} \lal \eta\ral^{4\lambda-4-\sigma} }.
\end{aligned}
\]
By using an interpolation argument, we obtain
\[
\begin{aligned}
	\|\mathfrak{M} \varphi_N \|_{\dot{H}_{\xi}^{k}} 
	&\lesssim   \lt( \frac{N^{3/2}}{\lal N \ral^{2\lambda} \lal \eta\ral^{2\lambda-2+\sigma/2}} \rt)^{1-k/2}  \lt( \frac{1}{N^{1/2}\lal N\ral^{2\lambda} \lal \eta\ral^{2\lambda-2-\sigma/2} } \rt)^{k/2} \\
	&=
	\frac{N^{3/2-k}}{\lal N \ral^{2\lambda} \lal \eta\ral^{2\lambda -2 -\sigma(k-1)/2 } }     
	\lesssim
	\frac{1}{\lal N\ral^{2\lambda -3/2+k}}
\end{aligned}
\]
for $0\leq k\leq 3/2$ if $\lambda> 1+\sigma(k-1)/4$.
It is summable over $N$ when $\lambda>3/4-k/2$, thereby establishing $\mathfrak{M} \in M_{\eta, \xi}^{k, \infty}$ for $0\leq k\leq 3/2$.

We next consider the case (C2). Applying \eqref{low_Phi_2} gives
\[
\begin{aligned}
	|\Phi| 
	&\gtrsim|\xi-\eta|\lt( \frac{1}{1+|\xi|^{\sigma}|\xi-\eta|^{\sigma}}+\gamma^2 +\theta^2 \rt) \gtrsim |\xi-\eta| \lt(\frac{1}{1+|\xi|^{2\sigma} +|\eta|^{2\sigma}}+\beta^2 \rt) =|\xi-\eta| (d^2+\beta^2 ) 
\end{aligned}
\]
with $d^2 =\frac{1}{1+|\xi|^{2\sigma} +|\eta|^{2\sigma}}$.
Since $\min\{|\xi-\eta|, |\xi|\}=|\xi-\eta|\ge1$, we get 
\[
|\nabla_{\xi} \Phi| \lesssim |\eta|+\sin \frac{\gamma}{2}\lesssim |\eta| \quad \text{and} \quad |\Delta_{\xi}\Phi| \lesssim |\eta|
\]
by Lemma \ref{Lemma_deriv_Phi} together with \eqref{sin_gamma} and $|\xi|\ge 1$. Then, it follows that 
\[
\begin{aligned}
	|\nabla_{\xi} (\frac{1}{\Phi})| 
	&\lesssim \frac{|\eta|}{ |\xi-\eta|^2 (d^2 +\beta^2 )^2},\\
	|\Delta_{\xi}(\frac{1}{\Phi})| 
	&\lesssim \frac{|\eta|}{|\xi-\eta|^2(d^2+\beta^2)^2} 
	+\frac{|\eta|^2}{|\xi-\eta|^{3}(d^2  +\beta^2)^3} 
	\lesssim \frac{|\eta|^2}{|\xi-\eta|^{3}(d^2  +\beta^2)^3}
\end{aligned}
\]
owing to $|\xi-\eta| \leq |\eta|$ and $d^2 +\beta^2 \lesssim 1$. This leads to
\[
\begin{aligned}
	|\mathfrak{M} | 
	& \lesssim \frac{|\eta|^{2-\sigma}}{|\xi-\eta|^{\sigma/2}  \lal \xi-\eta \ral^{2\lambda}\lal \eta \ral^{2\lambda} (d^2+\beta^2)}   \lesssim \frac{|\eta|^{2-\sigma}}{  \lal \xi-\eta \ral^{2\lambda}\lal \eta \ral^{2\lambda} (d^2+\beta^2)},   \\
	|\Delta_{\xi}\mathfrak{M} |
	& \lesssim \frac{|\eta|^{2-\sigma}}{|\xi-\eta|^{2+\sigma/2}\lal \xi-\eta \ral^{2\lambda}\lal \eta \ral^{2\lambda} }\cdot \Big(   \frac{|\eta|^2}{(d^2+\beta^2)^3}+\frac{|\eta|}{(d^2+\beta^2)^2}+ \frac{1}{d^2+\beta^2}  \Big)  \lesssim \frac{|\eta|^{4-\sigma}}{\lal \xi-\eta \ral^{2\lambda}\lal \eta \ral^{2\lambda} (d^2+\beta^2)^3} .
\end{aligned}
\]
since $|\eta| \ge |\xi-\eta | \ge 1$. By employing spherical coordinates with the north pole $-\frac{\eta}{|\eta|}$, we derive
\[
\begin{aligned}
	\int_{\R^3} |\mathfrak{M} \varphi ( \frac{\xi-\eta}{N}) |^2 \, d\xi 
	&\lesssim \int_{\rho\sim N}  \frac{\rho^2|\eta|^{4-2\sigma}}{\lal \rho \ral^{4\lambda}\lal \eta \ral^{4\lambda} } \lt( \int_{d^2 \leq \beta^2} \frac{\sin \beta}{\beta^4 } \, d\beta +\int_{\beta^2 <d^2} \frac{\sin \beta}{d^4} \,d\beta\rt)  d\rho  \lesssim     \frac{N^3 |\eta|^{4-2\sigma}}{\lal N \ral^{4\lambda} \lal\eta\ral^{4\lambda-2\sigma} }
	\end{aligned}
\]
and
\[
\begin{aligned}
	\int_{\R^3} |\Delta_{\xi}  \mathfrak{M} \varphi(\frac{\xi-\eta}{N}) |^2 \, d\xi 
	&\lesssim \int_{\rho \sim N} \frac{\rho^2 |\eta|^{8-2\sigma}}{ \lal \rho\ral^{4\lambda}\lal \eta \ral^{4\lambda}}\lt( \int_{d^2 \leq \beta^2} \frac{\sin \beta}{\beta^{12}} \, d\beta +\int_{\beta^2 <d^2} \frac{\sin \beta}{d^{12}} \, d\beta \rt)d\rho  \lesssim   \frac{ N^3|\eta|^{8-2\sigma} }{ \lal N \ral^{4\lambda}  \lal \eta \ral^{4\lambda-10\sigma} }
\end{aligned}
\]
by noting that $d^2 \sim 1/\lal \eta\ral^{2\sigma}$.  Hence, we obtain
\[
\begin{aligned}
	\|\mathfrak{M} \varphi_N \|_{\dot{H}_{\xi}^{k}} 
	&\lesssim   \lt( \frac{N^{3/2} |\eta|^{2-\sigma} }{\lal N \ral^{2\lambda} \lal\eta\ral^{2\lambda-\sigma} }\rt)^{1-k/2}  \lt(  \frac{N^{3/2} |\eta|^{4-\sigma} }{  \lal N \ral^{2\lambda} \lal \eta \ral^{2\lambda-5\sigma}} \rt)^{k/2}  =
	\frac{N^{3/2}|\eta|^{2-\sigma+k}}{\lal N\ral^{2\lambda} \lal \eta\ral^{2\lambda-\sigma-2\sigma k} }     
	\lesssim
	\frac{1}{\lal N\ral^{2\lambda -3/2}}
\end{aligned}
\]
if $\lambda >1+\sigma k +k/2$ for $0\leq k\leq 2$. It is summable over $N$, leading to the desired conclusion.

Cases (C3)-(C6) can be addressed by analogously applying the methodology used for Cases A and B, owing to the significant observation that $|\xi|\sim |\xi-\eta|\sim |\eta| $. Specifically, in case (C3), we establish
\[
\begin{aligned}
	|\Phi| \gtrsim|\eta|^{(2-\sigma)/2}, \quad	|\nabla_{\xi}\Phi| 
	\lesssim \frac{1}{|\xi-\eta|^{\sigma/2}} \sim \frac{1}{|\xi|^{\sigma/2}}, \quad \mbox{and} \quad |\Delta_{\xi}\Phi| \lesssim \frac{1}{|\xi-\eta|^{1+\sigma/2}}\sim \frac{1}{|\xi|^{1+\sigma/2}}.
\end{aligned}
\]
By following the procedure employed for case (A1), we arrive at a result akin to (A1). Similarly, cases (C4), (C5), and (C6) can be deduced by a similar approach used for cases (A2), (B1), and (B2) respectively. 

For $\mathfrak{M}\in M_{\xi, \eta}^{k, \infty}$, we take into account the following cases; 
\begin{itemize}
	\item (C$'$1) $|\xi-\eta|$ is the smallest and $|\xi-\eta|<1 $, \\[-4mm]
	\item (C$'$2) $|\xi-\eta|$ is the smallest and $|\xi-\eta|\ge 1$,\\[-4mm]
	\item (C$'$3) $|\xi|$ is the smallest, \\[-4mm]	
	\item (C$'$4) $|\eta|$ is the smallest.\\[-4mm]
\end{itemize}
Since $|\xi|\sim |\eta|$, it is easy to see that the cases (C$'$1) and (C$'$2) can be handled in a manner akin to that of cases (C1) and (C2) respectively. As previously indicated, cases (C$'$3) and (C$'$4) are addressed in a similar manner to that employed for cases A and B, benefiting from the observation that $|\xi|\sim |\xi-\eta|\sim |\eta|$.

%
%
%
%
%
%
%
%
%
\section{Asymptotic behavior of $|p''(r)|$}\label{appendix_3}
In this part, we investigate the asymptotic behavior of $|p''(r)|$. Note that 
\[
p''(r)=\frac{2\sigma(\sigma-1)r^{\sigma}-\sigma(2-\sigma)}{4 r^{1+\sigma/2}(r^{\sigma}+1)^{3/2}},
\]
for $0<\sigma<2$. For this, we mainly separate cases $0 < \sigma \leq 1$ and $1 < \sigma < 2$.

When $0<\sigma<1$, we note that $p''(r)$ is negative, and thus
\[
|p''(r)|=\frac{2\sigma(1-\sigma)r^{\sigma}+\sigma(2-\sigma)}{4r^{1+\sigma/2}(r^{\sigma}+1)^{3/2}}  \sim 
\begin{cases}
	r^{-1-\sigma} \qquad \ \ \text{if }  r\ge 1\\
	r^{-1-\sigma/2}
	\qquad \text{if } r<1\\
\end{cases}
\]
due to $1+r^{\sigma} \sim r^{\sigma}$ when $r\ge 1$ and $1+r^{\sigma} \sim 1$ when $r<1$.
For the case $\sigma=1$, the expression simplifies to
\[
|p''(r)|=\frac{1}{4 r^{3/2}(r+1)^{3/2}} \sim 
\begin{cases}
	r^{-3} \qquad \ \ \text{if }  r\ge 1\\
	r^{-3/2}
	\qquad \text{if } r<1.\\
\end{cases}
\]

In contrast to the case $0<\sigma\leq 1$, the case $1<\sigma<2$ introduce a degenerate point at
\[
r_0 =\lt( \frac{2-\sigma}{2(\sigma-1)}\rt)^{1/\sigma},
\]
where $p''(r_0)=0$. Then, it can be verified that $p''(r)>0$ (respectively $p''(r)<0$) when $r>r_0$ (respectively $r<r_0$). 
When $r<r_0$, we let $r^{\sigma}=r_0^{\sigma}-c^{\ast}$ for some $c^{\ast}>0$. Then, we get
\[
|p''(r)|=\frac{\sigma(2-\sigma)-2\sigma(\sigma-1)r^{\sigma}}{4 r^{1+\sigma/2}(r^{\sigma}+1)^{3/2}} = 
\frac{2\sigma(\sigma-1)c^{\ast}}{4 r^{1+\sigma/2}(r^{\sigma}+1)^{3/2}} \sim
	\frac{1}{r^{1+2\sigma}} , \qquad 1\leq r <r_0
\]
under the condition $1<\sigma\leq 4/3$ (i.e., $r_0\ge1$). Additionally, it follows that
\[
|p''(r)|\sim
\frac{1}{r^{1+\sigma/2}} , \qquad r<1
\]
whenever $1<\sigma<2$.

Similarly for $r>r_0$, letting $r^{\sigma}=r_0^{\sigma}+c^{\ast}$ with some $c^{\ast}>0$ yields
\[
|p''(r)|= \frac{2\sigma(\sigma-1)r^{\sigma}-\sigma(2-\sigma)}{4 r^{1+\sigma/2}(r^{\sigma}+1)^{3/2}}\sim
\frac{1}{r^{1+\sigma/2}} , \qquad r_0<r<1
\]
for $4/3<\sigma<2$ which means $r_0 <1$. Moreover, we obtain
\bq\label{note_7}
|p''(r)| = \frac{2\sigma(\sigma-1)r^{\sigma}-\sigma(2-\sigma)}{4 r^{1+\sigma/2}(r^{\sigma}+1)^{3/2}}
\sim \frac{1}{r^{1+2\sigma}} , \qquad r_0 \leq 1\leq r
\eq
for $1<\sigma\leq 4/3$.
The argument used for \eqref{note_7} also works for $\sigma>4/3$ though, the following approach allows us to require slightly less regularity when $r>1$. 
Indeed, when $\sigma>4/3$ the condition $r\ge 1$ implies  
\[
|p''(r)|=\frac{2\sigma(\sigma-1)r^{\sigma}-\sigma(2-\sigma)}{4 r^{1+\sigma/2}(r^{\sigma}+1)^{3/2}} \gtrsim  
	\frac{\sigma r^{\sigma} \{ 2(\sigma-1)-(2-\sigma)\}}{4 r^{1+2\sigma}} 
	\gtrsim \frac{1}{r^{1+\sigma}}
	\]
	and
	\[
	|p''(r)|\sim \frac{2\sigma(\sigma-1)r^{\sigma}-\sigma(2-\sigma)}{4 r^{1+2\sigma}} \lesssim \frac{1}{r^{1+\sigma}}.
\]
Hence, we have
\[
|p''(r)|\sim \frac{1}{r^{1+\sigma}}, \qquad r\ge1
\]
for $4/3<\sigma<2$.

%
%
%
%
%
%
%
%
%
\section{Proof of Lemma \ref{GNT_estimate_product}}\label{proof_Lemma_product}
Here, we provide detailed proof for Lemma \ref{GNT_estimate_product}. We only establish \eqref{bilinear_11} since the other estimate can be derived through a similar argument. We write $B=B_{\mathfrak{M}}$ for simplicity.

By using the Littlewood--Paley theorem, we get
	\bq\label{note_3}
		\|\lal \nabla\ral^{\alpha_1} D^{\alpha_2} B[f, g]\|_{L^{l_1'}}^2 \lesssim \lt\|P_{\leq 1} B[f, g] \rt\|_{L^{l_1'}}^2	+\sum_{N>1}N^{2(\alpha_1+\alpha_2)}\lt\|P_{N} B[f, g]\rt\|_{L^{l_1'}}^2,
	\eq
	where $D^{\alpha_2}=\lal |\nabla|^{\alpha_2} \ral$ or $D^{\alpha_2}=|\nabla|^{\alpha_2}$.
	The first term can be estimated as follows:
\[
\lt\|P_{\leq 1} B[f, g] \rt\|_{L^{l_1'}} \lesssim \|f \|_{L^{l_2}} \|g\|_{L^2} .
\]
For the second term, we observe
	\[
	\begin{aligned}
		\lt\|P_{N}B[f, g] \rt\|_{L^{l_1'}}
		\lesssim  \|P_{N} B[P_{\leq \frac{N}{8}}f, g] \|_{L^{l_1'}}
		+\sum_{M>\frac{N}{8}}   \|P_{N} B[P_{M}f, g] \|_{L^{l_1'}}.
	\end{aligned}
	\]
By using \eqref{bilinear_1}, we deduce
	\[
	\begin{aligned}
		\|P_{N} B[P_{\leq \frac{N}{8}} f, g] \|_{L^{l_1'}} 
		 \lesssim  \| B[P_{\leq \frac{N}{8}}f, P_{\frac{N}{8} <\cdot<8N}g]  \|_{L^{l_1'}} \lesssim \|P_{\leq \frac{N}{8}}f \|_{L^{l_2}} \| P_{\frac{N}{8} <\cdot<8N}g\|_{L^{2}} \lesssim \|f \|_{L^{l_2}} \sum_{M\sim N}\| P_{M}g\|_{L^{2}}.
	\end{aligned}
	\]
Consequently, we arrive at
	\[
	\begin{aligned}
	\sum_{N>1}N^{2(\alpha_1+\alpha_2)}  \|P_{N} B[P_{\leq \frac{N}{8}} f, g] \|_{L^{l_1'}}^2
	&\lesssim \|f \|_{L^{l_2}}^2	\sum_{N>1}\lt(\sum_{M\sim N}N^{\alpha_1+\alpha_2}\| P_{M}g\|_{L^{2}} \rt)^2 \\
	&\lesssim \|f \|_{L^{l_2}}^2	\sum_{M\gtrsim 1}M^{2(\alpha_1+\alpha_2)}\| P_{M}g\|_{L^{2}}^2\\
	&\lesssim \|f \|_{L^{l_2}}^2\| g\|_{\dot{H}^{\alpha_1+\alpha_2}}^2.
	\end{aligned}
	\]
	Meanwhile, we find
	\[
	\begin{aligned}
		\sum_{M>\frac{N}{8}}   \|P_{N} B[P_{M}f, g] \|_{L^{l_1'}}
		&\lesssim  \sum_{M>\frac{N}{8}} \|P_{M}f \|_{L^2} \|g\|_{L^{l_2}} \lesssim  \sum_{M>\frac{N}{8}} \frac{1}{M^{\alpha_1+\alpha_2}}\|P_{M}f \|_{\dot{H}^{\alpha_1+\alpha_2}} \|g\|_{L^{l_2}}
	\end{aligned}
	\]
	due to \eqref{bilinear_1}.
Then, it follows that
	\[
	\begin{aligned}
	&\sum_{N>1}N^{2(\alpha_1+\alpha_2)}  \lt( \sum_{M>\frac{N}{8}}   \|P_{N} B[P_{M}f, g] \|_{L^{l_1'}}\rt)^2  \\
		&\quad \lesssim \sum_{N>1}  \lt( \sum_{M>\frac{N}{8}}\lt( \frac{N}{M}\rt)^{\alpha_1+\alpha_2}\|P_{M}f \|_{\dot{H}^{\alpha_1+\alpha_2}}\|g\|_{L^{l_2}} \rt)^2 \\
		&\quad \lesssim\|g\|_{L^{l_2}} ^2
		\sum_{N>1} \lt(  \sum_{M>\frac{N}{8}} \lt( \frac{N}{M}\rt)^{\alpha_1+\alpha_2}\sum_{M>\frac{N}{8}}\lt( \frac{N}{M}\rt)^{\alpha_1+\alpha_2}\|P_{M}f \|_{\dot{H}^{\alpha_1+\alpha_2}}^2\rt)\\
		&\quad \lesssim\|g\|_{L^{l_2}} ^2
		\sum_{M>\frac{1}{8}} \sum_{1<N<8M} \lt( \frac{N}{M}\rt)^{\alpha_1+\alpha_2}\|P_{M}f \|_{\dot{H}^{\alpha_1+\alpha_2}}^2\\
		&\quad \lesssim\|g\|_{L^{l_2}} ^2
		\sum_{M>\frac{1}{8}}\|P_{M}f \|_{\dot{H}^{\alpha_1+\alpha_2}}^2
	\end{aligned}
	\]
	by the Cauchy--Schwartz inequality and the fact that
	\[
	\sum_{N<8M} \lt( \frac{N}{M} \rt)^{\alpha_1+\alpha_2} \lesssim 1.
	\]
	Therefore, we obtain
	\[
	\begin{aligned}
		\sum_{N>1}N^{2(\alpha_1+\alpha_2)}\lt\|P_{N} B[f, g]\rt\|_{L^{l_1'}}^2 
		&\lesssim \|f \|_{L^{l_2}} ^2\sum_{N\gtrsim 1}\| P_{N}g\|_{\dot{H}^{\alpha_1+\alpha_2}} ^2 +\|g\|_{L^{l_2}} ^2
		\sum_{M>\frac{1}{8}}\|P_{M}f \|_{\dot{H}^{\alpha_1+\alpha_2}}^2\\
		&\lesssim \|f \|_{L^{l_2}}^2 \|g\|_{\dot{H}^{\alpha_1+\alpha_2}}^2 +\|g\|_{L^{l_2}} ^2
		\|f \|_{\dot{H}^{\alpha_1+\alpha_2}}^2.
	\end{aligned}
	\]
	Hence combining the estimates for the first and the second terms of \eqref{note_3}, we conclude the desired result.
%
%
%
%
%
%
%
%
%

\end{document}